\documentclass[leqno,11pt,oneside]{amsart}
\usepackage{amsmath,amssymb,amsthm,mathrsfs}
\usepackage{epsfig,color}
\usepackage[latin1]{inputenc}
\usepackage[english]{babel}
\usepackage{mathtools}
\usepackage{braket}
\usepackage[toc]{appendix}
\usepackage{enumitem}
\usepackage{ulem}
\usepackage{graphicx}
\usepackage{xfrac, nicefrac}
\usepackage{csquotes}
\allowdisplaybreaks

\numberwithin{equation}{section}

\usepackage{esint}

\usepackage{hyperref}

\newcommand{\rn}{\mathbb{R}^n}

\def\A{\mathcal A}

\def\H{\mathcal H}
\def\K{\mathcal K}

\def\R{\mathbb R}
\def\N{\mathbb N}

\def\e{\varepsilon}
\def\s{\sigma}

\def\vphi{\varphi}

\def\om{\omega}
\def\l{\lambda}
\def\g{\gamma}
\def\k{\kappa}
\def\Om{\Omega}

\def\pa{\partial}

\renewcommand{\a}{\alpha}

\renewcommand{\d}{\mathrm{d}}

\renewcommand{\l}{\lambda}
\renewcommand{\L}{\Lambda}

\renewcommand{\t}{\tau}

\renewcommand{\om}{\omega}

\newcommand{\ov}{\overline}

\newcommand{\diver}{{\rm div }}

\newcommand{\mres}{\mathbin{\vrule height 1.6ex depth 0pt width
0.13ex\vrule height 0.13ex depth 0pt width 1.3ex}}

\def\ued{u_{\e,m}}

\def\B{\mathcal{B}}

\def\A{\mathcal{A}}
\def\d{\delta}

\newtheorem*{theorem*}{Theorem}
\newtheorem{theorem}{Theorem}[section]
\newtheorem{lemma}[theorem]{Lemma}
\newtheorem{proposition}[theorem]{Proposition}
\newtheorem*{proposition*}{Proposition}

\newtheorem{corollary}[theorem]{Corollary}
\newtheorem{remark}[theorem]{Remark}
\newtheorem*{remark*}{Remark}

\newtheorem{definition}[theorem]{Definition}

\newtheorem{lemmaalph}{Lemma}

\newcommand{\todo}[1]{\text{\colorbox{yellow}{#1}}}
\newcommand{\cambio}[1]{\text{\colorbox{red}{#1}}}

\title{Global second-order estimates in anisotropic elliptic problems}

\begin{document}

\begin{abstract}
 We  deal with boundary value problems for second-order nonlinear elliptic equations in divergence form, which emerge as Euler-Lagrange equations of integral functionals of the Calculus of Variations built upon possibly anisotropic norms of the gradient of trial functions. Integrands with non polynomial growth are included in our discussion. The $W^{1,2}$-regularity of the stress-field  associated with solutions, namely the nonlinear expression of the gradient subject to the  divergence operator, is established under the weakest possible assumption that the datum on the right-hand side of the equation is a merely   $L^2$-function. Global regularity estimates are offered in domains enjoying minimal assumptions on the boundary. They depend on the weak curvatures of the boundary via either their degree of integrability or an isocapacitary inequality. By contrast, none of these assumptions is needed in the case of convex domains. An explicit estimate for the constants appearing in the  relevant estimates is exhibited in terms of the Lipschitz characteristic of the domains, when their boundary is endowed with H\"older continuous curvatures.
\end{abstract}

\author{Carlo Alberto Antonini \textsuperscript{1}}
\address{\textsuperscript{1} Dipartimento di Matematica  ``Federigo Enriques'', Universit\`a di Milano, Via Cesare Saldini 50, 20133 Milano, Italy}
\email{carlo.antonini@unimi.it}
\urladdr{}

\author{Andrea Cianchi\textsuperscript{2}}
\address{\textsuperscript{2}Dipartimento di Matematica e Informatica ``Ulisse Dini'',
Universit\`a di Firenze,
Viale Morgagni 67/A, 50134
Firenze,
Italy}
\email{andrea.cianchi@unifi.it}
\urladdr{0000-0002-1198-8718}

\author{Giulio Ciraolo \textsuperscript{3}}
\address{\textsuperscript{3}Dipartimento di Matematica  ``Federigo Enriques'', Universit\`a di Milano, Via Cesare Saldini 50, 20133 Milano, Italy}
\email{giulio.ciraolo@unimi.it}
\urladdr{ }

\author{Alberto Farina \textsuperscript{4}}
\address{\textsuperscript{4}LAMFA, UMR CNRS 7352, Universit\'e de Picardie Jules Verne, 33, Rue St Leu, 80039 Amiens, France}
\email{alberto.farina@u-picardie.fr}
\urladdr{}

\author{Vladimir Maz'ya \textsuperscript{5}}
\address{\textsuperscript{5}  Department of Mathematics, Link\"oping University, SE-581
83 Link\"oping, Sweden
}
\email{e-mail: vladimir.mazya@liu.se }
\urladdr{}

\date{\today}

\subjclass[2020]{35J25, 35J60}
\keywords{Anisotropic elliptic equations, second-order derivatives, $p$-Laplacian, Orlicz-Laplacian
Dirichlet problems, Neumann problems,  local solutions, convex domains}

\maketitle

\section{Introduction}\label{intro1}

We are concerned with global regularity properties of solutions to boundary value problems for elliptic equations modeled upon (possibly) anisotropic norms of the gradient. The relevant equations arise as Euler-Lagrange equations of functionals of the form
 \begin{equation}\label{funct_H}
J_H(u)= \int_\Omega B(H(\nabla u)) \,dx - \int_\Omega f u  \, dx
\end{equation} 
where $\Omega$ 
is an open bounded set in $\rn$, with $n \geq 2$. Here:
\begin{itemize} 
\item{$B: [0, \infty) \to [0, \infty)$ is a convex function  such that  $B(0)=0$,}
\item{$H: \rn \to [0, \infty)$ is a norm,}
\item{ $f: \Omega \to \mathbb R$ is a prescribed function.}
\end{itemize}
Any   function $B$ as above is called a Young function, and  takes the form
\begin{equation}\label{B}
B(t) = \int_0^t b(s) \,ds \qquad \text{for $t\geq 0$,}
\end{equation}
for some non-decreasing function $b:  [0, \infty) \to [0, \infty)$. Thus, the Euler-Lagrange equation of the functional $J_H$ reads
\begin{equation} \label{eq_anisotr_intro}
    -\mathrm{div}\big(\A(\nabla u)\big)= f \quad\text{in } \Omega \,,
\end{equation} 
where the function $\A : \rn \to \rn$ is defined as 
\begin{equation} \label{def_A}
    \A(\xi)=\begin{cases} b\big(H(\xi)\big) \nabla_\xi H(\xi) &\quad \text{if $\xi \neq 0$}
\\ 0  &\quad \text{if $\xi = 0$.}
\end{cases}
\end{equation}
Here, the subscript $\xi$ stands for differentiation in the $\xi$ variable.

In the isotropic case when $H$ is the Euclidean norm,  namely
$$H(\xi)=|\xi| \qquad \text{for $\xi \in \rn$,}$$ 
equation \eqref{eq_anisotr_intro} reproduces the classical linear Poisson equation 
with the choice $B(t)= \frac 12 t^2$, the $p$-Laplace equation
if $B(t)= \frac 1p t^p$ for some $p>1$,  and the Orlicz-Laplace equation
for a general Young function $B$.

Anisotropic elliptic problems, where $H(\xi)\neq |\xi|$, have been analyzed in various respects, at least in the case when $B(t)= \frac 1p t^p$.   The literature on this topic includes, for instance, \cite{AFLT, acf, BiCi, CiSa1, CiFiRo,cfv, del, FeKa, FiMa}.

The present paper is focused on an $L^2$-second-order regularity theory of solutions to boundary value problems for equation \eqref{eq_anisotr_intro} under homogeneous Dirichlet or Neumann conditions. Namely, we deal with solutions to Dirichlet problems of the form
\begin{equation}\label{eq:dir2}
\begin{cases}
        -\mathrm{div}\big(\A(\nabla u) \big)=f &\quad\text{in } \Omega
        \\
        u=0 &\quad\text{on }\partial \Omega \,,
\end{cases}
\end{equation}
and co-normal Neumann problems of the form
\begin{equation}\label{eq:neu2}
\begin{cases}
        -\mathrm{div}\big(\A(\nabla u) \big)=f &\quad\text{in } \Omega
        \\
        \A(\nabla u) \cdot \nu =0 &\quad\text{on }\partial \Omega \,,
\end{cases}
\end{equation}
where 
\begin{equation}\label{fL2} 
f \in L^2(\Omega)
\end{equation}
 and $\nu$ denotes the outward unit normal on $\partial \Omega$.  

As suggested by diverse recent investigations, the gradient regularity of solutions to second-order nonlinear elliptic equations in divergence form, such as \eqref{eq_anisotr_intro}, can be properly formulated in terms of the stress-field, namely the function of the gradient under the divergence operator, see e.g. \cite{AKM, BDGP, BDW, BCDS, BCDKS, ciama, DKS, KuMi}. A similar point of view is adopted in this paper. Under suitable regularity assumptions on the Young function $B$, on the norm $H$ and on the domain $\Omega$, we show that if $u$ is the solution to either the Dirichlet problem \eqref{eq:dir2} or the Neumann problem \eqref{eq:neu2}, then 
\begin{equation}\label{dic1}
 \A(\nabla u)  \in W^{1,2}(\Omega).
\end{equation}
Since, conversely,  the latter piece of information trivially ensures that $f \in L^2(\Omega)$, property \eqref{dic1} amounts to the maximal possible regularity of $\nabla u$ under assumption \eqref{fL2}.

The emphasis of this contribution is not as much on the general Orlicz growth of the differential operators considered, but rather on their anisotropy. Indeed, our results for anisotropic norms $H$ are original even in the case when $B(t)$ is a power, just $t^2$ for instance.
%
By contrast, results in the same vein are available in the isotropic regime.
They are classical for Dirichlet problems for linear equations in smooth domains \cite{Bernstein, Schauder}. Linear problems in non-smooth domains satisfying minimal regularity assumptions were considered in \cite{Maz67,Maz73}. Global results for isotropic nonlinear equations are the subject of \cite{cia}. Systems are treated in \cite{cia19} and in \cite{balci}.  The same papers also deal with local estimates for local solutions. Earlier local results, under stronger assumptions on the function $f$ on the right-hand side of the equations, can be found in \cite{DaSci, lou}.  Gradient regularity of fractional-order for
solutions to  $p$-Laplacian type equations is the subject of the early
paper \cite{SimonJ}, and of
\cite{AKM, BSY, Cel1, Mi1, Mis}. Generalization of \cite{cia} are in the recent papers \cite{acf,  gm}.
Related contribution are \cite{miao, montoro, sarsa}. In particular, the papers \cite{acf, gm}  provide regularity estimates in the same spirit as \eqref{dic1}, but for local solutions, and hence only of local nature. With this regard, let us stress that minimal regularity assumptions on the domain $\Omega$ 
will be imposed for our global bounds. 

For completeness, we yet include a local estimate, asserting that, if $u$ is a local solution to equation \eqref{eq_anisotr_intro}, with 
\begin{equation}\label{fL2loc} 
f \in L^2_{\rm loc}(\Omega),
\end{equation}
then 
\begin{equation}\label{dic1loc}
 \A(\nabla u)  \in W^{1,2}_{\rm loc}(\Omega).
\end{equation}
This result complements those of \cite{acf} and  \cite{gm}, which are restricted to data $f$ in the dual of the natural Sobolev ambient space.

Loosely speaking, our approach consists in squaring both sides of equation \eqref{eq_anisotr_intro}, integrating the resultant equation over $\Omega$, and exploiting a new anisotropic  Reilly type   identity and some integral inequalities which eventually yield the desired Sobolev regularity of $\mathcal A(\nabla u)$, via an estimate for the corresponding norm.
Dealing with anisotropic operators calls for the introduction of original differential and integral identities and inequalities for vector fields. In particular, an anisotropic second fundamental form on $\partial \Omega$, associated with the norm $H$, comes into play. 
 The overall argument requires a degree of smoothness of the function $u$ and of the domain $\Omega$ which are not guaranteed for the solutions to problems  \eqref{eq:dir2} and \eqref{eq:neu2} under the assumptions to be imposed on $B$, $H$, $f$ and $\Omega$. Substantiating this approach entails approximations at various levels, involving smoothing of the differential operator, the right-hand side of the equation, and the domain. The approximation process presents additional difficulties, which do not arise in the isotropic regime. Indeed,    the intrinsic irregularity of any non-Euclidean norm $H$ at the origin does not guarantee the required regularity of solutions to customarily regularized problems.
 
 Let us add that new   information is also provided about the dependence of the constants in our regularity estimates. 
 The offered quantitative bounds on the relevant constants   can be of use, for instance, in the numerical analysis of problems \eqref{eq:dir2} and \eqref{eq:neu2} in sequences of regular domains whose Lipschitz constant blows up at a known rate. Keeping track of the constants at various steps of our proofs calls for precise capacity estimates near the boundary of the domains. Altogether,
this yields novel insight 
 even in the case of standard (isotropic, polynomial-driven) nonlinear elliptic operators.

\section{Main results}\label{main}

We begin by formulating the basic assumptions of our regularity results. 
The function $B$ is supposed to be twice continuously differentiable and to have a nonlinear growth. Precisely, on setting
\begin{equation}\label{dic102}
    i_b=\inf_{t>0}\frac{t\,b'(t)}{b(t)}\quad\text{and }\quad s_b=\sup_{t>0}\frac{t\,b'(t)}{b(t)},
\end{equation}
the nonlinear growth condition on the function $B$ is imposed by requiring that 
\begin{equation}\label{ib}
i_b>0.
\end{equation}
Property \eqref{ib} is equivalent to the so-called $\nabla_2$-condition in the theory of Young functions. A doubling condition, known as $\Delta_2$-condition in this theory, is also demanded on $B$. The latter is  equivalent to 
\begin{equation}\label{sb}
s_b<\infty.
\end{equation}
As mentioned above, the standard choice 
\begin{equation}\label{power}B(t) = t^p,
\end{equation}
corresponds to operators with plain $p$-growth, with $p>1$. Multiplying the function in \eqref{power} by powers of logarithms results in 
 functions, that are still admissible, or the form:
$$B(t) = t^p \log^q (c+t), $$
where $p>1$, $q\in \mathbb R$, and $c$ is a  positive, sufficiently large constant for $B$ to be convex.
\\ More elaborated instances, borrowed from \cite{Talenti}, are:
\begin{align*}
 &B(t) = t^3(1 + (\ln t)^2)^{-\frac 12} \exp(\ln t \arctan (\ln t)));
\\ &B(t) = t^{4+ \sin \sqrt{1+(\ln t)^2} }.
\end{align*}

  As far as the norm $H$ is concerned, we assume that $H\in C^2(\rn \setminus \{0\})$. Hence, since $H^2$ is homogeneous of degree $2$, the largest eigenvalue of the matrix $ \tfrac{1}{2} \nabla^2_\xi H^2(\xi)$ is uniformly bounded from above for $\xi\neq 0$ by some positive constant $\Lambda$. As an ellipticity condition, we also assume that its smallest eigenvalue is uniformly bounded from below  by some positive constant $\lambda$. Therefore,
\begin{equation}\label{ell:H}
    \l\,|\eta|^2\leq \tfrac{1}{2} \nabla^2_\xi H^2(\xi)\,\eta \cdot \eta \leq \L\,|\eta|^2\quad\text{for $\xi\neq 0$, and $\eta\in\rn$.}
\end{equation}
The first inequality in \eqref{ell:H}  is equivalent to the geometric condition that
\begin{equation*}
\text{the unit ball}\,\,  \{H(\xi)<1\}\text{ is uniformly convex},
\end{equation*}
 see \cite{cfv}. Recall that a convex set in $\rn$ is said to be uniformly convex 
if all the principal curvatures of its boundary are bounded away from zero.
\\ For instance, any function $H$ of the form
\begin{equation*}
    H(\xi)=\big( \alpha \,K^p(\xi)+\beta \,|\xi|^p\big)^{1/p},
\end{equation*}
for some norm $K\in C^2(\rn\setminus\{0\})$, and some $p>1$ and $\alpha, \beta >0$, fulfills condition  \eqref{ell:H}, and it is hence allowed -- see \cite{cfv} and \cite[Example 2.1]{acf}. 
In particular, the norm
\begin{equation*}
    H(\xi)=\bigg( \Big(\sum _{i=1}^n |\xi_i|^q\Big)^{p/q}+\,|\xi|^p\bigg)^{1/p}
\end{equation*}
 is admissible for any $p>1$ and any $q \geq 2$.
\\ Further examples of norms $H$ that fit the present setting are supplied by the Minkowski gauge of any bounded,  uniformly convex set $E$, with $\partial E \in C^2$, such that $E=-E$. Namely,
\begin{equation*}
   H(\xi) = \inf \{r>0\,:\, \xi\in r\,E \}.
\end{equation*}
This is a consequence of the fact that $\{H<1\} = E$.

Finally, as mentioned above, the function $f$ is supposed to belong to $L^2(\Omega)$. As is well known, weak solutions to problems \eqref{eq:dir2} and \eqref{eq:neu2} are not well defined and need not exist under this assumption on $f$. Indeed, the space $L^2(\Omega)$ is not included in the dual of the natural Orlicz-Sobolev space associated with these problems unless $B$ grows fast enough near infinity. For instance, if $B(t)$ behaves like $t^p$ near infinity, then $p$ has to exceed $\frac {2n}{n+2}$. Thus, an even weaker notion of generalized solution has to be employed. Various definitions of solutions to nonlinear elliptic equations in divergence form, with a right-hand side affected by a low integrability degree, have been introduced in the literature, see e.g. \cite{boccardo, BBGGPV, Daglio, lionsmurat}. 

A posteriori, they turn out to be equivalent. Precise formulations of the definitions adopted here are given at the beginning of Sections \ref{sec:local}, \ref{sec:dir} and \ref{sec:neumann}, that deal with local solutions, solutions to Dirichlet problems and solutions to Neumann problems, respectively. Let us just 
disclose here that the solutions in question are not weakly differentiable in general. Hence, the expression $\nabla u$ appearing in our statements is an abuse of notation for a surrogate gradient which has to be properly interpreted. In this connection, we point out one trait of our results, which, in particular, reveal the regularizing effect of the nonlinear function $\mathcal A$, that turns $\mathcal A(\nabla u)$ into a true Sobolev map.

\smallskip

We begin with our local result, which does not require any additional assumption on $\Omega$.
In what follows, $B_R$ denotes a ball with radius $R$ and $B_{2R}$ a ball with the same center and radius $2R$.

\begin{theorem}[Local estimate]\label{thm:local} 
 Assume that $B\in C^2(0,\infty)$ is a Young function fulfilling conditions \eqref{ib} and \eqref{sb}, and that $H\in C^2(\rn \setminus \{0\})$ is a norm satisfying property \eqref{ell:H}.
Let $\Om$ be any open set in $\rn$. Assume that $f\in L^2_{loc}(\Om)$, and 
let $u$ be a generalized local solution to equation \eqref{eq_anisotr_intro}.
Then,
\begin{equation}\label{gen30}
\A(\nabla u)\in W^{1,2}_{\rm loc}(\Om), 
\end{equation}
and there exists a constant $c=c(n,i_b,s_b,\l,\L)$ such that
\begin{equation}\label{loc:est}
    \begin{split}
     & \|\A(\nabla u)\|_{L^2(B_{R})}\leq 
c\,R \|f\|_{L^2(B_{2R})}  
\,+ c\, R^{-\frac n2} \|\A(\nabla u)\|_{L^2(B_{2R})}
\\
       &\|\nabla (\A(\nabla u))\|_{L^2(B_{R})}\leq c\, \|f\|_{L^2(B_{2R})} + c\,R^{-\frac n2 -1} \|\A(\nabla u)\|_{L^2(B_{2R})}
    \end{split}
\end{equation}
%
%
%
for every ball  $B_{2R} \subset\subset \Om$.
\end{theorem}

When it comes to global estimates -- the core of our investigations --  the geometry of $\Omega$ plays a crucial role. The first result with this regard deals with plainly bounded convex domains $\Omega$. In this case, 
no additional regularity of $\Omega$ is needed. As will be clear from the proof, this is possible thanks to the fact that a priori bounds for $\|\mathcal A(\nabla u)\|_{W^{1,2}(\Omega)}$ involve
  certain integrals over  $\partial \Omega$, which depend on its curvatures. If $\Omega$ is convex, these integrals have a definite sign, which makes the integrals in question negligible in the relevant bounds. Importantly, the constants in the bounds depend on a convex domain $\Omega$ only through its diameter $d_\Om$.


\begin{theorem}[Convex domains]\label{an:mainthm}  Let $B$ and $H$ be as in Theorem \ref{thm:local}.
 Let $\Omega$  be a bounded convex set in $\rn$. 
Assume that $f\in L^2(\Omega)$ and let $u$ be a generalized solution to either the Dirichlet problem  \eqref{eq:dir2} or the Neumann problem \eqref{eq:neu2}.  
Then,
\begin{equation*}
    \A(\nabla u)\in W^{1,2}(\Omega).
\end{equation*}
Moreover, 
\begin{equation}\label{est:A}
    \|\A(\nabla u)\|_{L^2(\Omega)}\leq c_1\,\|f\|_{L^2(\Omega)} \quad \text{and} \quad \|\nabla (\A(\nabla u))\|_{L^2(\Omega)}\leq c_2\,\|f\|_{L^2(\Omega)},
\end{equation}
where
\begin{equation*}
c_1=c(n,i_b,s_b,\l,\L)\,d_\Om\quad\text{and}\quad
c_2= \frac{\L\,\max\{1,s_b\}}{\l\,\min\{1,i_b\}}.
\end{equation*}
In particular, the constant $c_2$ is independent of $\Omega$.
\end{theorem}


As soon as the realm of convex domains is abandoned, the conclusions of Theorem \eqref{an:mainthm} can fail, in the absence of additional assumptions on the curvatures of $\partial \Omega$. Indeed, counterexamples in this connection can be exhibited, even for the plain Laplace operator, for slight perturbations of convex domains. Consider, for instance, a bounded open set $\Omega$ whose boundary is smooth outside a small portion, where it agrees with the graph of a function $\Theta$ of the variables $(x_1, \dots, x_n)$, given by
\begin{equation}\label{Theta}
\Theta (x_1, \dots, x_n)= \frac{c |x_1|}{\log |x_1|}
\end{equation}
for some constant $c$ and for small $x_1$. As shown in \cite{Maz67, Maz73}, if the constant $c$ is not small enough, then one can exhibit Dirichlet problems for the Laplacian, with smooth right-hand sides, whose solutions do not belong to $W^{2,2}(\Omega)$.

 \smallskip\par 
A suitable assumption on $\Omega$ that restores the result involves integrability properties of the weak curvatures of $\partial \Omega$. One can request that $\Omega$ is a bounded Lipschitz domain such that the functions of $(n-1)$ variables, that locally describe $\Omega$ around boundary points, are endowed with second-order weak derivatives which belong to a specific Marcinkiewicz space depending on the dimension $n$. Also, the norm of the curvatures in this space, evaluated on balls centered on $\partial \Omega$, has to be sufficiently small for small radii of the balls. Specifically, denote by $\mathcal B$ the weak second fundamental form on $\partial \Omega$, and define the function $\Psi_\Omega : (0, \infty) \to [0, \infty]$ as 
\begin{equation}\label{def:psi}
\Psi_\Omega(r)=
\begin{cases}
     \sup\limits_{x\in\partial\Omega} \| \mathcal B\|_{L^{n-1,\infty}(\partial\Omega\cap B_r(x))}\quad\text{if }n\geq 3,
\\
\\
  \sup\limits_{x\in\partial\Omega} \| \mathcal B\|_{L^{1,\infty}\log L(\partial\Omega\cap B_r(x))}\quad\text{if }n= 2
    \end{cases}
\end{equation}
for $r > 0$. Here, $L^{n-1,\infty}$ and $L^{1,\infty}\log L$ denote Marcinkiewicz type spaces, with respect to the $(n-1)$-dimensional Hausdorff measure $\mathcal H^{n-1}$ on $\partial \Omega$.
Recall that
$$
\|g\|_{L^{n-1, \infty} (\partial \Omega\cap B_r(x))} = \sup _{s \in (0, \mathcal H^{n-1}(\partial \Omega\cap B_r(x)))} s ^{\frac 1{n-1}} g^{**}(s)
$$
for $n\geq 3$, and
$$
\|g\|_{L^{1, \infty} \log L  (\partial \Omega\cap B_r(x))} = \sup _{s \in (0, \mathcal H^{n-1}(\partial \Omega\cap B_r(x)))} s \log\big(1+ \tfrac{1}s\big) g^{**}(s),
$$
for a measurable function $g: \partial \Omega \to \mathbb R$. 
Here $g^*$ denotes the decreasing rearrangement of $g$ with respect to the measure $\mathcal H^{n-1}$, and $g ^{**}(s)=s^{-1} \smallint _0^s g^* (r)\, dr$ for $s >0$.
In what follows, the notation $\partial \Omega \in W^2L^{n-1,\infty}$ or $\partial \Omega \in W^2L^{1,\infty}\log L$ means that the weak curvatures of $\partial \Omega$ belong to the respective Marcinkiewicz spaces. An analogous notation will be adopted to denote that the weak curvatures in question belong to some other space.

Furthermore,   by  $\mathfrak L_\Omega$  we indicate a Lipschitz characteristic of $\Om$, which depends on the Lipschitz constant $L_\Om$ of the functions which locally describe $\partial \Om$ and on the radius $R_\Om$ of their ball domains.
 Its precise definition is given at the beginning of Section \ref{S:trace}.   

\begin{theorem}[Domains with minimally integrable curvatures]\label{an:mainthm1}
 Let $B$ and $H$ be as in Theorem \ref{thm:local}.  Let  $\Omega$ be a bounded Lipschitz domain in $\R^n$, with Lipschitz characteristic $\mathfrak L_\Om$, such that  $\partial \Omega \in W^2L^{n-1,\infty}$ if $n \geq 3$ or $\partial \Omega \in W^2L^{1,\infty}\log L$ if $n=2$. 
Assume that $f\in L^2(\Omega)$ and let $u$ be a generalized solution to either the Dirichlet problem  \eqref{eq:dir2} or the Neumann problem \eqref{eq:neu2}.  There exists a constant $\k_0=\k_0(n, i_b, s_b,\l,\L, \mathfrak L_\Omega, d_\Omega)$ such that, if 
\begin{equation}\label{bhc:1}
    \lim_{r\to 0^+}\Psi_{\Omega}(r)<\k_0,
\end{equation}
then,
\begin{equation*}
    \A(\nabla u)\in W^{1,2}(\Omega).
\end{equation*}
Moreover, 
\begin{equation}\label{est:A1}
  \|\A(\nabla u)\|_{W^{1,2}(\Omega)}\leq c\,\|f\|_{L^2(\Omega)}
%
\end{equation}
for a suitable constant  $c=c(i_b, s_b, \lambda, \Lambda, \Om)$.
\end{theorem}

We stress that the use of Marcinkiewicz norms and, in particular, the smallness condition \eqref{bhc:1} are not just due to technical reasons. They are in fact minimal assumptions in terms of integrability properties of the curvatures of $\partial \Omega$, for $\A(\nabla u)$ to belong to $W^{1,2}(\Omega)$. 
This can be shown, for instance, via an example from \cite{Krol72}, for $n=3$ and $p \in (\frac 23, 2]$, in the standard isotropic case. In that paper, open sets $\Omega \subset \mathbb R^3$ are displayed such that $\partial \Omega \in W^2L^{2,\infty}$, for which the limit in \eqref{bhc:1} is yet too large, and the solution $u$ to the Dirichlet problem for the $p$-Laplace equation, with a smooth right-hand side, is such that $|\nabla u|^{p-2}\nabla u \notin W^{1,2}(\Omega)$. In \cite{Maz67} two-dimensional Dirichlet problems for the Laplace operators are considered. In particular, open sets $\Omega$, with $\partial \Omega \in W^2L^{1,\infty}\log L$ are exhibited where the solution to the Poisson equation with a smooth right-hand side does not belong in $W^{2,2}(\Omega)$. This is again due to a large value of the limit in \eqref{bhc:1}. Related Neumann problems are considered in \cite[Section 14.6.1]{MaSch}. 

\smallskip\par 
The result of Theorem \ref{an:mainthm1} can still be sharpened, if assumptions of a somewhat different nature are allowed. They entail the use of a weighted isocapacitary function for subsets of $\partial \Omega$, the weight being the norm of the second fundamental form on $\partial \Omega$. This function is denoted by $\mathcal K_\Omega : (0,\infty) \to [0, \infty)$ and   defined by
\begin{equation}\label{dic7}
\mathcal K_\Omega(r) =
\displaystyle 
\sup _{
\begin{tiny}
 \begin{array}{c}{E\subset   B_r(x)} \\
x\in \partial \Omega
 \end{array}
  \end{tiny}
} \frac{\int _{\partial \Om \cap E} |\mathcal B|d\mathcal H^{n-1}}{{\rm cap} (E, B_r(x))}\qquad \hbox{for $r>0$}\,.
\end{equation}
Here, $B_r(x)$ stands for the ball centered at $x$, with radius $r$,  and ${\rm cap} (E, B_r(x))$   for the classical capacity of a compact set $E$ relative to  $B_r(x)$. This is the content of the next result.

\begin{theorem}[Domains satisfying a boundary isocapacitary inequality]\label{an:mainthm2}
 Let $B$ and $H$ be as in Theorem \ref{thm:local}.  Let  $\Omega$ be a bounded Lipschitz domain in $\R^n$, with Lipschitz characteristic $\mathfrak L_\Om$, such that  $\partial \Omega \in W^{2,1}$.
Assume that $f\in L^2(\Omega)$ and let $u$ be a generalized solution to either the Dirichlet problem  \eqref{eq:dir2} or the Neumann problem \eqref{eq:neu2}.  There exists a constant $\k_1=\k_1( n, i_b, s_b,\l,\L,\mathfrak L_\Om, d_\Omega)$ such that, if 
\begin{equation}\label{condK}
    \lim_{r\to 0^+}\mathcal K_{\Omega}(r)<\k_1,
\end{equation}
then,
\begin{equation*}
    \A(\nabla u)\in W^{1,2}(\Omega).
\end{equation*}
Moreover, 
\begin{equation}\label{est:A2}
 \|\A(\nabla u)\|_{W^{1,2}(\Omega)}\leq c\,\|f\|_{L^2(\Omega)}
%
\end{equation}
for a suitable constant  $c=c(i_b, s_b, \lambda, \Lambda, \Om)$.
%
%
\end{theorem}

Theorem \ref{an:mainthm2} is  stronger than \ref{an:mainthm1}. Indeed, the former not only implies the latter, but also applies to less regular domains. This is the case, for instance, of the sets described above, whose boundary locally agrees with the graph of the function $\Theta$ given by \eqref{Theta}. Actually, condition \eqref{condK} is fulfilled by these domains, provided that the constant $c$ appearing in \eqref{Theta} is small enough, whereas $\partial \Omega \notin W^2L^{n-1, \infty}$. The same domains also demonstrate the necessity of the smallness condition  \eqref{condK}, since, as mentioned above, the conclusions of Theorem \ref{an:mainthm2} fail if the constant $c$, and hence the limit in \eqref{condK}, exceeds some threshold. 

\smallskip\par
 We conclude this section with a statement concerning sufficiently regular domains -- specifically, domains $\Om$ such that $\partial \Om \in C^{2,\a}$. Under this assumption, the constants  appearing in the $W^{1,2}(\Om)$ estimate of the stress field $\A(\nabla u)$ admit  bounds  with an explicit dependence on $d_\Om$, $L_\Om$, $R_\Om$, and $\|\mathcal B\|_{L^\infty(\partial \Omega)}$. Thanks to the monotonicity of this dependence, the bounds in question are uniform in classes of domains $\Omega$ where $d_\Om$, $L_\Om$ and $\|\mathcal B\|_{L^\infty(\partial \Omega)}$ are uniformly bounded from above, and $R_\Om$ from below.

\begin{theorem}[Domains with   bounded curvatures]\label{thm:C11}
    Let   $B$ and $H$ be as in Theorem \ref{thm:local}.  
Let  $\Om$ be a bounded open set in $\rn$ such that $\partial\Om\in C^{2,\a}$ for some $\a \in (0,1)$, and let  $\mathfrak L_\Om = (L_\Om, R_\Om)$ be a   Lipschitz characteristic of $\Om$. Assume that  $f \in L^2(\Om)$ and  let 
$u$  be 
   a generalized solution to either the Dirichlet problem \eqref{eq:dir2} or the Neumann problem \eqref{eq:neu2}.  Then,
    \begin{equation*}
        \A(\nabla u)\in W^{1,2}(\Om)\,.
    \end{equation*}
    Moreover, 
    \begin{equation}
        \|\A(\nabla u)\|_{L^2(\Om)}\leq c_1\,\|f\|_{L^2(\Om)}\quad\text{and}\quad \|\nabla\big(\A(\nabla u)\big)\|_{L^2(\Om)}\leq c_2\,\|f\|_{L^2(\Om)},
    \end{equation}
   where:
\\ if $n\geq 3$, then
$$ \begin{cases}
c_1=
           c\, d_\Om^{\,p(n)}\,(1+{L_\Om})^{n+2}\,\max\Big\{(1+L_\Om)^{t(n)}\,\|\mathcal B\|_{L^\infty(\partial \Omega)}^{(2n+2)(n+2)}, R_\Om^{-(2n+2)(n+2)}\Big\}
\\ \\
 c_2= c\,d_\Om^{\,p(n)+n}\,(1+{L_\Om})^{n+2}\,\max\Big\{(1+L_\Om)^{t(n)+9(n+2)}\,\,
\|\mathcal B\|_{L^\infty(\partial \Omega)}
^{(2n+3)(n+2)}, R_\Om^{-(2n+3)(n+2)}\Big\}
\end{cases}
$$
and   $c=c(n,\l,\L,i_b,s_b)$,   $p(n)=(2n+1)(n+2)+n$,  and  $t(n)=9(n+2)(2n+2)$.
\\ If $n=2$, then
$$\begin{cases}
c_1 &  =\displaystyle c'\,{d}_\Om^{22}\,(1+L_\Om)^4\,\max\bigg\{ \frac{(1+{L}_\Om)^{288}\,\big(1+\|\mathcal B\|_{L^\infty(\partial \Omega)}\big)^{24}}{\log ^{24}\big(1+ c(1+L_\Om)\big(1+\|\mathcal B\|_{L^\infty(\partial \Omega)}\big)\big)}, R_\Om^{-24}\bigg\}
        \\ \\
        c_2 & = \displaystyle c' \,{d}_\Om^{24}\,(1+L_\Om)^4\,\max\bigg\{\frac{(1+{L}_\Om)^{336}\,\big(1+\|\mathcal B\|_{L^\infty(\partial \Omega)}\big)^{28}}{\log ^{28}\big(1+ c(1+L_\Om)\big(1+\|\mathcal B\|_{L^\infty(\partial \Omega)}\big)\big)}, R_\Om^{-28}\bigg\},
\end{cases}
$$
where $c= c (\l,\L,i_b,s_b)$ and $c'= c' (\l,\L,i_b,s_b)$.
%
%
%
%
%
\end{theorem}

Let us point out that, in the case of solutions to Neumann problems, our proof also applies as soon as that $\partial \Om \in C^2$. In fact, a refinement of the arguments in the proof of Theorem \ref{thm:C11} would lead to the same conclusions, for both Dirichlet and Neumann problems, under the weaker assumption that $\partial \Om \in C^{1,1}$. This generalization is skipped in order to avoid additional technicalities.

\section{The functions $H$, $B$ and $\mathcal A$}\label{sec:prel}


In this section we collect some analytic properties of norms $H\in C^2(\rn \setminus \{0\})$ satisfying condition \eqref{ell:H} and of Young functions $B\in C^2(0, \infty)$ fulfilling assumptions \eqref{ib}--\eqref{sb}.  Ensuing properties of the function $\mathcal A$ given by \eqref{def_A}, and of a family of associated regularized functions,  will also be presented.
\\
To begin with, observe that $$H^2 \in C^1(\rn),$$
 inasmuch as $H\in C^2(\rn \setminus \{0\})$. 
 \\ 
The homogeneity of $H^2$ implies that
\begin{equation} \label{H2omog}
    \tfrac{1}{2} \nabla^2_\xi H^2(\xi)\,\xi \cdot \xi =H^2(\xi)\quad \text{for $\xi\neq 0$.}
\end{equation}
Hence, owing to assumption \eqref{ell:H},
\begin{equation}\label{bounds:H}
    \l\,|\xi|^2\leq H^2(\xi)\leq \L\,|\xi|^2\quad  \text{for $\xi\in \rn$.}
\end{equation}
%
%
%
%
We denote by $H_0$ the dual norm of $H$, defined as 
\begin{equation} \label{def_H0}
    H_0(x)=\sup_{\xi\neq 0}\frac{\xi\cdot x}{H(\xi)}\quad \text{for $x \in \rn$.}
\end{equation}
As a consequence of  \eqref{bounds:H} and \eqref{def_H0}, one has that
\begin{equation}\label{bounds:H0}
   \frac{1}{\L}|x|^2\leq H_0^2(x)\leq \frac{1}{\l}|x|^2\quad \text{for $x \in \rn$.}
\end{equation}
By \cite[Lemma 3.1]{CiSa1},
\begin{equation}\label{feb300}
H_0(\nabla_\xi H(\xi)) = 1\quad \text{for $\xi\neq 0$.}
\end{equation}
Thereby, from \eqref{bounds:H0} we infer that
\begin{equation} \label{bound:nabH1}
\sqrt{\lambda} \leq |\nabla_\xi H(\xi)| \leq \sqrt{\Lambda}  \quad \text{for $\xi\neq 0$.}
\end{equation}
The homogeneity of the function $H$ ensures that
\begin{equation}\label{dic101}
    \xi \cdot \nabla_\xi H(\xi) =H(\xi) \quad \text{for $\xi\neq 0$.}
\end{equation}
and
\begin{equation}\label{dic100}
    \nabla_\xi^2 H(\xi)\,\xi=0\quad \text{for $\xi\neq 0$.}
\end{equation}
The behavior of the matrix  $\nabla_\xi^2 H(\xi)$ when acting on $\xi^\perp$, the subspace of $\rn$ orthogonal to the vector $\xi\neq 0$, is described in the following result. In the statement, $\mathrm{Id}$ denotes the unit matrix in $\rn$.
%
%
%
%
%

\begin{lemma}\label{pr:ellh1}
Let $\xi\in \mathbb{S}^{n-1}$. Then,  
\begin{equation}\label{jan10}
\nabla^2_\xi H(\xi)\,:\xi^\perp\to \xi^\perp.
\end{equation} 
Moreover, tha map \eqref{jan10} is 
 an isomorphism and  
\begin{equation}\label{ucH11}
    \frac{\l}{\sqrt{\L}} \mathrm{Id} \leq \nabla^2_\xi H(\xi)\leq \frac{\L}{\sqrt{\l}}\,\bigg(1+\frac{\L}{\l}\bigg)\,\mathrm{Id} \qquad \text{on  $\xi^\perp$.}
\end{equation}
\end{lemma}

\begin{proof}
We recall that
\begin{equation}\label{expand}
    \tfrac{1}{2}\nabla^2_\xi H^2(\xi)=H(\xi)\,\nabla^2_\xi H(\xi)+\nabla_\xi H(\xi)\otimes\nabla_\xi H(\xi) \qquad \text{for $\xi \neq 0$.}
\end{equation}
Let $\xi\in\mathbb{S}^{n-1}$. 
Since \eqref{dic101} implies $\xi \cdot\nabla_\xi H(\xi) \neq 0$, we have that 
$\xi$ and $\nabla_\xi H(\xi)^\perp$ span the whole $\rn$.
Thus, for every $\eta\in\xi^\perp$ such that $|\eta|=1$, there exist $\a\in \R$ and $\zeta\in \nabla_\xi H(\xi)^\perp $ such that
\begin{equation*}
    \eta=\a\,\xi+\zeta.
\end{equation*}
In particular
\begin{equation*}
    \eta\cdot\nabla_\xi H(\xi)=\a\,\xi \cdot \nabla_\xi H(\xi)=\a\,H(\xi).
\end{equation*}
Thus, from estimates \eqref{bounds:H}, \eqref{bound:nabH1} and $|\eta|=|\xi|=1$, we infer
\begin{equation}\label{saup}
    |\a|=\bigg|\frac{\eta\cdot\nabla_\xi H(\xi)}{H(\xi)}\bigg|\leq \sqrt{\frac{\L}{\l}}.
\end{equation}
Since   $\eta\perp \xi$,  
\begin{equation*}
    |\zeta|^2=|\eta-\a\,\xi|^2=1+\a^2.
\end{equation*}
An application of inequalities   \eqref{ell:H} with $\eta=\zeta$, equation  \eqref{expand}, and 
the fact that $\zeta\in\nabla_\xi H(\xi)^\perp $ imply
\begin{equation}\label{rr1}
    \l\,(1+\a^2)\leq H(\xi)\,\nabla^2_\xi H(\xi)\zeta\cdot\zeta\leq \L\,(1+\a^2),
\end{equation}
Inasmuch as $\nabla^2_\xi H(\xi)\,\xi=0$ and $\zeta=\eta-\a\xi$,  
\begin{equation*}
    \nabla^2_\xi H(\xi)\zeta\cdot\zeta=\nabla^2_\xi H(\xi)\,\eta\cdot \eta.
\end{equation*}
Coupling the latter equality with inequalities \eqref{rr1} implies that
\begin{equation}\label{rr2}
     \l\leq\l\,(1+\a^2)\leq H(\xi)\,\nabla^2_\xi H(\xi)\,\eta\cdot \eta\leq \L\,(1+\a^2).
\end{equation}
Hence, \eqref{ucH11} follows via  \eqref{saup}.
\\ From the symmetry of the matrix $\nabla^2_\xi H(\xi)$ one can deduce that it maps $\rn$ into $\xi^\perp$. Furthermore, thanks to property \eqref{ucH11}, $\nabla^2_\xi H(\xi)\,\eta \neq 0$ if $\eta \in \xi^\perp \setminus \{0\}$. Hence, the map \eqref{jan10} is actually an isomorphism.
\end{proof}

Let $b$ be the function appearing in equation \eqref{B}. By assumption \eqref{ib},
\begin{equation}\label{limb}
\lim _{t\to 0^+}b(t)=0.
\end{equation}
The monotonicity of the function $b$ ensures that
\begin{equation} \label{Bb}
\tfrac t2 b(\tfrac t2)\leq B(t)  \leq b(t) t \quad \text{for $t \geq 0$.}
\end{equation}
Also, as a consequence of assumption \eqref{dic102},   the functions $\frac{b(t)}{t^{i_b}}$ and $\frac{b(t)}{t^{s_b}}$
are non-decreasing and non-increasing, respectively.
Hence,
\begin{equation}\label{usual:b}
    b(1)\,\min\{t^{i_b},t^{s_b}\}\leq b(t)\leq b(1)\,\max\{t^{i_b},t^{s_b}\}\quad \text{for $t\geq 0$,}
\end{equation}
and there exist positive constants $c$ and $C$, depending only on $i_b$ and $s_b$, such that
\begin{equation}\label{usual:b1}
    c\,b(s)\leq b(t)\leq C\,b(s) \quad \text{if $0\leq s\leq t\leq 2s$.}
\end{equation}
Now, define the function $a: (0, \infty) \to [0, \infty)$ as
 $$a(t) = \frac{b(t)}t \qquad \text{for $t>0$.}$$
Thereby, $a\in C^1(0, \infty)$. Moreover, on setting
\begin{equation}\label{dic103}
    i_a=\inf_{t>0}\frac{t\,a'(t)}{a(t)}\quad\text{and }\quad s_a=\sup_{t>0}\frac{t\,a'(t)}{a(t)},
\end{equation}
one has that
\begin{equation}\label{dic105}
\text{$i_b= i_a+1$\quad  and \quad $s_b=s_a+1$.}
\end{equation}
Hence, assumptions \eqref{ib} and \eqref{sb} are equivalent to 
\begin{equation}\label{c:a}
    -1<i_a\leq s_a<\infty\,,
\end{equation}
and a counterpart of \eqref{usual:b1} holds;  namely,
\begin{equation}\label{usual:a1}
    c\,a(s)\leq a(t)\leq C\,a(s)\quad\text{if $0<s\leq t\leq 2s$.}
\end{equation}
Also 
\begin{equation}\label{jan201}
\text{$i_B\geq i_b+1$ \quad and \quad  $s_B\leq s_b +1$.}
\end{equation}
Thus, assumption \eqref{sb} implies that $s_B< \infty$ as well.
\\ Assumption \eqref{sb} also ensures 
 that for every $k >0$ there exists  a constant $c$, depending only on $k$ and $s_b$, such that
\begin{equation}\label{bdelta2}
b(k t) \leq c b(t) \qquad \hbox{for $t \geq 0$.}
\end{equation}
Similarly, the fact that $s_B< \infty$ ensures that for every $k>0$ there exists  a constant $c$, depending only on $k$ and $s_B$, such that
\begin{equation}\label{Bdelta2}
B(k t) \leq c B(t) \qquad \hbox{for $t \geq 0$.}
\end{equation}
Moreover, since $i_B>1$, a parallel property holds for $\widetilde B$. Namely,
for every $k >0$ there exists  a constant $c$, depending only on $k$ and $i_B$, such that
\begin{equation}\label{Btildedelta2}
\widetilde B(k t) \leq c \widetilde B(t) \qquad \hbox{for $t \geq 0$.}
\end{equation}
The property  $s_B< \infty$ also implies that there exists a constant $c$ such that
\begin{equation}\label{feb1}
\widetilde B(b(t)) \leq c B(t)  \qquad \hbox{for $t \geq 0$.}
\end{equation}
Thanks to \eqref{Bdelta2},
%
%
%
%
inequalities  \eqref{bounds:H} imply that
\begin{equation}\label{B:equiv}
     c  B(|\xi|)  \leq B(H(\xi))\leq C  B(|\xi|) \quad \text{for $\xi \in \rn$,}
\end{equation}
for suitable constants $c=c(s_B, \l,\L)$ and $C=C(s_B,\l,\L)$,
and
\begin{equation}\label{b:equiv}
     c  b(|\xi|)  \leq b(H(\xi))\leq   C b(|\xi|) \quad \text{for $\xi \in \rn$,}
\end{equation}
for suitable constants $c=c(s_b, \l,\L)$ and $C=C(s_b,\l,\L)$. Moreover, inequality 
\eqref{bound:nabH1} yields
\begin{equation}\label{A:bds_sopra}
\sqrt{\l}\, b(H(\xi)) \leq |\A(\xi)|\leq \sqrt{\L}\, b(H(\xi)) \quad \text{for $\xi \in \rn$.}
\end{equation}       
Notice that    the function $\mathcal A$ admits the alternate expression 
\begin{equation} \label{def_Abis}
    \A(\xi) = a\big(H(\xi)\big)\tfrac{1}{2}\nabla_\xi H^2(\xi) \quad \text{for $\xi \neq 0$.}
\end{equation}
 Hence, via equation 
\eqref{H2omog} and the 
 homogeneity of $H^2$, we deduce that
\begin{equation*} 
    \begin{split}
        \A(\xi)\cdot\xi = a(H(\xi)) H^2(\xi) = b(H(\xi))H(\xi) \quad \text{for $\xi \in \rn$.}
    \end{split}
\end{equation*}
Coupling the latter equation with inequality \eqref{Bb} yields
\begin{equation}\label{A:bds}
    \begin{split}
        \A(\xi)\cdot\xi  \geq B\big(H(\xi)\big) \quad \text{for $\xi \in \rn$.}
    \end{split}
\end{equation}

The following lemma provides us with some additional properties of the function $\A$. In the statement,
$\l_{\rm min}(\xi)$ and $\l_{\rm max}(\xi)$ denote the smallest and largest eigenvalue, respectively, of the symmetric matrix $\nabla_\xi \A(\xi)$ given by
$$\nabla_\xi \A(\xi)=
\left(\frac{\partial \A^i(\xi)}{\partial\xi_j}\right)_{i,j=1,\ldots,n}  \quad \text{for $\xi \neq 0$.}
$$

\begin{lemma} \label{lemma_Amonot}
We have that
\begin{equation}\label{monot:A}
    \big(\A(\xi)-\A(\eta)\big)\cdot(\xi-\eta)>0 \quad \text{for  $\xi, \eta \in \rn$, with $\xi\neq \eta$.}
\end{equation}
 Moreover,
\begin{equation}\label{mon:updown}
\l\,\min\{1,i_b\}\,a\big(H(\xi)\big)\,|\eta|^2 \leq \nabla_\xi\A(\xi)\, \eta \cdot \eta
 \leq \L\,\max\{1,s_b\}\,a\big(H(\xi)\big)\,|\eta|^2\,,
\end{equation}
for $\xi\neq 0$ and $\eta\in \rn$. In particular
\begin{equation} \label{lambda_ratio}
    \frac{\l_{\rm max}(\xi)}{\l_{\rm  min}(\xi)}\leq \frac{\L\,\max\{1,s_b\}}{\l\,\min\{1,i_b\}} \quad \text{for  $\xi \neq 0$.}
\end{equation}
\end{lemma}

\begin{proof}
Let $\xi \neq 0$ and $i,j \in \{1,\ldots,n\}$. Computations show that
\begin{equation}\label{daidj}
    \frac{\partial \A^i(\xi)}{\partial\xi_j} = a\big(H(\xi)\big)\bigg\{\Big[1+\frac{H(\xi)\,a'\big(H(\xi)\big)}{a\big(H(\xi)\big)}\Big]\,\partial_{\xi_i}H(\xi)\,\partial_{\xi_j}H(\xi)+H(\xi)\,\partial_{\xi_i\xi_j}H(\xi) \bigg\}.
\end{equation}
From \eqref{expand}, \eqref{c:a} and \eqref{ell:H} we deduce that 
\begin{align}\label{mon:up}
     \sum_{i,j=1}^n  \frac{\partial \A^i(\xi)}{\partial\xi_j}\,\eta_i\eta_j&\geq a\big(H(\xi)\big)\min\{1,1+i_a\}\bigg\{ \partial_{\xi_i}H(\xi)\,\partial_{\xi_j}H(\xi)\,\eta_i\,\eta_j+H(\xi)\,\partial_{\xi_i\xi_j}H(\xi)\,\eta_i\,\eta_j \bigg\} 
       \\ \nonumber
       &=a\big(H(\xi)\big)\min\{1,1+i_a\}\,\tfrac{1}{2}\nabla^2_\xi H^2(\xi)\,\eta\cdot\eta \geq \l\,\min\{1,1+i_a\}\,a\big(H(\xi)\big)\,|\eta|^2 \,,
    \end{align}
  for  $\eta\in\rn$ and $\xi\neq 0$. Hence, the first inequality in 
 \eqref{mon:updown} follows, thanks to equation \eqref{dic105}.
%
%
%
\\
By \eqref{dic105}, the second inequality in \eqref{mon:updown} can be deduced from equations \eqref{daidj}, \eqref{expand} and \eqref{c:a}, which imply that
$$
  \sum_{i,j=1}^n  \frac{\partial \A^i(\xi)}{\partial\xi_j}\eta_i\,\eta_j\leq a\big(H(\xi)\big)\,\max\{1,1+s_a\}\,\tfrac{1}{2}\nabla^2_\xi H^2(\xi)\,\eta\cdot\eta \quad \quad \text{for  $\eta\in\rn$.}
$$
  Equation \eqref{limb} ensures that the function $\A$ is continuous also at $0$. 
Therefore,
\begin{equation*}
\begin{split}
    \big(\A(\xi)-\A(\eta)\big)\cdot(\xi-\eta)=\int_0^1\frac{d}{dt}\A\big(t\,\xi+(1-t)\eta\big)\cdot(\xi-\eta)\,dt=
    \\
    =\int_0^1\frac{\partial \A^i}{\partial \xi_j}\big(t\,\xi+(1-t)\eta\big)\,(\xi-\eta)_i\,(\xi-\eta)_j\,dt \quad\text{for  $\xi, \eta \in \rn$.}
\end{split}
\end{equation*}
Hence, by inequality \eqref{mon:up}, 
\begin{equation*} 
    \big(\A(\xi)-\A(\eta)\big)\cdot(\xi-\eta)\geq \l\min\{1,1+i_a\}\Big(\int_0^1 a\big(H(t\,\xi+(1-t)\eta)\big)\,dt\Big)\,|\xi-\eta|^2>0
\end{equation*}
if $\xi \neq \eta$. This establishes 
inequality  \eqref{monot:A}.
\\ Finally, \eqref{lambda_ratio} is a straightforward consequence of  \eqref{mon:updown}. 
\end{proof}

We shall need to approximate the function $\A$ via a family of functions with quadratic growth. To this purpose, for $\e\in (0,1)$, we define the function  $a_\e(t): [0, \infty) \to [0, \infty)$ as
\begin{equation}\label{a_e}
    a_\e(t)=\frac{{a}(\sqrt{\e+t^2})+\e}{1+\e\,{a}(\sqrt{\e+t^2})} \quad \text{for $t\geq 0$,}
\end{equation}
the function $ b_\e(t): [0, \infty) \to [0, \infty)$ as
\begin{equation}\label{b_e}
    b_\e(t)=a_\e(t) t \quad \text{for $t\geq 0$,}
\end{equation}
and the function $\A_\e(\xi): \rn \to [0, \infty)$ as
\begin{equation}\label{def:Ae'}
    \A_\e(\xi)=\begin{cases} b_\e\big(H(\xi)\big)\nabla_\xi H(\xi) & \quad \text{if $\xi \neq0$}
\\ 0 & \quad \text{if $\xi =0$.}
\end{cases}
\end{equation}
Notice the alternative formula
\begin{equation}\label{def:Ae}
    \A_\e(\xi) = a_\e\big(H(\xi)\big)\tfrac{1}{2}\nabla_\xi H^2(\xi) \quad \text{if $\xi \neq 0$.}
\end{equation}

Some basic properties of these functions as $\e \to 0^+$ are provided by the following lemma.

\begin{lemma}\label{prop:ae}
Let  $\e \in (0,1)$. Then,  $ a_\e\in C^1([0,\infty))$,
\begin{equation}\label{appr:1}
  \e\leq a_\e(t)\leq \e^{-1}\quad\text{for }t\geq 0,
\end{equation}
and
\begin{equation}\label{appr:2}
    \min\{i_a,0\}\leq i_{a_\e}\leq s_{a_\e}\leq \max\{s_a,0\}.
\end{equation}
Moreover, $\A_\e(\xi)\in C^1(\rn\setminus\{0\})$ and, given any  $M>0$,
\begin{equation}\label{cortona10}
     \lim\limits_{\e\to 0^+}b_\e (t)=b(t) 
\end{equation}
uniformly in $[0,M]$, and 
\begin{equation}\label{appr:3}
    \lim\limits_{\e\to 0^+}\A_\e(\xi)=\A(\xi) 
\end{equation}
uniformly in $\{\xi\in\rn\,:\,|\xi|\leq M\}$.
\end{lemma}
\noindent The proof of this lemma is analogous to that of  \cite[Proof of Lemma 4.5]{cia14} and will be omitted.


%

The next result provides us with information about the symmetric matrix $\nabla_\xi A_\e(\xi)$ given by
$$\nabla_\xi A_\e(\xi)=
\left(\frac{\partial \A_\e^i(\xi)}{\partial\xi_j}\right)_{i,j=1,\ldots,n}  \quad \text{for $\xi \neq 0$,}
$$
for $\e \in (0,1)$, and about its smallest and largest eigenvalues
$\l_{\rm min}^\e(\xi)$ and $\l_{\rm max}^\e(\xi)$.

\begin{lemma}\label{ellip:ae}
Let $\e \in (0,1)$. Then,
\begin{equation}\label{pr:ae}
    \l\,\min\{1,i_b\}\, a_\e\big(H(\xi)\big)\,|\eta|^2 \leq \nabla_\xi \A_\e(\xi)\, \eta \cdot \eta
\leq \L\,\max\{1,s_b\}\, a_\e\big(H(\xi)\big)\,|\eta|^2
\end{equation}
for $\xi\neq 0$ and $\eta\in \rn$. 
In particular,  
\begin{equation}\label{quotients:eigen}
    \begin{split}
        \frac{\l_{\rm max}^\e(\xi)}{\l_{\rm  min}^\e(\xi)}\leq \frac{\L\,\max\{1,s_b\}}{\l\,\min\{1,i_b\}} \quad \text{for $\xi \neq 0$,}
    \end{split}
\end{equation}
and
\begin{equation}\label{est:Ae}
\begin{split}
   \e\,\l\,\min\{1,i_{b}\}\,\mathrm{Id} \leq \nabla_\xi \A_\e(\xi)
\leq\e^{-1}\,\L\,\max\{1,s_{b}\}\,\mathrm{Id} \quad \text{for $\xi \neq 0$.}
\end{split}
\end{equation}
Hence, the function $\A_\e : \rn \to [0,\infty)$ is Lipschitz continuous.
\end{lemma}

Lemma \ref{ellip:ae} follows via Lemma \ref{lemma_Amonot}, applied with the function $a$ replaced with $a_\e$, and Lemma \ref{prop:ae}.
 
We conclude this section with an  algebraic inequality which is crucial for our approach.  In what follows, 
given a matrix $M=(M_{ij}) \in \mathbb R^{n\times n}$, we  set $|M|^2=\sum\limits_{i,j=1}^n M_{ij}^2$, the Frobenius norm of $M$.

\begin{lemma}\label{ineqAep} Let $\e \in (0,1)$ and  let $M \in \mathbb R^{n\times n}$ be a symmetric matrix. Then,
\begin{equation}\label{elem:ineqt}
    \mathrm{tr}\big((\nabla_\xi \A_\e(\xi) M)^2\big)\geq \bigg(\frac{\l\,\min\{1,i_b\}}{\L\,\max\{1,s_b\}}\bigg)^2\,\big|\nabla_\xi \A_\e(\xi) M\big|^2\quad \text{for  $\xi \in \mathbb R^n\setminus\{0\}$.}
\end{equation}
\end{lemma}

Lemma \eqref{ineqAep} is a consequence of
inequality  \eqref{quotients:eigen} and of the next lemma.

\begin{lemma}\label{algineq}
    Let $X,Y\in \mathbb R^{n\times n}$. Assume that $Y$ is symmetric and $X$ is symmetric and positive definite. Denote by $\l_{{\rm min}}$ and  $\l_{{\rm  max}}$ the smallest and the largest eigenvalue of $X$, respectively. Then,
    \begin{equation}\label{tel:ineq1}
        \mathrm{tr}\big((X\,Y)^2\big)\geq \bigg(\frac{\l_{\rm min}}{\l_{\rm max}}\bigg)^2\,|X\,Y|^2.
    \end{equation}
%
\end{lemma}

\begin{proof}
The elementary inequality
\begin{equation}\label{june40}
        \l_{{\rm  min}}\,\mathrm{tr}(M)\leq \mathrm{tr}(X M)=\mathrm{tr}(M X)\leq \l_{{\rm  max}}\,\mathrm{tr}(M)\,,
    \end{equation}
     holds for any positive semi-definite matrix $M$. To verify this fact, recall that there exist unitary matrix $U$ and a diagonal matrix whose entries are the eigenvalues of $X$ such that $X=U^T\Lambda U$. Hence,
     $$\mathrm{tr}(X M) = \mathrm{tr}(U^T\Lambda U M)= \mathrm{tr}(\Lambda U M U^T)\geq \l_{{\rm  min}} \mathrm{tr}(U M U^T) = \l_{{\rm  min}} \mathrm{tr}(M).$$
    Note that the inequality holds since the matrix $U M U^T$ is positive semi-definite, and hence all the entries in its diagonal are nonnegative. This establishes the first inequality in \eqref{june40}. The second one follows analogously.
    Thanks to the first inequality in \eqref{june40} we have that
    \begin{equation}\label{june41}
        \mathrm{tr}\big( (XY)^2\big)=\mathrm{tr} (X Y X Y)\geq \l_{{\rm min}}\,\mathrm{tr} (YXY)=\l_{{\rm  min}}\,\mathrm{tr} (X Y^2)\geq \l_{{\rm  min}}^2 \,\mathrm{tr} (Y^2)=\l_{{\rm  min}}^2\,|Y|^2,
    \end{equation}
where we have made use of the fact that, by the very definition, the matrix $YXY$ is  symmetric and positive definite since $Y$ is symmetric and $X$ is symmetric and positive definite. On the other hand, from the second inequality in \eqref{june40} we infer that
    \begin{equation}\label{june42}
            |X Y|^2=\mathrm{tr}\big(X Y (X Y)^t  \big)=\mathrm{tr}(X Y Y X)=\mathrm{tr}(X^2 Y^2)\leq \l_{{\rm  max}}^2 \mathrm{tr}(Y^2)=\l_{{\rm  max}}^2 |Y|^2,
    \end{equation}
    thanks to the fact that 
$X^2$ is symmetric and its eigenvalues agree with the eigenvalues of $X$ squared.
   Inequality \eqref{tel:ineq1} is a consequence of  inequalities \eqref{june41} and \eqref{june42}.
\end{proof}


%
%

\section{Fundamental lemmas for vector fields}\label{sec:aux}

Several 
 pointwise identities and inequalities involving functions and vector fields   are offered in this section. They are critical in the proofs of our regularity estimates.

We begin   with an identity  for  vector fields $V: \Omega \to \rn$.
 If $V=(V^1, \dots , V^n)$, then we set
$$\nabla V = (\partial _jV^i)_{ij}.$$
Hence, $\nabla V \, V$ is the vector whose $i$-th component agrees with $V^j\partial _jV^i$. Here, and in what follows, we adopt the convention about summation over repeated indices.
\begin{lemma}\label{lemma:cruc}
Assume  that   $V: \Om \to \rn$ and  $V\in C^2(\Omega)$.
Then
\begin{equation}\label{cruch1}
    \big(\mathrm{div}\,V\big)^2=\mathrm{tr}\big((\nabla V)^2\big)+\mathrm{div}\Big(V\,\mathrm{div}V- \nabla V \, V \Big) \,.
\end{equation}
\end{lemma}

\begin{proof} Schwarz's theorem on second mixed derivatives and an exchange of the indices $i$ and $j$ ensure that
\begin{equation}\label{eux1}
    (\partial_j\partial_i V^i)\,V^j=(\partial_j\partial_i V^j)\,V^i.
\end{equation}
Notice that
\begin{equation}\label{cont1}
     \mathrm{div}(\nabla V \, V)=\partial_j\Big(V^i\,\partial_i V^j  \Big)=\partial_j V^i\,\partial_i V^j+V^i\,\partial_j\partial_i V^j 
   =\mathrm{tr}((\nabla V)^2)+V^i\,\partial_j\partial_iV^j\,,
\end{equation}
and
\begin{equation}\label{cont2}
       \mathrm{div}\Big(V\, \mathrm{div} V  \Big)=\big( \mathrm{div}V  \big)^2+V\cdot\nabla  (     \mathrm{div} V) =
      \big(\mathrm{div}V  \big)^2+V^i\,\partial_i\partial_j V^j.
\end{equation}
Subtracting equations \eqref{cont2} and \eqref{cont1} and the use \eqref{eux1} yield
\begin{equation*}
    \begin{split}
        \mathrm{div}\Big( V\, \mathrm{div}V - \nabla V\, V \Big)=\big( \mathrm{div}V  \big)^2-\mathrm{tr}\big((\nabla V)^2 \big),
    \end{split}
\end{equation*}
namely equation \eqref{cruch1}. 
\end{proof}

Let $\Omega$ be an open set in $\rn$ such that $\partial \Omega \in C^1$.  We 
denote by $\nu=\nu(x)$ the outward unit normal to $\Omega$ at a point $x\in \partial \Omega$.  Given a vector field $V: \partial \Omega \to \rn$, its tangential component $V_T$ is defined by 
\begin{equation*}
    V_T=V-(V\cdot \nu)\,\nu.
\end{equation*}
The notations $\nabla_T$ and $\mathrm{div}_T$ are adopted for the tangential gradient and divergence operators on $\partial \Omega$. 
Therefore, if $u\in C^1(\overline{\Omega})$, then
\begin{equation*}
    \nabla_T u=\nabla u- (\partial_\nu u)\,\nu \quad \text{on $\partial \Omega$,}
\end{equation*}
where $\partial_\nu u$ denotes the normal derivative of $u$.
Moreover, if $V\in C^1(\overline{\Omega})$,  then
\begin{equation*}
    \mathrm{div}_T V=\mathrm{div}V-(\partial_\nu V\cdot\nu) \quad  \text{on $\partial \Omega$.}
\end{equation*}
As a consequence, we obtain the following lemma.


\begin{lemma}\label{quick:form-new}
Let $\Omega$ be a bounded open set in  $\rn$ such that $\partial \Omega \in C^1$. Assume that  $V: \rn \to \rn$, $V\in C^{0,1}(\rn)$,  and  there exists a closed set $Z$ such that
\begin{equation}\label{gen15}
V(x) = 0 \quad \text{if $x\in Z$,}
\end{equation}
and
\begin{equation}\label{gen16}
V\in C^1 (\rn \setminus Z).
\end{equation}
Then, 
\begin{align}\label{formuletta11}
\int_\Omega\big(\mathrm{div}\,V\big)^2 \,\phi\,dx=&\int_\Omega\mathrm{tr}\big((\nabla\,V)^2\big)\phi\,dx+\int_{\partial\Omega}\Big((\mathrm{div}_T V)\,V\cdot\nu- \nabla_T V\, V_T \cdot \nu  \Big)\,\phi\,d\mathcal{H}^{n-1}
    \\ \nonumber
    &-\int_\Om\Big( (\mathrm{div}\,V)\,V\cdot\nabla\phi- \nabla V\, V \cdot\nabla\phi\Big)dx.
    \end{align}
for  every $\phi\in C^\infty(\rn)$. 
In particular
\begin{equation}\label{formuletta1}
    \int_\Omega\big(\mathrm{div}\,V\big)^2 \,dx=\int_\Omega\mathrm{tr}\big((\nabla\,V)^2\big)dx+\int_{\partial\Omega}\Big((\mathrm{div}_T V)\,V\cdot\nu-  \nabla_T V\, V_T \cdot \nu \Big)d\mathcal{H}^{n-1}.
\end{equation}
In equations \eqref{formuletta11} and \eqref{formuletta1}, the functions  $(\mathrm{div}_T V)\,V$ and  $ \nabla_T V\, V_T$  are defined as $0$ in the set $Z$.
\end{lemma}

\begin{proof}
By multiplying the vector field $V$ by a smooth compactly supported function, whose support contains $\overline \Omega$, we may assume, without loss of generality, that $V$ is compactly supported in $\rn$. Since $V\in \mathrm{Lip}(\rn)$  and  assumption \eqref{gen16} is in force, the vector field $V$ can be approximated, via standard convolutions, by a sequence $\{V_k\}$ of smooth, compactly supported functions in $\rn$, such that
\begin{equation}\label{gen17}
V_k(x) \to V(x) \quad \text{for every $x \in \rn$,}
\end{equation}
\begin{equation}\label{gen18}
\nabla V_k(x) \to \nabla V(x) \quad \text{for every $x \in \rn\setminus Z$,}
\end{equation}
and 
\begin{equation}\label{gen19}
|V_k(x)| \leq c, \quad   |\nabla V_k(x)| \leq c \quad \text{for every $x \in \rn$,}
\end{equation}
for some constant $c$. Assumption \eqref{gen15} and the second inequality in \eqref{gen19} also imply that
\begin{equation}\label{gen20}
\mathrm{div}_T V_k(x)\,V_k(x) \to 0 \quad \text{and}   \quad \nabla_T V_k(x)\, (V_k)_T(x)\to 0  \quad \text{for every $x \in \partial \Omega \cap Z$.}
\end{equation}
Fix $k \in \mathbb N$. Given $\phi\in C^\infty(\rn)$, by \eqref{cruch1} and the divergence theorem we have that
\begin{align}\label{formuletta11g}
\int_\Omega\big(\mathrm{div}\,V_k\big)^2 \,\phi\,dx=&\int_\Omega\mathrm{tr}\big((\nabla\,V_k)^2\big)\phi\,dx+\int_{\partial\Omega}\Big((\mathrm{div}V_k)\, V_k-  \nabla V_k\, V_k  \Big)\cdot \nu  \,\phi\,d\mathcal{H}^{n-1}
    \\ \nonumber
    &-\int_\Om\Big((\mathrm{div}\,V_k)\,V_k\cdot\nabla\phi- \nabla V_k\, V_k\cdot \nabla\phi\Big)\,dx.
    \end{align}
Subtracting  the equations
\begin{equation*}
    (\mathrm{div}\,V_k)\,V_k\cdot\nu=(\mathrm{div}_T\,V_k)\,V_k\cdot\nu+(\partial_\nu V_k\cdot \nu)\,V_k\cdot \nu \,,
\end{equation*}
and
\begin{equation*}
    \nabla V_k\, V_k \cdot\nu = \nabla_T V_k \, (V_k)_T \cdot \nu+ (V_k\cdot\nu)\,\partial_\nu V_k\cdot \nu \,,
\end{equation*}
results in
\begin{equation}\label{gen21}
    (\mathrm{div}\,V_k)\,V_k\cdot \nu-  \nabla V_k\, V_k \cdot\nu =(\mathrm{div}_T V_k)\,V_k\cdot\nu-\nabla_T V_k \, (V_k)_T \cdot \nu \quad\text{on }\partial \Omega \,.
\end{equation}
From equations \eqref{formuletta11g} and \eqref{gen21} one deduces that
\begin{align}\label{gen22}
\int_\Omega\big(\mathrm{div}\,V_k\big)^2 \,\phi\,dx=&\int_\Omega\mathrm{tr}\big((\nabla\,V_k)^2\big)\phi\,dx+\int_{\partial\Omega}\Big((\mathrm{div}_T V_k)\,V_k\cdot\nu-\nabla_T V_k \, (V_k)_T \cdot \nu \Big)\,\phi\,d\mathcal{H}^{n-1}
    \\ \nonumber
    &-\int_\Om\Big((\mathrm{div}\,V_k)\,V_k\cdot\nabla\phi- \nabla V_k\, V_k\cdot \nabla\phi\Big)\,dx.
    \end{align}
Owing to properties \eqref{gen17}--\eqref{gen20}, passing to the limit as $k \to \infty$ in the latter equation yields \eqref{formuletta11}, via the dominated convergence theorem.
\end{proof}

Our next task is a proof of a generalization to the anisotropic setting of the classical Reilly's identity \cite{reilly}. It involves the notion of anisotropic second fundamental form of the boundary of a set $\Omega$.  Recall that
the shape operator (also called Weingarten operator) on  $\partial \Omega$ agrees with $\nabla_T \nu$. Since $(\nabla_T \nu):   \nu^\perp \to  \nu^\perp$,  owing to \eqref{jan10} one also has that
\begin{equation}\label{gen11}
\nabla_\xi^2 H(\nu)\,(\nabla_T \nu) : \nu^\perp \to  \nu^\perp.
\end{equation}
 The \textit{anisotropic second fundamental form} $\mathcal{B}^H$ of $\partial \Omega$ is defined by
\begin{equation*}
\nabla_\xi^2 H(\nu)\,(\nabla_T \nu)\,\eta \cdot \zeta \quad \quad \text{for $\eta\ , \zeta\in \nu^\perp$.}
\end{equation*}
Namely,
\begin{equation} \label{def_BH}
\mathcal{B}^H=\nabla_T \big(\nabla_\xi H(\nu)\big)=\nabla_\xi^2 H(\nu)\,\nabla_T \nu.
\end{equation} 
Furthermore,  the \textit{anisotropic mean curvature} is given by
\begin{equation} \label{def_AMC}
    \mathrm{tr}\,\mathcal{B}^H=\mathrm{div}_T\big(\nabla_\xi H(\nu)\big) \,.
\end{equation}
Clearly, 
when $H$ is the Euclidean norm, $\nabla_\xi^2 H(\nu)= {\rm Id}_{\nu^\perp}$, and hence $\B^H=\B$, the standard second fundamental form on $\partial \Omega$.
\\
The functions $\Psi_\Omega^H : (0, \infty) \to [0, \infty)$ and $\mathcal K _\Omega^H\,: (0, \infty)\to [0, \infty)$ are defined as in \eqref{def:psi} and \eqref{dic7}, with $\B$ replaced by $\B^H$. Namely, 
\begin{equation}\label{def:psiH}
\Psi_\Omega^H(r)=
\begin{cases}
     \sup\limits_{x\in\partial\Omega} || \mathcal B^H||_{L^{n-1,\infty}(\partial\Omega\cap B_r(x))}\quad\text{if }n\geq 3,
\\
\\
  \sup\limits_{x\in\partial\Omega} || \mathcal B^H||_{L^{1,\infty}\log L(\partial\Omega\cap B_r(x))}\quad\text{if }n= 2
    \end{cases}
\end{equation}
for $r >0$,
and 
\begin{equation}\label{dic7H}
\mathcal K_\Omega^H(r) =
\displaystyle 
\sup _{
\begin{tiny}
 \begin{array}{c}{E\subset  B_r(x)} \\
x\in \partial \Om
 \end{array}
  \end{tiny}
} \frac{\int _{\partial \Om \cap E} |\mathcal B^H|d\mathcal H^{n-1}}{{\rm cap} (E, B_r(x))}\qquad \hbox{for $r>0$}\,.
\end{equation}
As a consequence of equation \eqref{def_BH} and Lemma \ref{pr:ellh1}, there exist positive constants $c=c(n,\l,\L)$ and $C=C(n,\l,\L)$  such that
\begin{equation}\label{equiv:BBH}
    c\,|\B|\leq|\B^H|\leq C\,|\B| \qquad \text{on $\partial \Omega$.}
\end{equation}
Hence, 
\begin{equation}\label{equiv:mars}
    c\,\Psi_\Omega (r)\leq \Psi_\Omega^H(r)\leq C\,\Psi_\Omega (r) \quad \text{for $r >0$,}
\end{equation}
and 
\begin{equation}\label{gen12}
    c\,\mathcal K_\Omega(r)\leq \mathcal K_\Omega^H (r) \leq C\,\mathcal K_\Omega (r) \quad \text{for $r >0$.}
\end{equation}
Also, equation \eqref{def_BH} and the first inequality in \eqref{june40} imply that
\begin{equation}\label{feb315}
     \text{if \,\, $\B\geq 0\,\,[>0]$, \,\, then \,\, $\mathrm{tr} (\B^H)\geq 0\,\,[>0]$. }
 \end{equation}




\begin{lemma} \label{prop_identG}
Let $\Omega$ be a bounded open set in $\rn$ such that $\partial \Omega \in C^2$. Assume that
 $h\in C^1(\overline{\Omega})$,   $v\in C^2(\overline{\Omega})$, and $v=0$ on $\partial \Omega$.
Then,
\begin{align}\label{re:temp3}
\mathrm{div}_T\Big(h\,\tfrac{1}{2}\nabla_\xi H^2(\nabla v) \Big) \,h\,\tfrac{1}{2}\nabla_\xi H^2(\nabla v)\cdot \nu
   &-h\, \nabla_T\Big[ h\,\tfrac{1}{2}\nabla_\xi H^2(\nabla v)\Big] \,\Big[\tfrac{1}{2}\nabla_\xi H^2(\nabla v)\Big]_T\cdot \nu
    \\ \nonumber
    & =h^2\,H(\nu)\,H^2(\nabla v)\,\mathrm{tr}\,\mathcal{B}^H
\end{align}
on $\partial\Omega \cap \{\nabla v \neq 0 \}$.
\end{lemma}

\begin{proof}
Throughout this proof, all formulas are understood to hold, without further mentioning, in the set
$ \partial\Omega \cap \{\nabla v \neq 0 \}$. Computations show that
\begin{align}
    \label{sss}
\mathrm{div}_T\Big(h&\,\tfrac{1}{2}\nabla_\xi H^2(\nabla v) \Big) \,h\,\tfrac{1}{2}\nabla_\xi H^2(\nabla v)\cdot \nu
   -h\, \nabla_T\Big[h\,\tfrac{1}{2}\nabla_\xi H^2(\nabla v)\Big] \,\Big[\tfrac{1}{2}\nabla_\xi H^2(\nabla v)\Big]_T\cdot \nu
        \\  \nonumber
    =h^2\bigg\{&\mathrm{div}_T\Big(\tfrac{1}{2}\nabla_\xi H^2(\nabla v)\Big)\,\tfrac{1}{2}\nabla_\xi H^2(\nabla v)\cdot \nu-  \nabla_T\big[\tfrac{1}{2}\nabla_\xi H^2(\nabla v)\big] \,\big[\tfrac{1}{2}\nabla_\xi H^2(\nabla v)\big]_T\cdot \nu  \bigg\} \,.
\end{align}
Hence, equation \eqref{re:temp3} will follow if we show that
  \begin{align}
    \label{ss0}
\mathrm{div}_T\Big(\tfrac{1}{2}\nabla_\xi H^2(\nabla v)\big)\Big)\,\tfrac{1}{2}\nabla_\xi H^2(\nabla v)\cdot \nu       
        &-  \nabla_T\big[\tfrac{1}{2}\nabla_\xi H^2(\nabla v)\big] \,\big[\tfrac{1}{2}\nabla_\xi H^2(\nabla v)\big]_T\cdot \nu
         \\ \nonumber & =H^2(\nabla v)\,H(\nu)\,\mathrm{tr}\,\mathcal{B}^H\,.
        \end{align}
Inasmuch as $v$ vanishes on $\partial \Omega$, one has that
\begin{equation}\label{ntw}
    \nabla v=(\partial_\nu v)\,\nu\quad\text{on }\partial \Omega,
\end{equation}
and, by the homogeneity of $H^2$,  
\begin{equation}\label{ss1}
   \tfrac{1}{2}\nabla_\xi H^2(\nabla v)\cdot   (\partial_\nu v) \nu=  H^2(\nabla v) \,.
\end{equation}
Since $\nabla_\xi H$ is homogeneous of degree zero, equation \eqref{ntw} ensures that
 \begin{align}\label{ss2}
\mathrm{div}_T&\Big(\tfrac{1}{2}\nabla_\xi H^2(\nabla v)\Big)=\mathrm{div}_T\Big(H(\nabla v)\,\nabla_\xi H(\nabla v)\Big)
        \\ \nonumber
        &=H(\nabla v)\,\mathrm{sign}(\partial_\nu v)\,\mathrm{div}_T\big(\nabla_\xi H(\nu)\big)+\nabla_T(H(\nabla v))\,\cdot\nabla_\xi H(\nabla v)
        \\ \nonumber
        &=H(\nabla v)\,\mathrm{sign}(\partial_\nu v)\,\mathrm{tr}\mathcal{B}^H+\nabla_T(H(\nabla v))\,\cdot\big[\nabla_\xi H(\nabla v)\big]_T.
    \end{align}
Notice that in the second equality we have made use of the fact that $ \text{sign}\,\partial_\nu v$ is constant in the sets  $\{\partial_\nu v>0\}$ and $\{\partial_\nu v<0\}$, which are  open on $\partial \Omega$ in the topology induced by $\rn$.
Combining equations \eqref{ss1} and \eqref{ss2} tells us that
 \begin{align}\label{ss3}
   \mathrm{div}_T\Big(\tfrac{1}{2}\nabla_\xi H^2&(\nabla v)\big)\Big)\,\tfrac{1}{2}\nabla_\xi H^2(\nabla v)\cdot \nu \\  \nonumber
        & = \frac{H(\nabla v)}{|\partial_\nu v|}\, H^2(\nabla v)\,\mathrm{tr}\mathcal{B}^H
        +\frac{H^2(\nabla v)}{\partial_\nu v}\nabla_T(H(\nabla v))\,\cdot\big[\nabla_\xi H(\nabla v)\big]_T
        \\ \nonumber
        & =H(\nu)\, H^2(\nabla v)\,\mathrm{tr}\mathcal{B}^H+\frac{H^2(\nabla v)}{\partial_\nu v}\nabla_T(H(\nabla v))\,\cdot\big[\nabla_\xi H(\nabla v)\big]_T.
    \end{align}
 Next, we have that
   \begin{align}\label{ss4}
               \nabla_T & \big[\tfrac{1}{2}\nabla_\xi H^2(\nabla v)\big]\,\big[\tfrac{1}{2}\nabla_\xi H^2(\nabla v)\big]_T\cdot \nu  =H(\nabla v)\,\nabla_T\big(H(\nabla v)\,\nabla_\xi H(\nabla v)\big)\,\big[\nabla_\xi H(\nabla v)\big]_T\cdot \frac{\nabla v}{\partial_\nu v}
            \\ \nonumber
            &=\frac{H^2(\nabla v)}{\partial_\nu v} \left\{ \Big(\nabla^2_\xi H(\nabla v)\,\nabla v\Big)\cdot\Big(\nabla_T\big(\nabla v\big)\,\big[\nabla_\xi H(\nabla v)\big]_T\Big)  +\big[\nabla_\xi H(\nabla v)\big]_T\cdot\nabla_T H(\nabla v) \right\}
            \\ \nonumber
            &=\frac{H^2(\nabla v)}{\partial_\nu v}\big[\nabla_\xi H(\nabla v)\big]_T\cdot\nabla_T H(\nabla v),
        \end{align} 
where the first equality holds owing to equation \eqref{ntw}, the second one to the chain rule and equality \eqref{dic101}, and the last one to equation \eqref{dic100}. Equation  \eqref{ss0} follows 
from \eqref{ss3} and \eqref{ss4}.
\end{proof}

\begin{remark}{\rm 
Notice that identity \eqref{re:temp3} can be extended by continuity also to those points $\overline x \in \partial \Omega$ such that $\nabla v(\overline x) = 0$. Indeed, since the tangential derivatives appearing in \eqref{re:temp3} are bounded in $\overline \Omega \setminus \{ \nabla v = 0 \} $, the two sides of identity \eqref{re:temp3} approach zero when $x$ tends to $\overline x$,  inasmuch as $\nabla_\xi H^2(\nabla v)$ is Lipschitz continuous and vanishes at such points. }\end{remark}

\begin{theorem}[Anisotropic Reilly's identity]\label{reil:main}
Let $\Omega$ be a bounded open set in $\rn$ such that $\partial \Omega \in C^2$. Assume that
 $h\in C^1(\overline{\Omega})$,   $v\in C^2(\overline{\Omega})$, and $v=0$ on $\partial \Omega$. Set
\begin{equation} \label{def_tildeV}
    W = h\,\tfrac{1}{2}\nabla_\xi H^2(\nabla v)\quad \text{in  $\overline \Omega$.}
\end{equation}
Then,
 \begin{align}\label{reilly:formula}
       \int_\Om \big(\mathrm{div}\,W\big)^2\,\phi\,dx=\int_\Om \mathrm{tr}\big( (\nabla W)^2\big)\,\phi\,dx & +\int_{\partial \Omega}h^2\,H(\nu)\,H^2(\nabla v)\,\mathrm{tr}\,\mathcal{B}^H\,\phi\, d\mathcal{H}^{n-1}
       \\ \nonumber
       &-\int_\Om\big\{(\mathrm{div} W)\,W\cdot\nabla \phi- \nabla W \,W\cdot\nabla\phi\big\}\,dx 
    \end{align}
for every $\phi\in C^\infty(\rn)$.
\end{theorem}

Let us mention that  formula \eqref{reilly:formula} was established in \cite[formula 3.11]{del} in the special case when both $h$ and $\mathrm{div}\, W$ are constant.

\begin{proof}[Proof of Theorem \ref{reil:main}]
Our assumptions on the domain $\Om$, and on the functions $h$ and $v$ ensure that they are, in fact, restrictions to $\overline{\Omega}$ of functions defined in the entire $\rn$ and enjoying the same regularity properties and compactly supported in $\rn$.
Therefore, the function $W$ is also defined on all $\rn$ via formula \eqref{def_tildeV}. The fact that $H$ is a norm and $H\in C^2(\rn \setminus \{0\})$ ensures that the hypotheses of Lemma \ref{quick:form-new} are fulfilled with $V= W$, and $Z=\{x\in\rn\,:\,\nabla v(x)=0\}$, hence the thesis follows by Lemmas \ref{quick:form-new}-\ref{prop_identG} .

\end{proof}

\section{Local regularity}\label{sec:local}

This section is devoted to the proof of Theorem \ref{thm:local}.
The definition of generalized local solutions to the equations considered in this theorem
involves the use of spaces of functions whose truncations are weakly differentiable.
For $t>0$,  denote by  $T_{t} : \R \rightarrow \R$  the function
defined as 
%
%
%
\begin{equation*}
T_{t} (s) = \begin{cases}
s & {\rm if}\,\,\, |s|\leq t \\
 t \,{\rm sign}(s) &
{\rm if}\,\,\, |s|>t \,.
\end{cases}
\end{equation*}
Given an open set $\Om$ in $\rn$,   define the space
\begin{equation}\label{503}
\mathcal T^{1,1}_{\rm loc}(\Omega) = \left\{u \, \hbox{is measurable in $\Omega$}: \hbox{$T_{t}(u)
\in W^{1,1}_{\rm loc}(\Omega)$ for every $t >0$} \right\}.
\end{equation}
When $\Om$ is   bounded, the spaces  $\mathcal T^{1,1}(\Omega)$ and    $\mathcal T^{1,1}_0(\Omega)$ are defined accordingly, on
replacing $W^{1,1}_{\rm loc}(\Omega)$ with $W^{1,1}(\Omega)$  and $W_{0}^{1,1}(\Omega)$ in \eqref{503}.
\\ As shown in \cite[Lemma 2.1]{BBGGPV}, 
to each function  $u \in\mathcal T^{1,1}_{\rm loc}(\Omega)$ one can associate
a (unique) measurable function  $Z_u : \Omega \to \rn$ such that
\begin{equation}\label{504}
\nabla \big(T_{t}(u)\big) = \chi _{\{|u|<t\}}Z_u  \qquad \quad
\hbox{a.e. in $\Omega$}
\end{equation}
for every $t>0$. Here $\chi _E$ denotes
the characteristic function of the set $E$.  With abuse of notation, the function $Z_u$ will be simply denoted 
by $\nabla u$ in what follows.

 Assume that  $f \in L^2_{\rm loc}(\Omega)$.
A function $u \in \mathcal T ^{1,1}_{\rm loc}(\Omega)$ is called a  generalized local solution to equation \eqref{eq_anisotr_intro} if $\A(\nabla u)\in L^1_{\rm loc}(\Omega)$, the equation 
\begin{equation}\label{231'}
\int_\Omega \A(\nabla u) \cdot \nabla \varphi \, dx=\int_\Omega f \varphi \,dx\,
\end{equation}
holds for every   $\varphi \in C^\infty _0(\Omega)$, and
there exists a sequence $\{f_k\}\subset  C^\infty_0 (\Omega)$ and a corresponding sequence of local weak solutions $\{u_k\}$ to equation \eqref{eq_anisotr_intro}, with $f$ replaced by $f_k$,
such that 
$f_k \to f$   in $L^2(\Omega')$ for every open set $\Omega' \subset \subset \Omega$,
\begin{equation}\label{approxuloc}
u_k \to u \quad \hbox{and} \quad \nabla u_k \to \nabla u \quad \hbox{a.e. in $\Omega$,}
\end{equation}
and
\begin{equation}\label{approxaloc}
\lim _{k \to \infty} \int_{\Omega'} |\A(\nabla u_k)| \, dx =\int_{\Omega'} |\A(\nabla  u)|\, dx.
\end{equation}
Here, $\nabla u$ stands for the function $Z_u$  satisfying property \eqref{504}.

%
%

\begin{proof}[Proof of Theorem \ref{thm:local}] For simplicity of notation, we shall prove the result with balls  $B_{2R}$ replaced by   $B_{3R}$.  
\\ Assume, for the time being, 
 that $f\in L^\infty(\Om)$. Under this assumption, thanks to \cite[Theorem 5.1]{Kor}, the function $u$ belongs to $L^\infty_{loc}(\Om)$ and, in particular, for any $\Om'\subset\subset \Om$ there exists a constant $c=c(n,i_b,s_b,\l,\L,\Om',\Om, \|f\|_{L^\infty(\Om)})$ such that
\begin{equation}\label{bd:l}
    \|u\|_{L^\infty(\Om')}\leq c\,.
\end{equation}
Thus we may apply \cite[Theorem 1.7]{lieb91} and infer that
\begin{equation}\label{c1:loc}
    u\in C^{1,\theta}(\Om')\,,
\end{equation}
for some $\theta\in (0,1)$ depending on $n,\l,\L,s_b,i_b,\Om',\Om, \|f\|_{L^\infty(\Om)}$.
\\
Next, fix $\e \in (0,1)$, and consider the weak solution $u_\e$ to the Dirichlet problem
\begin{equation}\label{eq:locue}
    \begin{cases}
    -\mathrm{div}\big(\A_\e(\nabla u_\e)  \big)=f  &\quad\text{in } B_{3R}
    \\
    u_\e=u  &\quad\text{on }\partial B_{3R},
    \end{cases}
\end{equation}
where $\A_\e$ is the function defined by \eqref{def:Ae'}. The function $u_\e$ agrees with the unique minimizer of the strictly convex functional $J^H_{\e}$ defined as 
\begin{equation}\label{def:minue}
 J^H_{\e}(w)=  \int_{B_{3R}}B_\e\big(H(\nabla w)\big)\,dx-\int_{B_{3R}}f\,w\,dx\,,
\end{equation}
among all functions $w \in W^{1,2}(B_{3R})$ such that $u=w$ on $\partial B_{3R}$. Here, $B_\e$ is the Young function given by
$$
B_\e(t)=\int_0^t b_\e(s)\,ds \quad \text{for $t \geq 0$.}
$$
%
Owing to  \cite[Lemma 3.2]{acf}, one has that $u_\e\in W^{2,2}_{\rm loc}(B_{3R})$. 
Since $\A_\e(\xi)\in C^{0,1}(\rn) \cap C^1(\rn\setminus\{0\})$,  the chain rule for vector-valued functions \cite{mm} ensures that $\A_\e(\nabla u_\e)\in W^{1,2}_{loc}(B_{3R})$ and \begin{equation*}
\nabla (\A_\e(\nabla u_\e)) = \nabla_\xi \A_\e(\nabla u_\e) \nabla^2 u_\e \quad \text{a.e. in  $B_{3R}$.}
\end{equation*}
Now, fix $R\leq \s<\t\leq 2R$ and let $\vphi$ be a cut-off function such that $\vphi\in C^\infty_0(B_\t)$, $0\leq \vphi\leq 1$,   $\vphi= 1$ in $B_\s$, and 
\begin{equation}\label{nab:phi}
    \vphi = 1\quad\text{on } B_\s\quad\text{and  }|\nabla \vphi|\leq c(n)/(\t-\s).
\end{equation}
Extend the vector field $\A_\e(\nabla u_\e)$ to the whole $\rn$ and, via convolution, consider its regularization $V_{\e,\d}=\A_\e(\nabla u_\e)\ast \varrho_\d$.
Here $\{\varrho_\d\}_{\d>0}$ denotes a family  of standard, radially symmetric mollifiers.  
\\
Standard properties of convolution imply that $V_{\e,\d}\xrightarrow{\d\to 0} \A_\e(\nabla u_\e)$ in $W^{1,2}_{loc}(B_{3R})$.
Thus, an application of equation \eqref{formuletta11} with $V=V_{\e,\d}$, $Z=\emptyset$ and  $\phi=\vphi^2$ yields, after letting $\d\to 0$,
 \begin{align}\label{gen90}
        \int_\Om f^2\,\vphi^2\,dx=&\int_\Omega\mathrm{tr}\big((\nabla (\A_\e(\nabla u_\e)))^2\big)\vphi^2\,dx
        \\ \nonumber
        &-2\,\int_\Om\bigg\{ \mathrm{div}\,\A_\e(\nabla u_\e)\,\A_\e(\nabla u_\e)\cdot\nabla\vphi- \nabla (\A_\e(\nabla u_\e)) \A_\e(\nabla u_\e)\cdot\nabla\vphi\bigg\}\,\vphi\,dx.
    \end{align}
From Young's inequality, we deduce that
 \begin{align}\label{last:intest'}
        \Big|  \int_\Om &2\,\bigg\{\big(\mathrm{div}\A_\e(\nabla u_\e)\big)\,\A_\e(\nabla u_\e)\cdot\nabla\vphi-\nabla (\A_\e(\nabla u_\e)) \A_\e(\nabla u_\e)\cdot\nabla\vphi \big) \bigg\}\,\vphi\,dx\Big|\leq 
        \\ \nonumber
        &\leq \gamma\int_\Om |\nabla (\A_\e(\nabla u_\e))|^2\,\vphi^2\,dx+\frac{c}{\gamma}\int_\Om |\A_\e(\nabla u_\e)|^2\,|\nabla \vphi|^2\,dx
    \end{align}
for some constant $c=c(n)$ and for every $\gamma >0$. Combining \eqref{gen90} and \eqref{last:intest'},  making use of 
 inequality \eqref{elem:ineqt}, and recalling property  \eqref{nab:phi} enable us to deduce that
%
\begin{equation}\label{main:loc1}
    \begin{split}
        \int_{B_\s} \big|\nabla (\A_\e(\nabla u_\e))\big|^2\,dx\leq c\,\int_{B_{2R}} f^2\,dx+\frac{c}{(\t-\s)^2}\,\int_{B_\t\setminus B_\s} |\A_\e(\nabla u_\e)|^2\,dx,
    \end{split}
\end{equation}
where  $c=c(n,i_a,s_a,\l,\L)$. Thanks to a Sobolev type inequality on annuli (see e.g.  \cite[formula (5.4)]{cia}),  one has that
 \begin{align}\label{tau:sigl12}
    \frac{1}{(\t-\s)^2}\int_{B_\t\setminus B_\s} |\A_\e(\nabla u_\e)|^2\,dx\leq c\,\int_{B_\t\setminus B_\s}& \big|\nabla (\A_\e(\nabla u_\e))\big|^2\,dx 
    \\ \nonumber
    &+\frac{c\,R}{(\t-\s)^{n+3}}\Big(\int_{B_\t\setminus B_\s}|\A_\e(\nabla u_\e)|\,dx \Big)^2
\end{align}
for some constant $c=c(n)$.
Coupling inequality \eqref{main:loc1} with \eqref{tau:sigl12} tells us that
 \begin{align*}
        \int_{B_\s} \big|\nabla (\A_\e(\nabla u_\e))\big|^2\,dx\leq c\,\int_{B_\t\setminus B_\s} \big|\nabla (\A_\e(\nabla u_\e))\big|^2\,dx&+ c\,\int_{B_{2R}} f^2\,dx
    \\ \nonumber
    &+\frac{c\,R}{(\t-\s)^{n+3}}\Big(\int_{B_\t\setminus B_\s}|\A_\e(\nabla u_\e)|\,dx \Big)^2
    \end{align*}
for some constant $c=c(n,i_a,s_a,\l,\L)$. After  adding the quantity $c\,\int_{B_\s} \big|\nabla \A_\e(\nabla u_\e)\big|^2\,dx$ to both sides of this inequality one infers that
 \begin{align}\label{in:iter}
        \int_{B_\s} \big|\nabla (\A_\e(\nabla u_\e))\big|^2\,dx\leq \,& \frac{c}{1+c}\,\int_{B_\t} \big|\nabla (\A_\e(\nabla u_\e))\big|^2\,dx+ c'\,\int_{B_{2R}} f^2\,dx
     \\ \nonumber
       &+\frac{c'\,R}{(\t-\s)^{n+3}}\Big(\int_{B_{2 R}\setminus B_R}|\A_\e(\nabla u_\e)|\,dx \Big)^2.
    \end{align}
for some constant $c'=c'(n,i_a,s_a,\l,\L)$.
A standard iteration argument (see e.g. \cite[Lemma 3.1, Chapter 5]{giaq}) enables us to deduce from inequality \eqref{in:iter} that
\begin{equation}\label{main:loc2}
    \begin{split}
        \int_{B_R}\big|\nabla (\A_\e(\nabla u_\e))\big|^2\,dx\leq c\,\int_{B_{2R}} f^2\,dx+\frac{c}{R^{n+2}}\Big(\int_{B_{2 R}}|\A_\e(\nabla u_\e)|\,dx \Big)^2,
    \end{split}
\end{equation}
for some constant $c=c(n,i_a,s_a,\l,\L)$. Moreover, a Poincar\'e type inequality implies that
\begin{equation*}
\begin{split}
    \int_{B_R}|\A_\e(\nabla u_\e)|^2\,dx\leq  c R^2\,\int_{B_R}\big|\nabla( \A_\e(\nabla u_\e))\big|^2\,dx+  \frac{c}{R^{n}}\Big(\int_{B_{R}}|\A_\e(\nabla u_\e)|\,dx \Big)^2
\end{split}
\end{equation*}
for some constant $c=c(n)$. Hence, via inequality \eqref{main:loc2}, we obtain that
%
%
\begin{equation}\label{main:loc3}
    \begin{split}
        \int_{B_R}|\A_\e(\nabla u_\e)|^2\,dx\leq c\,R^2\,\int_{B_{2R}} f^2\,dx+\frac{c}{R^{n}}\Big(\int_{B_{2 R}}|\A_\e(\nabla u_\e)|\,dx\Big)^2
    \end{split}
\end{equation}
for some constant 
$c=c(n,i_a,s_a,\l,\L)$ .
 \\ From
\cite[Theorem 2]{tal} one can deduce that
\begin{equation}\label{bd:l1}
    \|u_\e\|_{L^\infty(B_{3R})}\leq \|u\|_{L^\infty(B_{3R})}+ c\,R\,\hat{b}_\e^{-1}\Big(c\|f\|_{L^n(B_{3R})}\Big)\,,
\end{equation}
  where $\hat{b}_\e$ is the function defined by  $\hat{b}_\e(t) = B_\e(t)/t$. Hence, thanks to equations \eqref{Bb} and \eqref{dic105}, applied with $b$ replaced by $b_\e$, and formulas  \eqref{appr:2} and \eqref{bd:l}, one can deduce that
\begin{equation}\label{gl:bde}
    \|u_\e\|_{L^\infty(B_{3R})}\leq c
\end{equation}
for some constant $c$ independent of $\e$.
%
%
%
%
\\
This enables us to  apply \cite[Theorem 1.7]{lieb91} and obtain that 
\begin{equation}\label{unif:c1loc}
    \|u_\e\|_{C^{1,\theta}(B')}\leq c,
\end{equation}
for some constant $c$ independent of $\e \in (0,1)$ and for every ball $B'\subset\subset B_{3R}$. 
Hence, there exist a function $v\in C^1(B_{3R})$ and a sequence $\{\e_k\}$ such that $\e_k \to 0^+$ and
\begin{equation}\label{conv:eloc}
    u_{\e_k}\to v\quad\text{in $C^{1,\theta'}_{\rm loc}(B_{3R})$}
\end{equation}
for every $\theta'<\theta$,
In particular, this convergence and inequality \eqref{gl:bde} imply that $v\in L^\infty(B_{3R})$.
\\
Moreover, by equation \eqref{appr:3} and \eqref{unif:c1loc}, the  norms $\|\A_{\e_k}(\nabla u_{\e_k})\|_{L^\infty (B_{2R})}$ are  uniformly bounded for $k\in \mathbb N$. This piece of  information, coupled  with inequalities \eqref{main:loc2} and \eqref{main:loc3},  entails  that the sequence $\{\A_{\e_k}(\nabla u_{\e_k})\}$ is uniformly bounded in $W^{1,2}(B_{2R})$.
As a consequence, there exists a subsequence, still denoted by $\{u_{\e_k}\}$, such that
\begin{equation}\label{conv:Aeloc}
    \A_{\e_k}(\nabla u_{\e_k})\rightharpoonup \A(\nabla v)\quad\text{weakly in }W^{1,2}(B_{2R})\,.
\end{equation}
Hence,  from inequalities \eqref{main:loc2} and \eqref{main:loc3} we infer that
\begin{equation}\label{main:loc3v}
    \begin{split}
        \int_{B_R}\big|\nabla (\A(\nabla v))\big|^2\,dx\leq c\,\int_{B_{2R}} f^2\,dx+\frac{c}{R^{n+2}}\Big(\int_{B_{2 R}}|\A(\nabla v)|\,dx\Big)^2,
    \end{split}
\end{equation}
and
\begin{equation}\label{main:loc2v}
    \begin{split}
        \int_{B_R}|\A(\nabla v)|^2\,dx\leq c\,R^2\,\int_{B_{2R}} f^2\,dx+\frac{c}{R^{n}}\Big(\int_{B_{2 R}}|\A(\nabla v)|\,dx\Big)^2.
    \end{split}
\end{equation}
Also, passing to the limit in the weak formulation of problem \eqref{eq:locue} tells us that
\begin{equation}\label{quasi}
    \int_{B_{3R}}\A(\nabla v)\cdot \nabla \vphi\,dx=\int_{B_{3R}}f\,\vphi\,dx 
\end{equation}
for every function $\vphi\in C^\infty_0(B_{3R})$.
\\ We claim that 
\begin{equation}\label{cortona1}
    v-u\in W_0^{1,B}(B_{3R}).
\end{equation}
%
%
To verify this claim notice that, thanks to \eqref{c1:loc}, we can exploit the minimizing property of the function $u_\e$, which tells us that is $J^H_\e(u_\e)\leq J^H_\e(u)$. Coupling this piece of information with  \eqref{bounds:H}, and properties \eqref{Bdelta2} and \eqref{Btildedelta2} for $B$ and $B_\e$ ensures that 
\begin{equation}\label{cortona2}
    \int_{B_{3R}}B_\e\big( |\nabla u_\e|\big)\,dx\leq c\,\int_{B_{3R}}B_\e\big( |\nabla u|\big)\,dx+c\,\int_{B_{3R}}f\,(u_\e-u)\,dx\,,
\end{equation}
for some  positive constant $c=c(n,i_{b},s_{b},\l,\L)$. In particular, such a constant is independent of $\e$, thanks to inequalities \eqref{appr:2}, inasmuch as it depends on $\e$ only through a lower bound on $i_{B_\e}$ and an upper bound on $s_{B_\e}$. 
Owing to bounds \eqref{bd:l} and \eqref{gl:bde}, the    inequality \eqref{cortona2} yields:
\begin{equation}\label{minimality}
    \int_{B_{3R}} B_\e(|\nabla u_\e|)\,dx\leq c\,\int_{B_{3R}} B_\e(|\nabla u|)\,dx+c
\end{equation}
for some constant $c$ independent of $\e$.
From \eqref{Bb}, \eqref{usual:b} for the function $b_\e$, and  \eqref{appr:2}, there exist  positive constants $c,c'$ such that
\begin{equation}\label{ulb}
    c\,t^{\min\{i_b+1,2\}}\leq B_\e(t)\leq c'\,t^{\max\{s_b+1,2\}} \quad \text{for $t\geq 1$.}
\end{equation}
\\ Since $B_\e \to B$ locally uniformly in $[0, \infty)$, property \eqref{conv:eloc} implies  that $B_{\e_k}(|\nabla u_{\e_k}|)\to B(|\nabla v|)$ everywhere in $B_{3R}$.
 Thus, from \eqref{c1:loc}, \eqref{minimality} and \eqref{ulb} we deduce, via  Fatou's Lemma, that  
\begin{equation}\label{temp:loc1}
\begin{split}
    \int_{B_{3R}} B(|\nabla v|)\,dx & \leq c \liminf_{k\to \infty} \int_{B_{3R}} B_{\e_k}(|\nabla u|)\,dx+ c
    \\
    &\leq c|B_{3R}|\Big(  \|\nabla u\|_{L^\infty(B_{3R})}+1\Big)^{\max\{s_b+1,2\}}+c\,.
    \end{split}
\end{equation}
Hence, $v\in W^{1,B}(B_{3R})$.
On the other hand, from \eqref{minimality}, \eqref{c1:loc}  and \eqref{ulb} we infer that
 \begin{align*}
    \int_{B_{3R}} |\nabla u_\e|^{\min\{i_b+1,2\}}\,dx & \leq c\int_{B_{3R}}B_{\e}(|\nabla u|)\,dx+c
    \\ \nonumber
    &\leq c|B_{3R}|\Big(\|\nabla u\|_{L^\infty(B_{3R})}+1\Big)^{\max\{s_b+1,2\}}+c.
\end{align*}
The latter bound
implies that the family of functions $\{u_\e-u\}$ is uniformly bounded in the Sobolev space $W_0^{1,\min\{i_b+1,2\}}(B_{3R})$ for $\e \in (0,1)$.
The reflexivity of this space implies that $v-u\in W_0^{1,{\min\{i_b+1,2\}}}(B_{3R})$, Combining this membership  with inequality \eqref{temp:loc1} yields \eqref{cortona1}.
Our claim is thus proved.
\\  Thanks to \eqref{cortona1}, \eqref{bounds:H},  \eqref{feb1} and  \eqref{A:bds_sopra}, we have that
the vector field $\A(\nabla v)\in L^{\widetilde{B}}(B_{3R})$. 
The density of the space $C^\infty_0(B_{3R})$ in the  space  $W^{1,B}_0(B_{3R})$ and H\"older's inequality in Orlicz spaces \cite[Theorem 4.7.5]{funcsp}  ensure that equation \eqref{quasi} holds, in fact, for every function $\varphi \in  W^{1,B}_0(B_{3R})$. Hence, $v$ is a weak solution to the problem
\begin{equation}\label{eq:locuv}
    \begin{cases}
    -\mathrm{div}\big(\A(\nabla v)  \big)=f & \quad\text{in } B_{3R}
    \\
    v=u  &\quad\text{on }\partial B_{3R}.
    \end{cases}
\end{equation}
The uniqueness of this solution implies that $v=u$. Therefore, inequalities \eqref{main:loc3v} and \eqref{main:loc2v} read
\begin{equation}\label{main:loc3u}
    \begin{split}
        \int_{B_R}|\nabla (\A(\nabla u))|^2\,dx\leq c\,\int_{B_{2R}} f^2\,dx+\frac{c}{R^{n+2}}\Big(\int_{B_{2 R}}|\A(\nabla u)|\,dx\Big)^2,
    \end{split}
\end{equation}
and
\begin{equation}\label{main:loc2u}
    \begin{split}
        \int_{B_R}|\A(\nabla u)|^2\,dx\leq c\,R^2\,\int_{B_{2R}} f^2\,dx+\frac{c}{R^{n}}\Big(\int_{B_{2 R}}|\A(\nabla u)|\,dx\Big)^2.
    \end{split}
\end{equation}
\\ We conclude the proof by removing the assumption that $f \in L^\infty(\Omega)$.  Let  $f\in L^2_{\rm loc}(\Om)$, and let $\{f_k\}$ and  $\{u_k\}$ be sequences as in the definition of local approximable solution $u$ to equation \eqref{eq_anisotr_intro}. Hence, \eqref{approxuloc} and \eqref{approxaloc} hold.

Inequalities 
\eqref{main:loc3u} and \eqref{main:loc2u}, applied with $u$ replaced by $u_k$, and equation  \eqref{approxaloc}
tell us that
\begin{equation}\label{main:loc3k}
    \begin{split}
        \int_{B_R}|\nabla (\A(\nabla u_k))|^2\,dx\leq c\,\int_{B_{2R}} f_k^2\,dx+\frac{c}{R^{n+2}}\Big(\int_{B_{2 R}}|\A(\nabla u_k)|\,dx\Big)^2,
    \end{split}
\end{equation}
and
\begin{equation}\label{main:loc2k}
    \begin{split}
        \int_{B_R}|\A(\nabla u_k)|^2\,dx\leq c\,R^2\,\int_{B_{2R}} f_k^2\,dx+\frac{c}{R^{n}}\Big(\int_{B_{2 R}}|\A(\nabla u_k)|\,dx\Big)^2
    \end{split}
\end{equation}
for $k \in \mathbb N$.
%
%
%
Therefore, the sequence $\{\A(\nabla u_k)\}$ is   bounded in $W^{1,2}(B_{2R})$. As a consequence, there exist a function $U: B_{2R} \to \mathbb R^n$, such that $U \in W^{1,2}(B_{2R})$, and a subsequence of $\{\A(\nabla u_k)\}$, still indexed by  $k$, such that
\begin{equation}\label{gen:conv3}
    \A(\nabla u_k)\to U\quad\text{in }L^2(B_{2R})\quad\text{and }\A(\nabla u_k)\rightharpoonup U\quad\text{weakly in }W^{1,2}(B_{2R}).
\end{equation}
By \eqref{approxuloc},  we thus deduce that $U=\A(\nabla u)$ a.e. in $B_{2R}$. Hence, property \eqref{gen30} holds, and inequalities  \eqref{loc:est}    follow on passing to the limit in \eqref{main:loc3k} and \eqref{main:loc2k}.
%
\end{proof}

\section{Trace inequalities in Lipschitz domains}\label{S:trace}


Here we are concerned with weighted Poincar\'e trace inequalities on Lipschitz domains. Their validity is characterized in terms of weighted isocapacitary inequalities and, as a consequence, of integrability properties of the weight function. The focus of our discussion is on the explicit dependence of the constants in the relevant inequalities on both the weight and the Lipschitz characteristic of the domains. In our applications to the proofs of the global estimates, the weight in the relevant inequalities agrees with the norm of the curvatures on the boundary. 

We begin by specifying the definition of  Lipschitz domain and of Lipschitz characteristic.

\begin{definition}[Lipschitz characteristic of a  domain]\label{Lipdomain} 
 {\rm 
An open set $\Omega$ in $\rn$ is called a Lipschitz domain  if 
 there exist  constants $L_\Om>0$ and $R_\Om \in (0, 1)$ 
such that, for every $x_0\in \partial \Om$ and $R\in (0, R_\Om]$ there exist an orthogonal coordinate system centered at $0\in\rn$ and an $L_\Om$-Lipschitz continuous function 
$\varphi : B'_{R}\to (-\ell, \ell)$, where $B'_{R}$ denotes the ball in $\mathbb R^{n-1}$, centered at $0'\in \R^{n-1}$ and with radius $R$, and
\begin{equation}\label{ell}
\ell = R (1+L_\Om),
\end{equation}
satisfying
\begin{equation}\label{may100}
\begin{split}
    &\partial \Om \cap \big(B'_{R}\times (-\ell,\ell)\big)=\{(x', \varphi (x'))\,:\,x'\in B'_{R}\},
    \\
    & \Om \cap \big(B'_{R}\times (-\ell,\ell)\big)=\{(x',x_n)\,:\,x'\in B'_{R}\,,\,\varphi (x')<x_n<\ell\}.
\end{split}
\end{equation}
Moreover, we set
\begin{equation}\label{may101}
\mathfrak L_\Om = (L_\Om, R_\Om),
\end{equation}
and call $\mathfrak L_\Om$ a Lipschitz characteristic of $\Om$.}
\end{definition}




\begin{remark*}
    Generally speaking, a  Lipschitz characteristic $\mathfrak{L}_\Om = (L_\Om, R_\Om)$ is not   uniquely determined. For instance, if $\partial \Omega \in C^1$, then  $L_\Om$ may be taken arbitrarily small, provided that  $R_\Om$ is chosen sufficiently small. 
\end{remark*}

Let  $\Omega$ be  a Lipschitz domain with Lipschitz characteristic $\mathfrak L_\Omega= (L_\Omega, R_\Omega)$,  and let  $\varrho \in L^1 (\partial \Om)$ be
 a nonnegative function. 
We set, for $r\in (0,R_\Om]$,
\begin{equation}\label{kapparho}
    \mathcal K_{\Om,\varrho}(r) =
\displaystyle 
\sup _{
\begin{tiny}
 \begin{array}{c}{E\subset   B_r(x)} \\
x\in \partial \Omega
 \end{array}
  \end{tiny}
} \frac{\int _{ \partial \Omega \cap E} \varrho \,d\mathcal H^{n-1}}{{\rm cap} (B_r(x), E)},
\end{equation}
and
\begin{equation}\label{Psirho}
\Psi_{\Om,\varrho}(r)=
\begin{cases}
     \sup\limits_{x\in\partial\Omega} \| \varrho\|_{L^{n-1,\infty}(\partial\Omega\cap B_r(x))}\quad\text{if }n\geq 3
\\
\\
  \sup\limits_{x\in\partial\Omega} \| \varrho\|_{L^{1,\infty}\log L(\partial\Omega\cap B_r(x))}\quad\text{if }n= 2.
    \end{cases}
\end{equation}

A bound for the constant in a trace inequality in terms of the quantity $\mathcal K_{\Om,\varrho}(r)$ is provided by the following result.

\begin{proposition}\label{prop:isocapgen}
    Let $\Om$ be a bounded Lipschitz domain in $\rn$,  with Lipschitz characteristic $\mathfrak L_\Omega= (L_\Omega, R_\Omega)$. Assume that   $\varrho$ is a nonnegative function on $\partial \Omega$ such that $\varrho\in L^1(\partial \Om)$. Then, 
    \begin{equation}\label{giu1}
        \int_{\partial \Om\cap B_R(x_0)}v^2\,\varrho\, d\H^{n-1}\leq 32\, (1+L_\Om)^4\,\displaystyle 
\K_{\Om,\varrho}(R)\, \,\int_{\Om\cap B_R(x_0)}|\nabla v|^2\,dx\,,
    \end{equation}
for every $x_0\in \partial \Om$, for every $R\in (0,R_\Om]$, and for every $v\in W^{1,2}_0(B_R(x_0))$.
\end{proposition}

%
%
%

\begin{proof} Under the notations of Definition \ref{Lipdomain}, observe that the function 
$$
        \psi:  B'_R\times (-\ell, \ell)\to \psi\big(B'_R\times (-\ell,\ell)\big)
$$ given by
\begin{equation}\label{diffeom}
         \psi(x',x_n)=(x',x_n-\vphi(x')) \quad \text{for $(x',x_n)\in B'_R\times (-\ell, \ell)$,}
\end{equation}
defines a bi-Lipschitz diffeomorphism, whose  
 inverse 
\begin{equation*}
    \psi^{-1} : \psi\big(B'_R\times (-\ell, \ell)\big)\to B'_R\times (-\ell,\ell)
\end{equation*}
obeys:
$$ \psi^{-1}(y',y_n)=(y',y_n+\vphi(y'))\quad \text{for $(y',y_n) \in \big(B'_R\times (-\ell, \ell)\big)$.}
$$
Notice  also that $\psi\big(B'_R\times (-\ell,\ell)\big)\subset B'_R\times (-2\ell,2\ell)$, and
\begin{equation}\label{june221}
    \|\nabla \psi\|_{L^{\infty}}+\|\nabla \psi^{-1}\|_{L^\infty}\leq c\,(1+L_\Om)
\end{equation}
for some constant $c=c(n)$.
\\  In what follows, with a slight abuse of notation, we shall identify $\mathbb R^{n-1} \times \{0\}$ with $\mathbb R^{n-1}$, and subsets of the former set with subsets of the latter.
\\ 
 We may assume, without loss of generality, that $x_0=0$.
 Fix a compact set $F\subset \psi\big(B_R\big)$. Hence, the set $E= \psi^{-1}(F)$ is a compact subset of $B_R$.  Moreover,
    \begin{align}\label{cappp2}
        \int_{\partial \Om \cap E}\varrho\,d\H^{n-1} & =\int_{\psi(\partial \Om \cap E)} (\varrho\circ \psi^{-1})(y',0)\sqrt{1+|\nabla \vphi(y')|^2}\,dy'
        \\ \nonumber
        &=\int_{F\cap B'_R} (\varrho\circ \psi^{-1})(y',0)\sqrt{1+|\nabla \vphi(y')|^2}\,dy'\,.
        \end{align}
We claim  that
\begin{equation}\label{capacity1}
    \mathrm{cap}\big(E,B_R\big)\leq 2\,(1+L_\Om)^2\,\mathrm{cap}\big(F,\psi(B_R) \big)\,.
\end{equation}
    To prove  this claim, fix a function $\theta \in C^{0,1}_0\big(\psi(B_R)\big)$ such that $\theta\geq 1$ in $F$. Therefore, $\theta\circ \psi\in C^{0,1}_0(B_R)$, $\theta\circ \psi\geq 1$ in $E$, and
 \begin{align*}
           \mathrm{cap}\big(E,B_R\big)\leq \int_{B_R}|\nabla (\theta\circ \psi)|^2\,dx
           &\leq 2\,(1+L_\Om)^2\,\int_{B_R}\big|\nabla \theta(\psi(x))\big|^2\,dx
           \\
           &= 2\,(1+L_\Om)^2\,\int_{\psi(B_R)}|\nabla \theta|^2\,dy.
        \end{align*}
  Inequality \eqref{capacity1} hence follows by taking the infimum over $\theta$.
Combining  \eqref{cappp2} and \eqref{capacity1} entails
\begin{equation}\label{cappp3}
\frac{\int_{F\cap B'_R}(\varrho\circ \psi^{-1})(y',0)\sqrt{1+|\nabla \vphi(y')|^2}dy'}{\mathrm{cap}(F,\psi(B_R))}\leq 2\,(1+L_\Om)^2\frac{\int_{E\cap \partial \Om}\varrho\,d\H^{n-1}}{\mathrm{cap}(E,B_R)}
\end{equation}
It suffices to prove inequality \eqref{giu1} for functions $v\in C^{0,1}_0(B_R)$, the general case following via a standard density argument.
Since the function $v\circ \psi^{-1}\in C^{0,1}_0\big(\psi\big(B'_R\times(-\ell,\ell)\big)\big)\subset C^{0,1}_0\big(B'_R\times(-2\ell,2\ell)\big)$, it can be extended to a function $w: \rn \to \mathbb R$ as   
    \begin{equation}\label{def:evenext}
        w(y',y_n) =\begin{cases}
            v\circ \psi^{-1}(y',y_n)\quad & \text{if $y_n\geq 0$}
            \\
            v\circ \psi^{-1}(y',-y_n)\quad & \text{if $y_n<0$.}
        \end{cases}
    \end{equation}
Observe that
    \begin{equation}\label{even:norm}
        \|w\|^2_{L^2(\rn)}=2\,\|v\circ \psi^{-1}\|^2_{L^2(\rn_+)}\quad\text{and }\quad \|\nabla w\|^2_{L^2(\rn)}=2\,\|\nabla (v\circ \psi^{-1})\|^2_{L^2(\rn_+)}\,,
    \end{equation}
    where we have set $\rn_+=\{(x',x_n)\in\rn\,:\,x_n\geq 0\}$. Denote by ${\rm Tr}$ the trace operator on  $\R^{n-1}$, which, according to the convention above, it is identified with $\partial \rn_+$. Thus,
    \begin{equation*}
        {\rm Tr}(w)= {\rm Tr}(v\circ \psi^{-1})\quad\text{on }\R^{n-1}\,,
    \end{equation*}
    and $\mathrm{supp}\big({\rm Tr}(w)\big)\subset B'_R$.
An application of a Poincar\'e type trace inequality \cite[Theorem 2.4.1]{maz} tells us that
 \begin{align*}
            \int_{\partial \Om\cap B_R}v^2\varrho\,d\H^{n-1}& =\int_{B'_R}|{\rm Tr}(w)(y')|^2(\varrho\circ \psi^{-1})(y',0)\sqrt{1+|\nabla \vphi(y')|^2}dy'
            \\
            &\leq 4\bigg( \sup_{
\begin{tiny}
 \begin{array}{c}{F\subset  \psi(B_R)} \\
F \text{ compact}
 \end{array}
  \end{tiny}
}\frac{\int_{F\cap B'_R}(\varrho\circ\psi^{-1})(y',0)\sqrt{1+|\nabla \vphi(y')|^2}dy'}{\mathrm{cap}(F,\psi(B_R))}\bigg)\,\int_{\rn}|\nabla w|^2 dy.
        \end{align*}
Combining the latter inequality with \eqref{cappp3} and \eqref{even:norm} yields:
 \begin{align*}
    \int_{\partial \Om\cap B_R}v^2\varrho\,d\H^{n-1} & \leq 16\,(1+L_\Om)^2\bigg( \sup_{
\begin{tiny}
 \begin{array}{c}{E\subset  B_R} \\
E \text{ compact}
 \end{array}
  \end{tiny}
}\frac{\int_{\partial \Om \cap E }\varrho\,d\H^{n-1}}{\mathrm{cap}(E,B_R)}\bigg)\int_{ \{y_n\geq 0\}\cap \psi(B_R)}|\nabla (v\circ \psi^{-1})(y)|^2\,dy
\\
&\leq 32\,(1+L_\Om)^4\,\bigg( \sup_{
\begin{tiny}
 \begin{array}{c}{E\subset  B_R} \\
E \text{ compact}
 \end{array}
  \end{tiny}
}\frac{\int_{\partial \Om \cap E }\varrho\,d\H^{n-1}}{\mathrm{cap}(E,B_R)}\bigg)\,\int_{\Om\cap B_R}|\nabla v|^2\,dx.
\end{align*}
Hence, inequality \eqref{giu1} follows.
\end{proof}

Proposition \ref{prop:isocapgen} enables one to deduce a parallel result, where the role of the quantity  $\mathcal K_{\Om,\varrho}(r)$ is instead played by  $\Psi_{\Om,\varrho}(r)$.

\begin{proposition}\label{prop:isocapmar1}
    Let $\Om$ be a bounded Lipschitz domain in $\rn$, $n \geq 2$, with Lipschitz characteristic $\mathfrak L_\Omega= (L_\Omega, R_\Omega)$.   Assume that   $\varrho$ is a nonnegative function on $\partial \Omega$ such that $\varrho\in L^1(\partial \Om)$. Then, 
    \begin{equation}\label{giu30}
        \int_{\partial \Om\cap B_R(x_0)}v^2\,\varrho\, d\H^{n-1}\leq \begin{cases}
        c\, (1+L_\Om)^{8+\frac{n-2}{n-1}}\,\displaystyle 
\Psi_{\Om,\varrho}(R)\, \,\int_{\Om\cap B_R(x_0)}|\nabla v|^2\,dx & \quad \text{for  $n\geq 3$}
\\
\\
c\, (1+L_\Om)^{11}\,\displaystyle 
\Psi_{\Om,\varrho}(R)\, \,\int_{\Om\cap B_R(x_0)}|\nabla v|^2\,dx & \quad \text{for  $n=2$}
\end{cases}
    \end{equation}
for some constant $c=c(n)$, for every $x_0\in \partial \Om$, for every $R\in (0,R_\Om]$, and for every $v\in W^{1,2}_0(B_R(x_0))$.
\end{proposition}

The derivation of Proposition \ref{prop:isocapmar1} from Proposition \ref{prop:isocapgen} relies upon some intermediate steps contained in the following lemmas.


In particular, the following inequalities for the fractional Sobolev space $W^{\frac{1}{2},2}(\R^{n-1})$  come into play.
In what follows,   $\|\,\cdot \,\|_{\exp L^2(\R)}$ denotes the  Luxemburg norm associated with the Young function $A(t)=e^{t^2}-1$.
\begin{lemmaalph}
[Fractional Sobolev-type embedding]\label{fractional} Let $n \geq 2$.
\\ (i) Assume that $n\geq 3$ and set $q=2\frac{(n-1)}{(n-2)}$. Then, there exists a constant $c_s=c_s(n)$ such that
    \begin{equation}\label{sob:imm1}
        \|v\|_{L^q(\R^{n-1})}\leq c_s\,\|v\|_{W^{\frac{1}{2},2}(\R^{n-1})}
    \end{equation}
for  every $v\in W^{\frac{1}{2},2}(\R^{n-1})$.
    \\ (ii) Assume that $n=2$.  Then, there exists an absolute constant $c_s$ such that
    \begin{equation}\label{sob:imm2}
        \|v\|_{\exp L^2(\R)}\leq c_s \|v\|_{W^{\frac{1}{2},2}(\R)}
    \end{equation}
for  every $v\in W^{\frac{1}{2},2}(\R)$ such that $\mathrm{supp }(v)\subset (-1,1)$.
\end{lemmaalph}

\begin{lemmaalph}[Trace embedding]\label{fractional1}
   Let  $n\geq 2$. Then, there exists an absolute constant $c$
such that   \begin{equation}\label{tr:gamma0}
       \|{\rm Tr}(v)\|_{W^{\frac{1}{2},2}(\R^{n-1})}\leq c\,\|v\|_{W^{1,2}(\rn)}\,.
   \end{equation}
   for every  $v\in W^{1,2}(\rn)$. 
\end{lemmaalph}

 Part (i) of the Lemma \ref{fractional} and Lemma \ref{fractional1} are standard. Part (ii) of Lemma \ref{fractional} follows as a special case of \cite{ACPS}.

\begin{lemma}
Let $\Om$ be a bounded Lipschitz domain in $\rn$, with Lipschitz characteristic  $\mathfrak L_\Om = (L_\Om, R_\Om)$. Let $x_0\in \partial \Om$ and $R\in (0,R_\Om]$.
    \\ (i) Assume that $n\geq 3$. Then,
    \begin{equation}\label{mag1}
    \bigg(\int_{\partial \Om \cap E}|v|\,d\H^{n-1}\bigg)^2\leq c\,(1+L_\Om)^{\frac{3n-4}{n-1}}(1+\ell^2)\,\H^{n-1}(\partial \Om \cap E)^{\frac{n}{(n-1)}}
\int_{\Om \cap B_R(x_0)}|\nabla v|^2 dx
%
    \end{equation}
for some constant $c=c(n)$, for
every  $v\in W^{1,2}_0(B_R(x_0))$, and for every compact set $E\subset B_R(x_0)$.
    \\ (ii) Assume that $n=2$.  Then,
    \begin{multline}\label{mag2}
        \bigg(\int_{\partial \Om \cap E}|v|\,d\H^1 \bigg)^2 \\ \leq c\,(1+L_\Om)^5\,(1+\ell^2)\,\big( \H^1(\partial \Om \cap E) \big)^2\,\log\Big( 1+\frac{1}{\H^1(\partial \Om \cap E)}\Big)\,\int_{\Om\cap B_R(x_0)}|\nabla v|^2 dx
    \end{multline}
for some absolute constant $c$, for every  $v\in W^{1,2}_0(B_R(x_0))$, and for every compact set $E\subset B_R(x_0)$.
\end{lemma}

\begin{proof}
    Without loss of generality, we may assume that $x_0=0$, and hence \eqref{may100} is in force. 
Moreover, one can deal  with functions $v\in C^{0,1}_0(B_R)$, since general case follows via a standard density argument.
    \\ Part (i). Set $q=2\frac{(n-1)}{(n-2)}$. Since $B_R\subset B'_R\times (-\ell,\ell)$, from inequalities \eqref{sob:imm1} and \eqref{tr:gamma0} one can deduce that
    \begin{equation}\label{giu10}
        \begin{split}
           &\int_{\partial \Om\cap B_R} |v|^q d\H^{n-1}  = \int_{\partial \Om\cap \big(B'_R\times(-\ell,\ell)\big)}|v|^q d\H^{n-1}=\int_{B'_R}\big|{\rm Tr}(v\circ \psi^{-1})(y')\big|^q\sqrt{1+|\nabla \vphi(y')|^2} dy'
           \\
           &\leq \sqrt{1+L_\Om^2}\,\int_{B'_R}|{\rm Tr}(w)|^q dy'\leq c(n)\,\sqrt{1+L_\Om^2}\,\|{\rm Tr}(w)\|^q_{W^{\frac{1}{2},2}(\R^{n-1})}\leq c'(n)\,\sqrt{1+L_\Om^2}\,\|w\|^q_{W^{1,2}(\rn)}\,,
        \end{split}
    \end{equation}
    where $w$ is the function defined by \eqref{def:evenext}.
    Hence, since $\mathrm{supp} (w)\subset B'_R\times (-2\ell,2\ell)\subset (-2\ell,2\ell)^n$, by equations \eqref{even:norm} and \eqref{june221} we have that
 \begin{align}\label{giu11}
            \int_{\partial \Om\cap B_R}|v|^q d\H^{n-1}&\leq c(n)(1+L_\Om^2)^{1/2}(1+\ell^2)^{q/2}\|\nabla w\|^{q}_{L^2(\rn)}
            \\ \nonumber
            &=c'(n)\,(1+L_\Om)(1+\ell^2)^{q/2}\Big(\int_{\R^n_+}|\nabla (v\circ \psi^{-1})|^2 dy\Big)^{q/2}
            \\ \nonumber
            &\leq c''(n)(1+L_\Om)^{1+q}(1+\ell^2)^{q/2}\,\bigg(\int_{\Om\cap B_R}|\nabla v|^2 dx\bigg)^{q/2}.
        \end{align}
Notice that, besides inequality \eqref{giu10},  the first inequality in \eqref{giu11} also relies upon a standard Poincar\'e inequality on the cube $(-2\ell,2\ell)^n$ (whose constant is $4l$). Inequality 
    \eqref{mag1} follows from \eqref{giu11}, via H\"older's inequality.
    \\ Part (ii). H\"older's inequality in Orlicz spaces ensures that 
 \begin{align}\label{mms}
            \int_{\partial \Om \cap E}|v|\,d\H^1& =\int_{\psi(E)\cap B'_R}\big|{\rm Tr}(v\circ\psi^{-1})\big|\sqrt{1+|\nabla \vphi(y')|^2}dy'
            \\ \nonumber
            &\leq \sqrt{1+L_\Om^2}\int_{\R}\big|{\rm Tr}(v\circ\psi^{-1})(y')\big|\,\chi_{\psi(E)\cap B'_R}(y') dy'
            \\ \nonumber
            &\leq 2 \sqrt{1+L_\Om^2}\,\|v\circ \psi^{-1}\|_{L^A(\R)}\,\|\chi_{\psi(E)\cap B'_R(0')}\|_{L^{\widetilde{A}}(\R)},
        \end{align}
    where $\chi_{F}$ denotes the characteristic function of a set $F$, and $\widetilde{A}$ the  Young conjugate  of $A$. One has that
    \begin{equation*}
        \|\chi_{F}\|_{L^{\widetilde{A}}(\R)}=  \frac{1}{\widetilde A^{-1} (1/|F|)} \leq |F|\,A^{-1}(1/|F|)
    \end{equation*}
for every measurable set $F \subset \R$.
    Since $A^{-1}(t)=\sqrt{\log(1+t)}$, and
    \begin{equation*}
        |\psi(E)\cap B'_R|\leq \H^1( \partial\Om \cap E)\leq \sqrt{1+L_\Om^2}\,|\psi(E)\cap B'_R|\,,
    \end{equation*}
    we obtain that 
    \begin{equation}\label{giu14}
        \|\chi_{\psi(E)\cap B'_R}\|_{L^{\widetilde{A}}(\R)}\leq \H^1( \partial\Om \cap E)\,\sqrt{\log\bigg(1+\frac{(1+L_\Om^2)^{1/2}}{\H^1( \partial\Om \cap E)}   \bigg)}\,.
    \end{equation}
   Inequalities \eqref{sob:imm2}, \eqref{tr:gamma0}, \eqref{mms} and \eqref{giu14} enable one to infer that
    \begin{align}\label{mms1}
         \int_{ \partial\Om \cap E}|v|\,d\H^1& \leq \sqrt{1+L_\Om^2}\,\|v\circ \psi^{-1}\|_{L^A(\R)}\, \H^1( \partial\Om \cap E)\,\sqrt{\log\bigg(1+\frac{(1+L_\Om^2)^{1/2}}{\H^1( \partial\Om \cap E)}   \bigg)}
         \\ \nonumber
         &\leq c \,\sqrt{1+L_\Om^2}\,\|{\rm Tr}(w)\|_{W^{\frac{1}{2},2}(\R)} \,\H^1( \partial\Om \cap E)\,\sqrt{\log\bigg(1+\frac{(1+L_\Om^2)^{1/2}}{\H^1( \partial\Om \cap E)}   \bigg)}
         \\ \nonumber
         &\leq c'\,\sqrt{1+L_\Om^2}\,\|w\|_{W^{1,2}(\R^n)}\,\H^1( \partial\Om \cap E)\,\sqrt{\log\bigg(1+\frac{(1+L_\Om^2)^{1/2}}{\H^1( \partial\Om \cap E)}   \bigg)}
         \\ \nonumber
         &\leq c'' \,(1+\ell^2)^{1/2}\,(1+L_\Om)^{2}\,\|\nabla v\|_{L^2(\Om\cap B_R)}\,\H^1( \partial\Om \cap E)\,\sqrt{\log\bigg(1+\frac{(1+L_\Om^2)^{1/2}}{\H^1( \partial\Om \cap E)}   \bigg)},
         \end{align}
         where $c''$ is an absolute constant. Note that the last inequality rests upon the inequality
    \begin{equation*}
        \|w\|_{W^{1,2}(\R^2)}\leq c \,(1+\ell^2)^{1/2}\,(1+L_\Om)\,\|\nabla v\|_{L^2(\Om\cap B_R)}\,,
    \end{equation*}
    with $c$ an absolute constant, which is in turn a consequence of \eqref{even:norm}, \eqref{june221} 
    and Poincar\'e inequality, as in the case $ n \geq 3.$
    Finally, since 
    \begin{equation*}
        1+\frac{(1+L_\Om^2)^{1/2}}{\H^1(\partial \Om \cap E)} \leq \bigg(1+\frac{1}{\H^1(\partial \Om \cap E)} \bigg)^{(1+L_\Om^2)^{1/2}}\,
    \end{equation*}
    we have that 
    \begin{equation*}
      \log\bigg(1+\frac{(1+L_\Om^2)^{1/2}}{\H^1(\partial \Om \cap E)}   \bigg)\leq (1+L_\Om^2)^{1/2}\,\log\bigg(1+\frac{1}{\H^1(\partial \Om \cap E)}   \bigg).
    \end{equation*}
    Coupling the latter inequality with \eqref{mms1} yields \eqref{mag2}.
\end{proof}


\begin{lemma}\label{prop:iscapmar}
    Let  $\Om\ $ be a bounded Lipschitz domain in $\rn$,   with Lipschitz characteristic  $\mathfrak L_\Om = (L_\Om, R_\Om)$.  Assume that   $\varrho$ is a nonnegative function on $\partial \Omega$ such that $\varrho\in L^1(\partial \Om)$. 
Then, 
    \begin{equation}\label{mssss}
            \sup_{
\begin{tiny}
 \begin{array}{c}{E\subset  B_R(x_0)} \\
E \text{ compact}
 \end{array}
  \end{tiny}
} \frac{\int_{\partial \Om\cap E}\varrho\,d\H^{n-1}}{\mathrm{cap}(E,B_R(x_0))}\leq \begin{cases}
    c\,(1+L_\Om)^{\frac{3n-4}{n-1}}\,(1+\ell^2)\,\|\varrho\|_{L^{n-1,\infty}(\partial\Om\cap B_R(x_0))} \quad \text{if $n\geq 3$}
    \\
    \\
    c\,(1+L_\Om)^{5}(1+\ell^2)\,\|\varrho\|_{L^{1,\infty}\log L(\partial \Om\cap B_R(x_0))}\quad \text{if $n=2$}
\end{cases}
    \end{equation}
for some constant $c=c(n)$, 
for every $x_0\in \partial \Om$ and every $R\in (0,R_\Om]$.
\end{lemma}

 
\begin{proof}[Proof of Lemma \ref{prop:iscapmar}]
    We may assume that the norms on the right-hand side of inequality \eqref{mssss} are finite, otherwise there is nothing to prove. 
    Let $\{E_k\}$ be a sequence of compact sets such that $E_k\subset B_R$,  and
    \begin{equation*}
        \sup_{
\begin{tiny}
 \begin{array}{c}{E\subset  B_R} \\
E \text{ compact}
 \end{array}
  \end{tiny}
} \frac{\int_{\partial \Om\cap E}\varrho\,d\H^{n-1}}{\mathrm{cap}(E,B_R)}=\lim_{k\to \infty}\frac{\int_{\partial \Om\cap E_k}\varrho\,d\H^{n-1}}{\mathrm{cap}(E_k,B_R)}\,.
    \end{equation*}
Applying  either inequality \eqref{mag1} or \eqref{mag2} with functions
$u\in C^{0,1}_0(B_R)$ such that $v\geq 1$ on $E_k$, and  taking the infimum of the ratio of  the integrals on their two sides among these functions  $u$ tell us that
 \begin{equation}\label{mms2}
     \begin{cases}
         \H^{n-1}(\partial \Om\cap E_k)^{\frac{n-2}{n-1}}\leq c(n)\,(1+L_\Om)^{\frac{3n-4}{n-1}}(1+\ell ^2)\,\mathrm{cap}(E_k, B_R)\quad & \text{if $n\geq 3$}
         \\
         \\
          \log\bigg( \displaystyle 1+\frac{1}{\H^1(\partial \Om\cap E_k)}\bigg)^{-1}\,\leq c\,(1+L_\Om)^{5}(1+\ell^2)\,\mathrm{cap}(E_k,B_R)\quad & \text{if $n=2$,}
     \end{cases}
 \end{equation}
 where $c$ is an absolute constant if $ n=2$.  
\\  If $n \geq 3$, then  inequality \eqref{mms2} and an application of the Hardy-Littlewood inequality for rearrangements enable one to deduce that
 \begin{align}\label{june20}
        \frac{\int_{\partial \Om\cap E_k}\varrho \,d\H^{n-1}}{\mathrm{cap}(E_k,B_R)}&\leq
            c(n)\,(1+L_\Om)^{\frac{3n-4}{n-1}}(1+\ell^2)\,\frac{\int_{\partial \Om\cap E_k}\varrho \,d\H^{n-1}}{\big(\H^{n-1}(\partial \Om \cap E_k)\big)^{q/2}}  
            \\ \nonumber
            &\leq c(n)\,(1+L_\Om)^{\frac{3n-4}{n-1}}(1+\ell^2)\,\frac{\int_0^{\H^{n-1}(\partial \Om \cap E_k)}\varrho^\ast(t) \,dt}{\big(\H^{n-1}(\partial \Om \cap E_k)\big)^{\frac{n-1}{n-2}}} 
            \\ \nonumber
            &\leq c(n)\,(1+L_\Om)^{\frac{3n-4}{n-1}}(1+\ell^2)\,\sup_{s>0}s^{\frac{1}{n-1}}\varrho^{\ast\ast}(s)
            \\ \nonumber
            &\leq c(n)\,(1+L_\Om)^{\frac{3n-4}{n-1}}(1+\ell^2)\,\|\varrho\|_{L^{n-1,\infty}(\partial\Om\cap B_R)}.
    \end{align}
If $n=2$, then   inequality \eqref{mms2} and  the Hardy-Littlewood inequality again yield:
  \begin{align}\label{june21}
         \frac{\int_{\partial \Om\cap E_k}\varrho \,d\H^{n-1}}{\mathrm{cap}(E_k,B_R)}&\leq c\,(1+L_\Om)^{5}(1+\ell^2)\,\log\Big( 1+\frac{1}{\H^1(\partial \Om\cap E_k)}\Big)\,\int_{\partial \Om\cap E_k}\varrho\,d\H^1
         \\ \nonumber         &\leq \,(1+L_\Om)^{5}(1+\ell^2)\,\log\Big( 1+\frac{1}{\H^1(\partial \Om\cap E_k)}\Big)\,\int_0^{\H^1(\partial \Om\cap E_k)}\varrho^\ast(t)\,dt
         \\ \nonumber
         &\leq c\,(1+L_\Om)^{5}(1+\ell^2)\,\sup_{s>0}\Big(s\,\log\big(1+\tfrac{1}{s}\big)\,\varrho^{\ast\ast}(s)\Big)
         \\ \nonumber
         &\leq c\,(1+L_\Om)^{5}(1+\ell^2)\,\|\varrho\|_{L^{1,\infty}\log L(\partial \Om\cap B_R)}\,.
     \end{align}
  Then, \eqref{mssss} follows by letting $k\to \infty$ in \eqref{june20} and \eqref{june21}.
\end{proof}

We are now in a position to prove Proposition \ref{prop:isocapmar1}.

\begin{proof}[Proof of Proposition \ref{prop:isocapmar1}] Recalling \eqref{ell}, from inequality \eqref{mssss} of 
Lemma \ref{prop:iscapmar} we infer that, for any $R\in (0,R_\Om)$,
 \begin{equation}\label{isocap:isomar}
 \K_{\Om,\varrho}(R)\leq
 \begin{cases}
          c\,(1+L_\Om)^{\frac{n-2}{n-1} + 4 }\,\Psi_{\Om,\varrho}(R)\quad & \text{if }\,n\geq 3
        \\
        \\
        c\,(1+L_\Om)^{7}\,\Psi_{\Om,\varrho}(R) \quad & \text{if }\,n=2\,,
        \end{cases}
    \end{equation}
    for some constant $c=c(n)$. Inequality \eqref{giu30} follows from \eqref{giu1} and \eqref{isocap:isomar}.
\end{proof}

We conclude with an estimate for the function $\K_{\Om,\varrho}$ in the special case when $\varrho\in L^\infty(\partial \Om)$. This is crucial in view of Theorem \ref{thm:C11}.

\begin{corollary}\label{cor:infty}
Let $\Om$ be a bounded Lipschitz domain in $\rn$, with Lipschitz characteristic  $\mathfrak L_\Om = (L_\Om, R_\Om)$. Assume that   $\varrho$ is a nonnegative function on $\partial \Omega$ such that $\varrho\in L^\infty(\partial \Om)$. Then,
    \begin{equation}
        \int_{\partial \Om\cap B_R(x_0)} v^2\,\varrho\,d\H^{n-1}\leq 
        \begin{cases}
            c(n)\,(1+L_\Om)^{9}\,R\,\|\varrho\|_{L^\infty(\partial \Om)}\,\int_{\Om\cap B_R(x_0)}|\nabla v|^2\,dx \quad &  \text{if $n\geq 3$}
            \\
            \\
             c(n)\,(1+L_\Om)^{12}\,R\,\log\big(1+\frac{1}{R}\big)\,\|\varrho\|_{L^\infty(\partial \Om)}\,\int_{\Om\cap B_R(x_0)}|\nabla v|^2\,dx\quad &  \text{if $n=2$}
        \end{cases}
    \end{equation}
for $x_0\in \partial \Om$,  for  $R\in (0,R_\Om]$,  and for $v\in W^{1,2}_0(B_R(x_0))$.
\end{corollary}

\begin{proof} 
There exist positive constants $c_1= c_1(n)$ and $c_2= c_2(n)$ such that
\begin{equation*}
    c_1(n)\,R^{n-1}\leq \H^{n-1}\big(\partial \Om\cap B_R(x_0)\big)\leq c_2(n)\,\sqrt{1+L_\Om^2}\,R^{n-1}
\end{equation*}
for every $x_0\in \partial \Om$ and $R\in (0,R_\Om]$, hence 
\begin{equation*}
    \begin{cases}
         \|\varrho\|_{L^{n-1,\infty}(\partial \Om\cap B_R(x_0))}\leq c(n)\,(1+L_\Om)^{\frac{1}{(n-1)}}\,R\,\|\varrho\|_{L^\infty(\partial \Om\cap B_R(x_0))}\quad & \text{if $n\geq 3$}
        \\ 
        \\
        \|\varrho\|_{L^{1,\infty}\log L(\partial \Om\cap B_R(x_0))}\leq c\,(1+L_\Om)\,R \log\big(1+\frac{1}{R} \big)\,\|\varrho\|_{L^\infty(\partial \Om\cap B_R(x_0))}\quad & \text{if $n=2$}
    \end{cases}
\end{equation*}
for every $x_0\in \partial \Om$ and $R\in (0,R_\Om]$. From inequality \eqref{mssss} of 
Lemma \ref{prop:iscapmar} we infer that, 
\begin{equation}\label{isocap:isomar-pippo}
        \K_{\Om,\varrho}(R)\leq
        \begin{cases}
            c(n)\,(1+L_\Om)^{5}\,R\,\|\varrho\|_{L^\infty(\partial \Om)} \quad & \text{if $n\geq 3$}
            \\
            \\
            c(n)\,(1+L_\Om)^{8}\,R\,\log\big(1+\frac{1}{R}\big)\,\|\varrho\|_{L^\infty(\partial \Om)}\quad & \text{if $n=2$}\,,
        \end{cases}      
    \end{equation}
for $R\in (0,R_\Om]$.
The desired conclusions then follow
from these estimates, via Proposition \ref{prop:isocapgen}. 
\end{proof}

\section{Global estimates: Dirichlet problems }\label{sec:dir}

Here, we are concerned with proofs of our global results for solutions to Dirichlet problems.  
Assume  that $f \in L^2(\Omega)$.  A function \textcolor{black}{$u \in \mathcal T ^{1,1}_0(\Omega)$} will be called a generalized solution to the Dirichlet problem \eqref{eq:dir2} if $\A(\nabla u) \in L^1(\Omega)$,
\begin{equation}\label{231}
\int_\Omega \A(\nabla u) \cdot \nabla \varphi \, dx=\int_\Omega f \varphi \,dx\,
\end{equation}
for every $\varphi \in C^{\infty}_0(\Omega)$, and there exists a sequence 
 $\{f_k\}\subset  C^\infty _0(\Omega)$
such that 
$f_k \to f$   in $L^2(\Omega)$ and the sequence of weak solutions  $\{u_k\}$   to  problem \eqref{eq:dir2}, with $f$ replaced by $f_k$, satisfies 
%
%
$$u_k\to u \quad
\hbox{ a.e. in $\Omega$.}$$
In \eqref{231}, $\nabla u$ stands for the function $Z_u$ fulfilling \eqref{504}.
\\
By \cite[Theorem 3.2]{cia17},  there exists a unique generalized solution $u$ to problem  \eqref{eq:dir2},  and
\begin{equation}\label{estgraddir}
\int_\Omega |\A(\nabla u)|\, dx \leq c\,|\Om|^{1/n} \int_\Omega |f|\, dx
%
\end{equation} 
for some constant  $c=c( n, i_b, s_b,\l,\L)$.  
Moreover, if $\{f_k\}$ is any sequence as above, and $\{u_k\}$ is the associated sequence of weak solutions, then 
\begin{equation}\label{gen1}
u_k \to u \quad \hbox{and} \quad \nabla u_k \to \nabla u \quad \hbox{a.e. in $\Omega$,}
\end{equation}
up to subsequences.

The generalized solutions introduced above agree with the classical weak solutions, provided that the function $f$ has a sufficiently high degree of integrability.  Recall that a function $u$ is called a weak solution to the Dirichlet problem \eqref{eq:dir2} if $u \in W^{1,B}_0(\Omega)$ and equation \eqref{231} holds for every function $\vphi \in W^{1,B}_0(\Omega)$. Here, $W^{1,B}_0(\Omega)$ denotes the Orlicz-Sobolev space, built upon $B$, of those functions vanishing in the usual appropriate sense on $\partial \Omega$. Minimal conditions on $f$ for a weak solution to be well-defined and to exist can be exhibited -- see \cite{ACCZ}. They rely upon a sharp embedding theorem for Orlicz-Sobolev spaces \cite{Ci96, Ci97}. In view of our purposes, we shall only need to deal with weak solutions under the assumption that $f \in L^\infty(\Omega)$, in which case they  certainly exist whatever $B$ is.
%


Having dispensed with the necessary definitions, 
we begin preparing for our proofs with a few lemmas concerning Sobolev functions. The first one deals with the continuity in Sobolev spaces of the composition operator  for vector-valued functions, see \cite[Proposition 2.6]{musina}.

%
%
%

\begin{lemmaalph}
\label{lemma:chain}
Let $\Om$ be a bounded open set  in $\mathbb R^n$ and let  $F\in C^{0,1}(\rn) \cap C^1(\rn\setminus\{0\})$. Assume that $V : \Om \to \rn$  is such that  $V\in W^{1,2}(\Om)$, and let $\{V_m\}$ be a sequence of functions $V_m : \Om \to \rn$ such that $V_m\in W^{1,2}(\Om)$ for $m\in \mathbb N$.
 If 
\begin{equation}\label{xe}
   V_m \to V \quad\text{in }\,W^{1,2}(\Om),
\end{equation}
then,
\begin{equation}
   F(V_m) \to F(V)\quad\text{in }\,W^{1,2}(\Om).
\end{equation}
\end{lemmaalph}

The following lemma is a straightforward consequence of Propositions \ref{prop:isocapgen} and \ref{prop:isocapmar1}, applied with $\varrho = |\B^H|$, and formulas \eqref{equiv:mars}--\eqref{gen12}.

\begin{lemma}\label{traceineq} 
Let $\Omega$ be a bounded Lipschitz domain in $\mathbb R^n$,   with Lipschitz characteristic $\mathfrak L_\Om= (L_\Om,R_\Om)$. Assume that $\partial \Omega \in W^{2,1}$.
\\ (i) If  $\mathcal K_\Omega(r)<\infty$ for $r \in \big(0,R_\Om]$, then, 

\begin{equation}\label{tr:ineqcap}
    \int_{\partial\Om\cap B_r(x)} v^2\,|\B^H|\,d\H^{n-1}\leq c_0(n,\l,\L) \,(1+L_\Om)^4\,\mathcal K_\Omega(r)\int_{\Om\cap B_r(x)}|\nabla v|^2\,dy,
\end{equation}
for every $x\in \partial\Omega$, $r\in \big(0,R_\Om]$ and $v\in W^{1,2}_0(B_r(x))$.
\\ (ii) If  $\Psi_\Omega(r)<\infty$ for $r\in \big(0,R_\Om]$,  then, 
\begin{equation}\label{tr:ineq}
    \int_{\partial\Om\cap B_r(x)} v^2\,|\B^H|\,d\H^{n-1}\leq c_0(n,\l,\L)\,
    \big(1+L_\Om\big)^{11}\, \Psi_\Omega(r)\int_{\Om\cap B_r(x)}|\nabla v|^2\,dy,
\end{equation}
for every $x\in \partial\Omega$, $r\in \big(0,R_\Om]$ and $v\in W^{1,2}_0(B_r(x))$.
\end{lemma}

%
%
%

The inequality provided by the next  lemma is well known. The point here is the dependence of the constants on  Lipschitz characteristics of  domains.

\begin{lemma} \label{poin:thm}
 Let $\Omega$ be a bounded Lipschitz domain in $\mathbb R^n$,   with Lipschitz characteristic $\mathfrak L_\Om= (L_\Om,R_\Om)$. Then,
  \begin{equation}\label{giu31}
      \|v\|^2_{L^2(\Om)}\leq \s\,\|\nabla v\|^2_{L^2(\Om)}+c\,\frac{(1+\s)^2}{\s}\,\frac{d_\Om^{2n(n+2)}}{r^{(2n+1)(n+2)}}\,(1+L_\Om)^{n+2}\,\|v\|_{L^1(\Om)}^2
  \end{equation}
for some constant $c=c(n)$ and for every  $\s>0$, $r\in (0,R_\Om]$ and $v\in W^{1,2}(\Om)$.
\end{lemma}

\begin{proof}
An application of an extension theorem by Stein, in the form of  \cite[Theorem 13.17]{leo}, ensures that there exists a
bounded linear operator $\mathcal E\,:\,W^{1,2}(\Om)\to W^{1,2}(\rn)$ such that,
    \begin{equation}\label{stein}
        \begin{split}
          \|\mathcal E(v)\|_{L^2(\rn)}&\leq c(n)\,\Big(\frac{d_\Om}{r}\Big)^n\,\|v\|_{L^2(\Om)}
          \\
          \|\nabla \mathcal E(v)\|_{L^2(\rn)}\leq c(n)\,&\frac{d_\Om^{2n}}{r^{2n+1}}(1+L_\Om)\,\Big(  \|v\|_{L^2(\Om)}+\|\nabla u\|_{L^2(\Om)}\Big)
        \end{split}
    \end{equation}
 for  $r \in (0, R_\Om)$ and for $v\in W^{1,2}(\Om)$.
    Set 
    \begin{equation*}
        s=\begin{cases}
           \frac{2n}{n-2}\quad & \text{if $n\geq 3$}
            \\
            4  \quad & \text{if $n=2$,}
        \end{cases}
    \end{equation*}
    and 
    \begin{equation*}
         \a=\begin{cases}
            \frac{2}{n+2}\quad & \text{if $n\geq 3$}
            \\
            \frac{1}{3}\quad & \text{if $n=2$,}
        \end{cases}
    \end{equation*}
whence $\frac{1}{2}=\a+\frac{1-\a}{s}$.
    From the H\"older inequality, the Sobolev inequality, and the inequalities in \eqref{stein} one deduces that
    \begin{equation*}
        \begin{split}
          \|v \|_{L^2(\Om)} & \leq  \|v \|_{L^1(\Om)}^\a\, \|v\|_{L^s(\Om)}^{1-\a}= \|v \|_{L^1(\Om)}^\a\, \|\mathcal E(v)\|_{L^s(\Om)}^{1-\a}
          \\
          &\leq  \|v \|_{L^1(\Om)}^\a\, \|\mathcal E(v)\|_{L^s(\rn)}^{1-\a}  \leq c(n)\, \|v \|_{L^1(\Om)}^\a\, \|E(v)\|^{1-\a}_{W^{1,2}(\rn)}
          \\
          &\leq c(n)\,C^{1-\a}\, \|v\|_{L^1(\Om)}^\a\,\Big( \|v\|_{L^2(\Om)}+\|\nabla v \|_{L^2(\Om)}\Big)^{1-\a}\,,
        \end{split}
    \end{equation*}
    where we have set
    \begin{equation*}
        C=\frac{d_\Om^{2n}}{r^{2n+1}}\,(1+L_\Om)\,.
    \end{equation*}
   An application  of Young's inequality with exponents $\tfrac{1}{\a}$ and $\tfrac{1}{1-\a}$ yields
     \begin{align*}
         \|v \|^2_{L^2(\Om)} & \leq c(n)\,C^{2(1-\a)}\, \|v \|^{2\a}_{L^1(\Om)}\Big( \|v \|^2_{L^2(\Om)}+\|\nabla v \|^2_{L^2(\Om)} \Big)^{1-\a}
        \\ 
        & \leq \e\Big( \|v\|^2_{L^2(\Om)}+\|\nabla v \|^2_{L^2(\Om)} \Big)+c'(n)\,\e^{-\frac{\a}{1-\a}}\, C^{2\frac{1-\a}{\a}}\, \|v \|_{L^1(\Om)}^2\,,
    \end{align*}
for $\e\in (0,1)$. Hence,   by setting $\s=\frac{\e}{1-\e}$, one obtains that
    \begin{equation}\label{giu32}
        \begin{split}
             \|v\|_{L^2(\Om)}^2\leq \s\,\|\nabla v \|^2_{L^2(\Om)}+c'(n)\,\Big(1+\frac{1}{\s} \Big)^{\frac{\a}{1-\a}}\,(1+\s)\,C^{2\frac{1-\a}{\a}}\, \|v\|_{L^1(\Om)}^2
        \end{split}
    \end{equation}
  Notice that
    \begin{equation*}
        \Big(1+\frac{1}{\s} \Big)^{\frac{\a}{1-\a}}\,(1+\s)\leq \frac{(1+\s)^2}{\s}.
    \end{equation*}
Moreover, since $r\leq  R_\Om \leq d_\Om$, $\frac{2(1-\a)}{\a}\leq n+2$ and $C\geq 1$, we have that
    \begin{equation*}
        C^{2\frac{1-\a}{\a}}\leq C^{n+2}=\Big( \frac{d_\Om^{2n}}{r^{2n+1}}\Big)^{n+2}\,(1+L_\Om)^{n+2}.
    \end{equation*}
   Inequality \eqref{giu31} thus follows from \eqref{giu32}.
\end{proof}

\begin{proof}[Proof of Theorem \ref{an:mainthm2}, Dirichlet problems]
We split the proof into several steps. 
\\ \emph{Step 1}. 
Here we assume that 
\begin{equation}\label{fCinf}
    f\in C^{0,\a}_0(\Omega)
\end{equation}
and
\begin{equation}\label{deominf}
    \partial\Om\in C^{2,\a}.
\end{equation}
For every $\e\in (0,1)$,  let $\A_\e$ be the function defined as in \eqref{def:Ae'}, and denote by $u_\e$  the weak solution to the Dirichlet problem
\begin{equation}\label{pb:ue}
\begin{cases}
    -\mathrm{div}\big(\A_\e(\nabla u_\e)\big)=f\quad&\text{in }\Omega
    \\
    u_\e=0\quad&\text{on }\partial\Omega\,.
\end{cases}
\end{equation}
The same argument exploited in connection with problem \eqref{eq:locue} ensures that there exists a unique solution $u_\e\in W^{1,2}_0(\Omega)$ to problem \eqref{pb:ue}.
\\ We claim that there exists  $\theta= \theta (n,i_b,s_b,\l,\L,\|f\|_{\infty})\in (0,1)$ such that
    \begin{equation}\label{marocco}
    u_\e\in W^{2,2}(\Om)\cap C^{1,\theta}(\overline{\Om})\,,
    \end{equation}
   and 
\begin{equation}\label{argentina}
    u_\e\to u\quad\text{in $C^{1,\theta'}_{\rm loc}(\Om)$}\,,
\end{equation}
for every $0<\theta'<\theta$.
\\ Furthermore, fixing  any $\e>0$, there exists a sequence $\{u_{\e,m}\}$ such that
\begin{equation}\label{r:um}
    u_{\e,m}\in C^{2,\a}(\overline{\Om})\quad \text{and} \quad u_{\e,m}=0\quad\text{on $\partial \Om$}\,,
\end{equation}
and
\begin{equation}\label{c:um} 
u_{\e,m}    \xrightarrow[]{m\to \infty} u_\e\quad\text{in  $W^{2,2}(\Om)$ and $C^{1,\theta'}(\overline{\Om})$.}
\end{equation}
To prove our claims, we make use of an argument from \cite[Section 3]{cin:deg}, and define, for $\e \in (0,1)$ and $\delta>0$ the regularized vector field $\A_{\e,\d} : \rn \to \rn$ by
\begin{equation}\label{aed}
    \A_{\e,\d} = \A_\e\ast \rho_\d \qquad \text{in $\rn$}.
\end{equation}
Here, $\{\rho_\d\}$ denotes a family of standard, radially symmetric mollifiers. By properties of convolutions, $\A_{\e,\d}\in C^\infty(\rn)$ and $\lim_{\delta \to 0^+}\A_{\e,\d}= \A_\e$ locally uniformly in $\rn$. Also, thanks to inequalities \eqref{est:Ae}, one can verify    that 
\begin{equation}\label{est:Aed}
    \e\,\l\,\min\{1,i_b\}\,\mathrm{Id}\leq \nabla_\xi \A_{\e,\d}(\xi)\leq \e^{-1}\,\L\,\max\{1,s_b\}\,\mathrm{Id} \quad \text{for  $\xi\in\rn$.}
\end{equation}
Next consider the family $\{w_{\e,\d}\}$ of the unique solutions  to the problems
\begin{equation}\label{eq:ged}
\begin{cases}
    -\mathrm{div}\big( \A_{\e,\d}(\nabla w_{\e,\d})\big)=f\quad&\text{in }\Om
    \\
    w_{\e,\d}=0\quad&\text{on }\partial\Om\,.
    \end{cases}
\end{equation}
Thanks to \eqref{est:Aed}, classical results tell us that
 $w_{\e,\d}\in W^{2,2}(\Om)$, and
\begin{equation}\label{H2:d}
    \|w_{\e,\d}\|_{W^{2,2}(\Om)}\leq c_0
\end{equation}
for some constant $c_0=c_0(n,\l,\L,\e,\Om,\|f\|_{L^\infty(\Om)})$, see e.g. \cite[pp. 270-277]{ladyz}, or \cite[Theorem 8.2]{ben:fre}, or \cite[Chapter 8.4]{giusti}.
Notice that, by standard elliptic regularity theory (\cite[Theorem 9.19]{gt} or \cite[Theorem 6.3,pag. 283]{ladyz}), 
\begin{equation}\label{zxx}
    w_{\e,\d}\in C^{2,\a}(\ov{\Om})\,.
\end{equation}
Next, by   \cite[Corollary 6.1]{stampacchia}, there exists a constant $c_1$, depending on the same quantities as $c_0$, such that
\begin{equation}
\|w_{\e,\d}\|_{L^\infty(\Om)}\leq c_1.
\end{equation}
Coupling  this piece of information with
 \cite[Theorem 1]{lieb88} entails that
\begin{equation}\label{unif:indeltag}
\|w_{\e,\d}\|_{C^{1,\theta}(\overline{\Om})}\leq c_2\,,
\end{equation}
for some constant $c_2$, with the same dependence as $c_0$ and $c_1$. In particular, these constants are independent of $\delta$.
\\ Thanks to inequalities \eqref{H2:d} and \eqref{unif:indeltag}, there exist a function $w_\e\in W^{2,2}(\Om)\cap C^{1,\theta}(\overline{\Om})$ and a sequence $\{\d_k\}$ such that $\d_k \to 0^+$,
\begin{equation}\label{zx}
   w_{\e,\d_k}  \xrightarrow[]{k\to \infty} w _\e \quad\text{in $C^{1,\theta'}(\overline{\Om})$} \quad  \text{and} \quad  
w_{\e,\d_k}   \xrightharpoonup[]{k\to \infty} w _\e
\quad \text{weakly in $W^{2,2}(\Om)$}
\end{equation}
for every $\theta ' \in (0,\theta)$.
\\ Passing to the limit as  $k\to \infty$ in the weak formulation of problem \eqref{eq:ged} shows that $w_\e$ is solution to problem \eqref{pb:ue}, whence $w_\e=u_\e$ by the uniqueness of the solution.  Property \eqref{marocco} is thus established.
\\ We next prove  properties \eqref{r:um} and \eqref{c:um}. Thanks to the Banach-Saks theorem, the weak convergence \eqref{zx} in the Hilbert space $W^{2,2}(\Omega)$ ensures that there exists a 
subsequence $\{\d_{k_l}\}_{l\in\N}$ such that, on setting
$$ u_{\e,m}= \frac{1}{m}\sum_{l=1}^m w_{\e,\d_{k_l}},$$
one has that
\begin{equation}\label{BS1}
     u_{\e,m}  \xrightarrow[]{m\to \infty}  u_\e\quad\text{in $W^{2,2}(\Om)$.}
\end{equation}
Moreover, by the convergence of $w_{\e,\d_k}$ to $u_\e$ in $C^{1,\theta'}(\ov{\Om})$, 
\begin{equation}\label{BS2}
     u_{\e,m} \xrightarrow[]{m\to \infty} u_\e \quad\text{in $C^{1,\theta'}(\ov{\Om})$}\,.
\end{equation}
Thanks to \eqref{zxx},\eqref{BS1} and \eqref{BS2}, the sequence $\{u_{\e,m}\}$ satisfies properties \eqref{r:um} and \eqref{c:um}.
\\ To complete the proof of this step, we establish the convergence in
 \eqref{argentina}.  By the minimizing property of  the function $u_\e$ for the functional $J^H_\e$, defined as in \eqref{def:minue} with $B_{3R}$ replaced by $\Omega$, we have that \mbox{$J^H_\e(u_\e) \leq J^H_\e (0)$}, hence 
\begin{equation}\label{unif:orlicz}
    \int_\Om B_\e\big(H(\nabla u_\e)\big)\,dx\leq \int_\Om f\,u_\e\,dx.
\end{equation}
Thanks to \cite[Theorem 2]{tal}, inequalities \eqref{appr:2},  and property \eqref{cortona10},  there exists a constant $c=c(n,i_b,s_b,\l,\L,|\Om|,\|f\|_{L^\infty(\Om)})$ such that
\begin{equation}\label{gl:b2}
    \|u_\e\|_{L^{\infty}(\Om)}\leq c.
\end{equation}
Owing to inequalities   \eqref{appr:2}, 
\eqref{bounds:H}, and \eqref{gl:b2}, one can deduce from   \eqref{unif:orlicz} that
\begin{equation}\label{uf:orlicz}
    \int_\Om B_\e(|\nabla u_\e|)\,dx\leq c
\end{equation}
for some constant $c=c(n,i_b,s_b,\l,\L,|\Om|,\|f\|_{L^\infty(\Om)})$.
Moreover, inequality  \eqref{gl:b2} allows one  to apply \cite[Theorem 1.7]{lieb91} and obtain
\begin{equation*}
    \|u_\e\|_{C^{1,\theta}(\Om')}\leq c \quad \text{for every open set $\Om'\subset\subset \Om$,}
\end{equation*}
and for some constant $c=c(n,\l,\L,i_b,s_b,\l,\L,\Om',\Om,\|f\|_{L^\infty(\Om)})$.
\\ Thus, there exists a sequence of $\{\e_k\}$ such that $\e_k \to 0^+$ and
\begin{equation}\label{conv:tempeps}
 u_{\e_k} \to  v \quad\text{in $C^{1}_{\rm loc}(\Om)$}\,,
\end{equation}
for some function $v\in C^{1}(\Om)$.
\\
Now, we want to show that 
\begin{equation}
\label{march386}
v=u. 
\end{equation}
To this purpose, one can use an analogous argument as at the end of Step 1 of the proof 
Theorem \ref{thm:local}.
%
Specifically, since $B_{\e_k}(t)\to B(t)$ locally uniformly in $[0, \infty)$, from \eqref{conv:tempeps} we have that $B_{\e_k}(|\nabla u_{\e_k}|)\to B(|\nabla v|)$ everywhere in $\Omega$.  From inequalities  \eqref{uf:orlicz} and \eqref{ulb}, and  
Fatou's Lemma  we obtain that
\begin{equation}\label{march399}
    \int_\Om B(|\nabla v|)\,dx\leq c\quad\text{and}\quad  \int_\Om |\nabla u_\e|^{\min\{i_b+1,2 \}}\,dx\leq c\,,
\end{equation}
for some constant $c$ independent on $\e$. Thanks to the reflexivity of the space $W^{1,\min\{i_b+1,2\}}_0(\Om)$,   inequalities \eqref{march399} imply that $v\in W^{1,B}_0(\Om)$.
\\
Owing to  \eqref{conv:tempeps},  passing to the limit as $\e \to 0^+$   in   \eqref{pb:ue} yields:
\begin{equation}\label{xzz}
    \int_\Om \A(\nabla v)\cdot \nabla \vphi\,dx=\int_\Om f\,\vphi\,dx\
\end{equation}
 for every  $\vphi\in C^\infty_0(\Om)$. A density argument as at the end of the proof of Step 1 of Theorem \ref{thm:local} implies that equation \eqref{xzz} holds, in fact, for every function $\varphi \in  W^{1,B}_0(\Om)$.  Thus, $v$ is the weak solution to problem \eqref{eq:dir2}, whence, by its uniqueness, equality \eqref{march386} follows. Thereby, property \eqref{argentina} is a consequence of  \eqref{conv:tempeps} and of the fact that the preceding argument applies to any sequence extracted from the family $\{u_\e\}$.

\vspace{0.3cm}
\noindent \emph{Step 2.} We show that, given any $L,d,M> 0$ and $\ov{r}\in (0,1)$, there exists a positive constant $\ov{c}=\ov{c}(n,\l,\L,i_b,s_b,{L})$ such that, if $\Omega$ is a bounded domain of class $C^{2,\a}$ and Lipschitz characteristic $\mathfrak L_\Om = (L_\Om,R_\Om)$ satisfying $L_\Om\leq {L}$, $d_\Om\leq {d}$, $R_\Om \geq \ov{r}$ and
\begin{equation}\label{assum:Kom}
    \mathcal{K}_\Om(r)\leq \ov{c}\,\quad \quad \text{for  $r\in (0,\ov{r}]$}\,,
\end{equation}
and $u\in W^{1,B}_0(\Om)$ is a weak solution to problem \eqref{231} with  $\|f\|_{L^2(\Om)} \leq {M}$, then


\begin{equation}\label{princc:est}
    \|\A(\nabla u)\|^2_{L^2(\Om)}\leq c_1\,\|f\|_{L^2(\Om)}^2\,\quad\text{and}\quad \|\nabla\A(\nabla u)\|^2_{L^2(\Om)}\leq c_2\,\|f\|_{L^2(\Om)}^2\,.
\end{equation}
Here,   $\ov{c}$, $c_1$ and $c_2$ are constants of the form:
\begin{align}\label{defin:constants}
        & \ov{c}=c(n,\l,\L,i_b,s_b)\,\frac{1}{\big(1+{L}\big)^4}\notag
        \\
        & c_1=c(n,\l,\L,i_b,s_b)\,\frac{d_\Om^{p(n)}\,(1+L_\Om)^{n+2}}{\ov{r}^{(2n+2)(n+2)}}
        \\
        & c_2=c(n,\l,\L,i_b,s_b)\,\frac{d_\Om^{p(n)+n}\,(1+L_\Om)^{n+2}}{\ov{r}^{(2n+3)(n+2)}}\,,\notag
\end{align}
where $p(n)=(2n+1)(n+2)+n$. 

In order to prove this assertion, let us first consider the families of functions $\{u_\e\}$ and $\{\ued\}$ defined in Step 1, and let $\vphi\in C^\infty_0(\rn)$. An application of formula \eqref{reilly:formula}, with $v=\ued$, $h=a_\e\big(H(\nabla \ued(x))\big)$, and $\phi=\vphi^2$ , yields
 \begin{align}\label{main:start}
        \int_\Om & \mathrm{div} \big(\A_\e(\nabla \ued)\big)^2\,\vphi^2\,dx \\ \nonumber & =\int_\Om\mathrm{tr}\big( (\nabla (\A_\e(\nabla \ued)))^2 \big)\,dx
  +\int_{\partial\Om} a_\e\big(H(\nabla \ued)\big)^2\, H(\nu)\,H^2(\nabla \ued)\,\mathrm{tr}\,\B^H\,\vphi^2\,d\H^{n-1}
        \\ \nonumber
        & \quad -2\,\int_\Om \bigg\{\mathrm{div}(\A_\e(\nabla \ued))\,\A_\e(\nabla \ued)\cdot\nabla\vphi-\nabla (\A_\e(\nabla \ued))\A_\e(\nabla \ued)\cdot\nabla \vphi \big) \bigg\}\,\vphi\,dx.
    \end{align}
From \eqref{c:um} we deduce, via  Lemma \ref{lemma:chain}, that
\begin{equation}\label{convAed}
   \A_\e(\nabla \ued)  \xrightarrow[]{m\to \infty}   \A_\e(\nabla u_\e)\quad\text{in  $W^{1,2}(\Om)$ and $C^{0,\theta'}(\overline{\Om})$.}
\end{equation}
In particular, $ \mathrm{div}\big(\A_\e(\nabla \ued)\big)  \xrightarrow[]{m\to \infty}f$ in $L^2(\Omega)$. Therefore, passing to the limit as
$m\to \infty$ in equation \eqref{main:start} yields:
\begin{equation}\label{main:start1}
    \begin{split}
        \int_\Om f^2\,\vphi^2\,dx=&\int_\Om\mathrm{tr}\big( (\nabla \A_\e(\nabla u_\e))^2 \big)\,dx+\int_{\partial\Om}  a_\e(H(\nabla u_\e))^2\, H(\nu)\,H^2(\nabla u_\e)\,\mathrm{tr}\,\B^H\,\vphi^2\,d\H^{n-1}+
        \\
        &-2\,\int_\Om \bigg\{\mathrm{div}\big(\A_\e (\nabla u_\e)\big)\,\A_\e(\nabla u_\e)\cdot\nabla\vphi-\nabla \big(\A_\e(\nabla u_\e)\big)\A_\e(\nabla u_\e)\cdot\nabla \vphi\bigg\}\,\vphi\,dx.
    \end{split}
\end{equation}
We begin by estimating the boundary integral in equation \eqref{main:start1}. Let $x\in \partial\Om$, $r\in (0, \ov{r}]$, and  $\vphi\in C^\infty_0(B_r(x))$.
Choosing $v=\A_\e^i(\nabla u_\e)\,\vphi$ in inequality \eqref{tr:ineqcap} and summing over $i=1,\dots,n$ imply that
\begin{align} \label{croazia}
   \int_{\partial\Om\cap B_r(x)} & \Big|\A_\e(\nabla u_\e)\,\vphi \Big|^2\big|\mathrm{tr}\,\B^H\big|\,d\H^{n-1}\leq c_0\,(1+L_\Om)^4\,\mathcal K_\Om(r)\,\int_{\Om\cap B_r(x)} \Big| \nabla \big(\A_\e(\nabla u_\e)\,\vphi\big)\Big|^2\,dx
   \\ \nonumber
   &\leq 2\,c_0\,(1+L_\Om)^4\,\mathcal K_\Om(r)\bigg\{\int_{\Om\cap B_r(x)} \Big|\nabla \big(\A_\e(\nabla u_\e)\big)\Big|^2\,\vphi^2\,dx+\int_{\Om\cap B_r(x)} \big|  \A_\e(\nabla u_\e)\big|^2\,|\nabla \vphi|^2\,dx\bigg\} \,.
 \end{align}
Observe that, by  equation \eqref{feb300}, $H_0(\A_\e(\xi))=a_\e(H(\xi))\,H(\xi)$ for $\xi \neq 0$. 
Also, owing to the second inequality in  \eqref{bounds:H},  $H(\nu)\leq \sqrt{\L}$. Thus,  from inequality \eqref{croazia}, we deduce, via  \eqref{bounds:H0}, that
\begin{align}\label{bdr:est}
    \int_{\partial\Om} 
 & a_\e(H(\nabla u_\e))^2\, H(\nu)\,H^2(\nabla u_\e)\,\mathrm{tr}\,\B^H\,\vphi^2\,d\H^{n-1}
    \\ \nonumber
    &=\int_{\partial\Om} H_0\big(\A_\e(\nabla u_\e) \big)^2\, H(\nu)\,\mathrm{tr}\,\B^H\,\vphi^2\,d\H^{n-1}\leq \frac{\sqrt{\L}}{\l}\int_{\partial\Om} \big|\A_\e(\nabla u_\e)\big|^2 \,\big|\mathrm{tr}\,\B^H\big|\,\vphi^2\, d\H^{n-1}
    \\ \nonumber
    &\leq \frac{2\,c_0\,(1+L_\Om)^4\,\sqrt{\L}}{\l}\,\mathcal K_\Om(r)\bigg\{\int_{\Om\cap B_r(x)} \Big|\nabla \big(\A_\e(\nabla u_\e)\big)\Big|^2\,\vphi^2\,dx+\int_{\Om\cap B_r(x)} \big|  \A_\e(\nabla u_\e)\big|^2\,|\nabla \vphi|^2\,dx\bigg\} .
    \end{align}
Next, we  use  Young's inequality  to bound   the last integral on the right-hand side of inequality \eqref{main:start1} and obtain
\begin{align}\label{last:intest}
        \bigg| 2\,\int_\Om \bigg\{\mathrm{div}\big(\A_\e(\nabla u_\e)\big)&\,\A_\e(\nabla u_\e)\cdot\nabla\vphi-\nabla \A_\e(\nabla u_\e)\A_\e(\nabla u_\e)\cdot\nabla \vphi \bigg\}\,\vphi\,dx\bigg|
        \\ \nonumber
        &\leq \gamma\int_\Om |\nabla \big(\A_\e(\nabla u_\e)\big)|^2\,\vphi^2\,dx+\frac{c_1}{\gamma}\int_\Om |\A_\e(\nabla u_\e)|^2\,|\nabla \vphi|^2\,dx,
    \end{align}
for some constant $c_1=c_1(n)$ and every $\gamma>0$. A combination of   \eqref{main:start1}, \eqref{bdr:est}, and \eqref{last:intest} enables us to deduce, via inequality \eqref{elem:ineqt}, that
\begin{align}\label{main:start2}
        \bigg( \Big(\frac{\l\,\min\{1,i_b\}}{\L\,\max\{1,s_b\}}\Big)^2&-\frac{2\,c_0\,(1+L_\Om)^4\,\sqrt{\L}}{\l}\, \mathcal K_\Om (r)-\gamma\bigg)\,\int_\Om |\nabla \big(\A_\e(\nabla u_\e)\big)|^2\,\vphi^2\,dx\leq
        \\ \nonumber
        &\leq \int_\Om f^2\,\vphi^2\,dx+\Big(\frac{2\,c_0\,(1+L_\Om)^4\,\sqrt{\L}}{\l}\, \mathcal K_\Om(r)+\frac{c}{\gamma}\Big)\int_\Om |\A_\e(\nabla u_\e)|^2\,|\nabla \vphi|^2\,dx.
    \end{align}
Now we choose
\begin{equation*}
    \gamma = \frac 12 \Big(\frac{\l\,\min\{1,i_b\}}{\L\,\max\{1,s_b\}}\Big)^2\,,
\end{equation*}
and assume that \eqref{assum:Kom} is in force with  
\begin{equation}\label{def:csegn}
    \ov{c}=\frac{1}{8}\,\bigg(\frac{\l\,\min\{1,i_b\}}{\L\,\max\{1,s_b\}} \bigg)^2\frac{\l}{\sqrt{\L}}\,\frac{1}{c_0\,\big(1+{L}\big)^4}\,,
\end{equation}
where $c_0=c_0(n,\l,\L)$ is the constant appearing in \eqref{tr:ineqcap}. With this choice of the constants, from \eqref{main:start2} we obtain that
\begin{equation}\label{main:start3}
    \begin{split}
        \int_\Om |\nabla (\A_\e(\nabla u_\e))|^2\,\vphi^2\,dx\leq c_2\,\int_\Om f^2\,\vphi^2\,dx+c_2\,\int_\Om |\A_\e(\nabla u_\e)|^2\,|\nabla \vphi|^2\,dx
    \end{split}
\end{equation}
 for some constant $c_2=c_2(n,\l,\L,i_b,s_b)$ and for every  $r\in (0,\ov{r}]$, $x\in \partial\Om$ and $\vphi\in C^\infty_0(B_r(x))$.
\\ On the other hand, inequality \eqref{main:start3} continues to hold if   $B_r(x)\subset\subset \Om$, since the boundary integral in \eqref{main:start1} simply vanishes in this case.
\\
Let us now chose a finite covering of $\Om$ by balls 
$$\text{$B_{\ov{r}/4}(x_j)$, with $x_j\in \partial \Om$, $j=1,\dots,N_B$}$$
and 
$$\text{$B_{\ov{r}/40}(z_i)$, with $z_i\in  \Om$ and $B_{\ov{r}/10}(z_i)\subset\subset \Om$, $i=1,\dots,N_I$,}$$
where $N_B\in \mathbb N$ and $N_I \in \mathbb N$.
Notice that such a covering can be chosen in such a way its cardinality $N=N_B+N_I$ admits the bound
\begin{equation}\label{Bound:N}
    N\leq c(n) \bigg(\frac{d_\Om}{\ov{r}}\bigg)^n\,.
\end{equation}
\\ Denote by $B_k$, with $k=1,\dots,N$, a generic ball from this covering, and 
let $\{\vphi_k\}_{k=1,\dots, N} $ be a family of  functions $\vphi_k\in C^\infty_0 (4 B_k)$, such that $ 0 \leq \vphi_k \leq 1$ on $4B_k$, 
\begin{equation*}
 \vphi_k = 1\quad\text{on $B_k$}\quad\text{and}\quad  |\nabla \vphi_k|\leq \frac{80}{\ov{r}}\quad\text{on $4 B_k$},
\end{equation*}
where $4 B_k$ denotes the ball having the same center as $B_k$ and whose radius is four times the  radius of $B_k$.
\\
Applying inequality \eqref{main:start3} with $\varphi = \varphi_k$, for $k=1,\dots, N$, and adding the resultant inequalities yields:
\begin{align}\label{main:start4}
         \int_\Om &\big|\nabla( \A_\e(\nabla u_\e)) \big|^2 dx \leq \sum_{k=1}^N \int_{\Om\cap B_k }\big|\nabla (\A_\e(\nabla u_\e)) \big|^2\,dx\leq \sum_{k=1}^N\int_{\Om\cap 4 B_k} \big|\nabla (\A_\e(\nabla u_\e)) \big|^2 \vphi_k^2\,dx
        \\ \nonumber
        & = \sum_{k=1}^N\int_{\Om} \big|\nabla (\A_\e(\nabla u_\e)) \big|^2 \vphi_k^2\,dx \leq \sum_{k=1}^N\,c_2\,\int_{\Om} f^2\,\vphi_k^2\,dx+\sum_{k=1}^N\,c_2\,\int_{\Om} |\A_\e(\nabla u_\e)|^2\,|\nabla \vphi_k|^2\,dx
        \\ \nonumber
        &\leq N\,c_2\,\int_\Om f^2\,dx+\frac{6400\,N\,c_2}{\ov{r}^2}\,\int_\Om |\A_\e(\nabla u_\e)|^2\,dx.
    \end{align}
Lemma \ref{poin:thm}
ensures that 
\begin{multline}\label{tee1}
    \int_\Om |\A_\e(\nabla u_\e)|^2\,dx \\ \leq \s\,\int_\Om  \big|\nabla( \A_\e(\nabla u_\e))\big|^2\,dx+c(n)\,\frac{(1+\s)^2}{\s}\,\frac{d_\Om^{2n(n+2)}}{r^{(2n+1)(n+2)}}\,(1+L_\Om)^{n+2}\,\bigg( \int_\Om |\A_\e(\nabla u_\e)|\,dx \bigg)^2
\end{multline}
for every 
$\s>0$ and $r\in (0,R_\Om)$.
  Owing to  \cite[Proposition 5.1]{cia17}, the last integral on the right-hand side of inequality \eqref{tee1} can be bounded by a constant $c=c(n,\l,\L,i_b,s_b)$ times $|\Om|^{1/n}\int_\Om |f|\,dx$. Hence, from H\"older's inequality we deduce that
%
%
\begin{equation}\label{argh}
\int_\Om |\A_\e(\nabla u_\e)|^2\,dx
        \leq \s\,\int_\Om  \big|\nabla( \A_\e(\nabla u_\e))\big|^2\,dx+c_3\,\vartheta(\s,n,r,L_\Om,d_\Om)\,\int_\Om f^2\,dx\,,
%
%
\end{equation}
for some constant $c_3=c_3(n,\l,\L,i_b,s_b)$, where  
\begin{equation}\label{vartheta}
\vartheta(\s,n, r,L_\Om,d_\Om)= \frac{(1+\s)^2}{\s}\,\frac{d_\Om^{(2n+1)(n+2)}}{r^{(2n+1)(n+2)}}\,(1+L_\Om)^{n+2} \,.
\end{equation}
Coupling  inequalities \eqref{main:start4} and \eqref{argh}, and making use of the bound from   \eqref{Bound:N} entail that
\begin{equation*}
    \begin{split}
        \int_\Om |\A_\e(\nabla u_\e)|^2\,dx\leq c_4\bigg[ \s\,\Big(\frac{d_\Om}{\ov{r}}\Big)^n+\vartheta\bigg]\,\int_\Om f^2 dx+\s\,c_4\,\frac{d_\Om^n}{\ov{r}^{n+2}}\,\int_\Om |\A_\e(\nabla u_\e)|^2\,dx\, ,
    \end{split}
\end{equation*}
for some constant $c_4=c_4(n,\l,\L,i_b,s_b)$, which can be assumed to be larger than $1$.
Now choose $\s>0$ in the above expression in such a way that
\begin{equation*}
    \s\,c_4\,\frac{d_\Om^n}{\ov{r}^{n+2}}=\frac{1}{2}\,
\end{equation*}
and observe that $\s<1$ and $\vartheta(\s,n, r,L_\Om,d_\Om)>1 $ since $r < \ov{r}\leq R_\Om \leq d_\Om$ and $c_4 >1$. Then, from definition \eqref{vartheta} to deduce that
\begin{equation}\label{main:eps}
    \int_\Om | \A_\e(\nabla u_\e)|^2\,dx
    \leq (1 + 2c_4 \vartheta)\int_\Om f^2\,dx
    \leq 32 c_4^2 \,\frac{d_\Om^{(2n+1)(n+2)+n}}{\ov{r}^{n+2}}\,\frac{(1+L_\Om)^{n+2}}{r^{(2n+1)(n+2)}}\int_\Om f^2\,dx
\end{equation}
for every $ r\in (0,\ov{r})$.
On the other hand, from inequalities \eqref{Bound:N},  \eqref{main:start4} and \eqref{main:eps} one can deduce that
\begin{equation}\label{main:eps1}
     \int_\Om \big|\nabla( \A_\e(\nabla u_\e)) \big|^2 dx\leq c_5 \,\frac{d_\Om^{(2n+1)(n+2)+2n}}{\ov{r}^{2(n+2)}}\,\frac{(1+L_\Om)^{n+2}}{r^{(2n+1)(n+2)}}\,\int_\Om f^2\,dx\,,
\end{equation}
for some constant $c_5=c_5(n,\l,\L,i_b,s_b)$ and for every $ r\in (0,\ov{r})$.  
\\
The choice $r=\frac{\ov{r}}{2}$ in \eqref{main:eps} and \eqref{main:eps1} implies that
$$\Vert \A_\e(\nabla u_\e)\Vert_{W^{1,2}(\Om)} \leq c(n,\l,\L,i_b,s_b, L, d, \ov{r}, M).$$ 
Combining the latter inequality  with \eqref{appr:3} and 
\eqref{argentina} entails that there exists a sequence $\e_k$ such that
\begin{equation*}
   \A_{\e_k}(\nabla u_{\e_k})  \rightharpoonup  \A(\nabla u)\quad\text{weakly in $W^{1,2}(\Om)$.}
\end{equation*}
Estimate \eqref{princc:est} thus follows by choosing $\e=\e_k$ and $r=\frac{\ov{r}}{2}$ in inequalities \eqref{main:eps}, \eqref{main:eps1} and passing to the limit as  $k\to \infty$.
%
%
%
%

\vspace{0.3cm}
\noindent \emph{ Step 3.}  Our task in this step is to
 remove assumption \eqref{deominf}, while maintaining \eqref{fCinf}.
To this purpose, let us extend $f$ to the whole of $\rn$ by setting  $f=0$ outside $\Om$. 
A variant of \cite[Lemma 5.2]{cia19}, established in \cite{Antonini}, ensures that there exist positive constants $\widehat{c}=\widehat{c}(n, \mathfrak L_\Om, d_\Om)$ and $\widehat{r}=\widehat{r}(n,\mathfrak L_\Om,d_\Om) <1$ and 
a sequence $\{\Om_m\}$ of open sets of $\rn$ such that: 

\noindent  $\partial\Om_m\in C^\infty$, $\Om\subset \Om_m$, $\lim_{m\to \infty}|\Om_m\setminus \Om|=0$, the Hausdorff distance between $\Om_m$ and $\Om$ tends to $0$ as $m\to \infty$,
\begin{equation}\label{om:m1}
    L_{\Om_m}\leq \widehat{c},\quad R_{\Om_m}\geq 1/{\widehat{c}},  \quad d_{\Om_m}\leq c(n)\,d_\Om ,
\end{equation}
and
    \begin{equation}\label{iscap:omm'}
        \K_{\Om_m}(r)\leq
        \begin{cases}
        \widehat{c}\Big( \K_{\Om}\big( \widehat{c}\,(r+\tfrac{1}{m})\big)+r\Big)\quad & \text{if $n\geq 3$}
        \\
        \\
        \widehat{c}\Big( \K_{\Om}\big(\widehat{c}\,(r+\tfrac{1}{m})\big)+r\,\log(1+\tfrac{1}{r})\Big)\quad & \text{if $n=2$}
        \end{cases}
    \end{equation}
for $m \in \mathbb N$ and $r\in (0,\widehat{r})$.
%
%
\\
Now let $u_m$ be the weak solution to the Dirichlet problem
\begin{equation}\label{pb:um}
\begin{cases}
    -\mathrm{div}\big(\A(\nabla u_m)\big)=f &\quad\text{in }\Om_m
    \\
    u_m=0&\quad\text{on }\partial\Om_m.
    \end{cases}
\end{equation}
Set ${L}=\widehat{c}$, ${d}=c(n)\,d_\Om$ and $ M= \|f\|_{L^2(\Om)}$ in Step 2, and assume that
condition \eqref{condK} is fulfilled with 
\begin{equation*}
    {\k}_1=\k_1(n,\l,\L,i_b,s_b, \mathfrak L_\Om, d_\Om)=\ov{c}/(2\widehat{c}),
\end{equation*}
where $\ov{c}$ is the constant defined by \eqref{defin:constants} in Step 2.
\\
This piece of information, combined with  \eqref{iscap:omm'}, implies that there exist a positive real number $\ov{r}=\ov{r}(\Om) < \min\Big\{1/{\widehat{c}}\,, \widehat{r} \Big\} < 1$ and a positive integer $\ov{m}=\ov{m}(\Om)$ such that
\begin{equation*}
    \mathcal{K}_{\Om_m}(r)\leq   \ov{c}
\end{equation*}
for  $r\in (0,\ov{r})$ and $m>\ov{m}$.
\\
Hence, we may apply the result of Step 2 to problem \eqref{pb:um}, and obtain
\begin{equation}\label{est:Am}
  \|\A(\nabla u_m)\|_{W^{1,2}(\Om)}  \leq  \|\A(\nabla u_m)\|_{W^{1,2}(\Om_m)}\leq c\,\|f\|_{L^2(\Om_m)}=c\,\|f\|_{L^2(\Om)}
\end{equation}
for some constant $c=c(i_b, s_b, \lambda, \Lambda, \Om)$. Consequently, there exist a  subsequence of $\{u_m\}$, still indexed by $m$, and a vector-valued function $U: \Om \to \rn$ such that $U\in W^{1,2}(\Om)$ and
\begin{equation}\label{feb310}
   \A(\nabla u_m)  \rightharpoonup U\quad\text{weakly in $W^{1,2}(\Om)$.}
\end{equation}
Via an analogous argument as  in the proof of inequality \eqref{gl:bde}, one infers from \cite[Theorem 2]{tal} that there exists a constant $c$, independent on $m$, such that
\begin{equation}\label{gl:bm}
\begin{split}
    \|u_m\|_{L^\infty(\Om_m)} & \leq c.
\end{split}
\end{equation}
Thereby, thanks to \cite[Theorem 1.7]{lieb91}, given any $\Om'\subset\subset\Om$, there exist $\theta\in (0,1)$ and a constant $c$ independent of $m$, such that $\|u_m\|_{C^{1,\theta}(\Om')}\leq c\,$.  Hence, there exist   a further subsequence, still denoted by $\{u_m\}$,  and a function $v\in C^{1,\theta}_{\rm loc}(\Om)$ such that
\begin{equation}\label{conv:c1m}
     u_{m}  \to v\quad\text{in $C^{1,\theta'}_{\rm loc}(\Om)$}
\end{equation}
for every $0<\theta'<\theta$. Owing to \eqref{feb310}, this implies that $A(\nabla v)=U$, whence
\begin{equation}\label{tee3}
    \A(\nabla u_m)   \rightharpoonup  \A(\nabla v) \quad \text{weakly in $ W^{1,2}(\Om)$.}
\end{equation}
On passing to the limit as $m\to \infty$ in the weak formulation of problem \eqref{pb:um}, from \eqref{tee3} we infer 
 that
\begin{equation}\label{eq:transit}
    \int_\Om \A(\nabla v)\cdot\nabla\vphi\,dx=\int_\Om f\,\vphi\,dx
\end{equation}
for  every  $\vphi\in C^\infty_0(\Om)$.
\\ Now, consider a ball $B_R$ such that $\Omega \subset \subset B_R$ and extend $u_m$ to $B_R$ by setting 
 $u_m= 0$ in  $B_R\setminus \Om_m$. Since $u_m \in W^{1,B}_0(\Omega_m)$, such an extension belongs to $W^{1,B}_0(B_R)$. By the minimality property of the function $u_m$ for the functional associated with problem \eqref{pb:um}, the fact 
that $f= 0$ in $B_R\setminus\Om$, and inequalities \eqref{B:equiv} and \eqref{gl:bm}, we have that
\begin{equation*}
 c_1 \int_{B_R}B(|\nabla u_m|)\,dx  \leq   \int_{B_R}B\big(H(\nabla u_m)\big)\,dx\leq \int_{\Om_m}f\,u_m\,dx\leq c_2 
\end{equation*}
for suitable positive constants $c_1$ and $c_2$ independent of $m$. 
  Hence, via a Poincar\'e type inequality for functions in the space $W^{1,B}_0(B_R)$, the sequence $\{u_m\}$ is bounded in  $W^{1,B}_0(B_R)$. The  reflexivity of 
 this space and the compactness of the embedding of this space into $L^1(\Omega)$ entail that there  exists a subsequence, again still denoted by $\{u_m\}$, and a function $w\in W^{1,B}_0(B_R)$ such that  
\begin{equation*}
    u_{m}  \rightharpoonup w\quad\text{weakly in $W^{1,B}_0(B_R)$}\quad{\text and}\quad   u_m \to w\quad\text{a.e. in  $B_R$.}
\end{equation*}
Since  the Hausdorff distance between $\Om_m$ and $\Om$ tends to zero, and $u_{m} =0$ in $ B_R\setminus \Om_m$, we have that $w=0$ almost everywhere in $B_R\setminus \Om$. Inasmuch as $w=v$ in $\Omega$, we can conclude that 
$v\in W^{1,B}_0(\Om)$.
\\ As in the previous steps, a density argument now ensures that \eqref{eq:transit} holds for any function $\vphi\in W^{1,B}_0(\Om)$, and hence $v$ is a  weak solution to problem   \eqref{eq:dir2}. The uniqueness of such a solution implies that 
$u=v$.  Passing to the limit as $m\to \infty$  in \eqref{est:Am}, and recalling \eqref{tee3} yield \eqref{est:A2}.

%
\vspace{0.3cm}
\noindent \emph{ Step 4.}  We conclude the proof by removing the remaining additional assumption \eqref{fCinf}. Suppose that $f\in L^2(\Om)$ and let $\{f_k\}\subset C^\infty_0({\Om})$ be any sequence such that $f_k\to f$ in $L^2(\Om)$. Let $\{u_k\}$ be the sequence of weak solutions to the Dirichlet problems
\begin{equation}
    \begin{cases}
    -\mathrm{div}\big( \A(\nabla u_k)\big)=f_k&\quad\text{in }\Om
    \\
    u_k=0&\quad\text{on }\partial\Om.
    \end{cases}
\end{equation}
Thanks to property \eqref{gen1}, one has that
\begin{equation}\label{conv:udeb}
    u_k\to u\quad\text{and }\quad \nabla u_k\to \nabla u\quad\text{a.e. in }\Om.
\end{equation}
By Step 3, there exists a constant $c=c(i_b, s_b, \lambda, \Lambda, \Om)$, such that for any $k \geq1$, 
\begin{equation}\label{est:w12k}
    \|\A(\nabla u_k)\|_{W^{1,2}(\Om)}\leq c\,\|f_k\|_{L^2(\Om)}
\end{equation}
Since $f_k\to f$ in $L^2(\Om)$, there exists a subsequence, still indexed by $k$, satisfying
\begin{equation}
    \A(\nabla u_k)\to U\quad\text{in }L^2(\Om)\quad\text{and }\quad \A(\nabla u_k)\rightharpoonup U\quad\text{weakly in }W^{1,2}(\Om),
\end{equation}
for some function $U: \Om \to \rn$ such that $U\in W^{1,2}(\Om)$. From properties \eqref{conv:udeb}, we infer that $\A(\nabla u)=U$. Hence, $\A(\nabla u)\in W^{1,2}(\Om)$ and inequality \eqref{est:A2} follows by passing to the limit as $k\to\infty $ in  estimate \eqref{est:w12k}.
\end{proof}

\begin{proof}[Proof of Theorem \ref{an:mainthm1}, Dirichlet problems] The proof proceeds through the same steps as that of Theorem \ref{an:mainthm2}. We limit ourselves to sketching the necessary changes.

\vspace{0.3cm}
\noindent \emph{Step 1} is unchanged.

\vspace{0.3cm}
\noindent \emph{Step 2.}
One has to replace condition \eqref{assum:Kom} with  
\begin{equation}\label{assum:kpsi}
    \Psi_\Om(r)\leq \ov{c}_1\quad \text{for $r\in(0,\ov{r}]$,}
\end{equation}
where the constant $\ov{c}_1$ is given by
\begin{equation*}
    \ov{c}_1 =\frac{1}{4}\bigg(\frac{\l\,\min\{1,i_b\}}{\L\,\max\{1,s_b\}} \bigg)^2\,\frac{\l}{\sqrt{\L}}\,\frac{1}{c_0\,\big(1+{L}\big)^{11}}.
\end{equation*}
One then makes use of Part (ii)  of Lemma \ref{traceineq}, instead of  Part (i), in order to estimate the boundary term  in \eqref{croazia}.  Inequality  \eqref{princc:est} hence follows.

\vspace{0.3cm}
\noindent \emph{Step 3.} 
 Coupling 
 inequality \eqref{isocap:isomar} with \eqref{iscap:omm'} tells us that there exist constants  $\widehat{c}=\widehat{c}(n, \mathfrak L_\Om, d_\Om)$ and $\widehat{r}=\widehat{r}(n, \mathfrak L_\Om, d_\Om)$ such that
\begin{equation}\label{om:m3}
  \K_{\Om_m}(r)
\leq
        \begin{cases}
        \widehat{c}\,\Big( \Psi_{\Om}\big( \widehat{c}\,(r+\tfrac{1}{m})\big)+r\Big)\quad & \text{if $n\geq 3$}
        \\
        \\
        \widehat{c}\,\Big( \Psi_{\Om}\big(\widehat{c}\,(r+\tfrac{1}{m})\big)+r\,\log(1+\tfrac{1}{r})\Big)\quad & \text{if $n=2$}
        \end{cases}
\end{equation}
for $r\in (0, \widehat r)$.
%
Assume that condition \eqref{bhc:1} is in force with constant
\begin{equation*}
    \k_0=\k_0(n,\l,\L,i_b,s_b,\mathfrak L_\Om,d_\Om)=\ov{c}/{(2\widehat{c})},
\end{equation*}
where $\ov{c}$ is  defined in \eqref{def:csegn}.
From \eqref{om:m3} we infer that there exist constants $\ov{r}=\ov{r}(\Om)$ and $\ov{m}=\ov{m}(\Om)$ such that
\begin{equation*}
    \K_{\Om_m}(r)\leq \ov{c}
\end{equation*}
for  $r\in \big(0,\ov{r}(\Om)\big)$ and $m>\ov{m}(\Om)$.
Therefore, starting from estimate  \eqref{est:Am}, one can 
 now conclude as in the proof Step 3 of Theorem \ref{an:mainthm2}.

\vspace{0.3cm}
\noindent \emph{Step 4} is unchanged.
\end{proof}

\begin{proof}[Proof of Theorem \ref{an:mainthm},  Dirichlet problems]
The proof parallels that of Theorems \ref{an:mainthm2} and \ref{an:mainthm1}. It is indeed simpler, since the boundary terms in the a priori  estimates can just be disregarded, thanks to their sign.  In what follows, we  just point out the necessary variants and simplifications.
\vspace{0.3cm}

\noindent \emph{ Step 1.}  This step agrees with that   of Theorem \ref{an:mainthm2}.

\vspace{0.3cm}
\noindent \emph{ Step 2.}  The convexity of the set $\Omega$ plays a major role in this step.  Owing to property \eqref{feb315}, it ensures that $\mathrm{tr}\,\B^H\geq 0$ on $\partial \Omega$. Therefore,
an application of equation  \eqref{reilly:formula}, with $v=\ued$, $h=a_\e\big(H(\nabla \ued(x))\big)$, and $\phi= 1$ tells us that
\begin{equation}\label{feb316}
    \int_\Om \mathrm{div}\big(\A_\e(\nabla \ued)\big)^2\,dx\geq\int_\Om\mathrm{tr}\big( \nabla \big(\A_\e(\nabla \ued)\big)^2\big)\,dx.
\end{equation}
Thanks to \eqref{convAed}, passing to the limit in inequality \eqref{feb316} as $m\to \infty$  and using inequality \eqref{elem:ineqt} yield:
\begin{equation}\label{repl:1}
    \int_\Om f^2\,dx\geq \int_\Om\mathrm{tr}\big( (\nabla \A_\e(\nabla u_\e))^2\big)\,dx\geq \bigg(\frac{\l\,\min\{1,i_b\}}{\L\,\max\{1,s_b\}}\bigg)^2\int_\Om|\nabla (\A_\e(\nabla u_\e))|^2\,dx.
\end{equation}
In order to estimate the $L^2$-norm of $\A_\e(\nabla u_\e)$, we exploit the fact that, since $\Omega$ is a  bounded convex domain in $\rn$, 
the constant in the Poincar\'e inequality on  $\Omega$ depends only on $d_\Om$ and $n$. Thus, on denoting by $ \A_\e(\nabla u_\e)_\Om$ the vector-valued mean value of  $\A_\e(\nabla u_\e)$ over $\Om$, we have that
\begin{align}\label{may120}
         \int_\Om |\A_\e(\nabla u_\e)|^2\,dx & \leq 2\int_\Om |\A_\e(\nabla u_\e)- \A_\e(\nabla u_\e)_\Om |^2\,dx+2\,|\Om|\,|\A_\e(\nabla u_\e)_\Om|^2
         \\ \nonumber
        &\leq 2\,c\,d_\Om^2 \int_\Om |\nabla \A_\e(\nabla u_\e)|^2 dx+2\,|\Om|^{-1}\bigg(\int_\Om |\A_\e(\nabla u_\e)|\,dx \bigg)^2\
    \end{align}
for some constant $c=c(n)$. The following chain holds:
\begin{align}\label{repl:2}
        \int_\Om |\A_\e(\nabla u_\e)|^2\,dx & \leq 2\,c(n)\,d_\Om^2 \int_\Om |\nabla \A_\e(\nabla u_\e)|^2 dx+2\,|\Om|^{-1}c(n,\l,\L,i_b,s_b)\,\Big(|\Om|^{1/n}\,\int_\Om |f|\,dx\Big)^2
        \\ \nonumber
       & \leq 2\,c(n)\,d_\Om^2 \int_\Om |\nabla \A_\e(\nabla u_\e)|^2 dx+c(n,\l,\L,i_b,s_b)\,d_\Om^2\,\int_\Om f^2\,dx
       \\ \nonumber
       &\leq c'(n,\l,\L,i_b,s_b)\,d_\Om^2\,\int_\Om f^2\,dx\,,
    \end{align}
where the first inequality is a consequence of inequality \eqref{may120} and of inequality  \eqref{estgraddir}, with $\A$ and $u$ replaced by $\A_\e$ and $u_\e$, the second inequality follows via 
 H\"older's inequality and the fact that $|\Om|\leq c(n)\,d_\Om^n$, and the last one is due to \eqref{repl:1}.
\\
Starting from inequalities \eqref{repl:1} and \eqref{repl:2}, instead of \eqref{main:start3}, estimate \eqref{est:A} follows via an analogous argument as in the proof of Theorem \ref{an:mainthm2} .
\vspace{0.3cm}

\noindent \emph{ Step 3.} The proof is the same as that of Theorem \ref{an:mainthm2}, save that the approximating domains $\Om_m$ have to be taken convex, and the bounds in \eqref{est:A}, with $\Om$ replaced by $\Om_m$, have to be used.
In order to construct the  convex approximating domains $\Om_m$, one can employ the regularized   signed distance $\rho$ of \cite[Theorem 1.4]{lieb:dist}. Since the latter is a concave function, which is smooth outside $\partial \Om$,   the open sets
\begin{equation}\label{may121}
    \Om_m = \{ x\in \rn\,:\,-\rho(x)<1/m\}
\end{equation}
satisfy the desired properties.
\vspace{0.3cm}

\noindent \emph{ Step 4.} This step is the same as that of Theorem \ref{an:mainthm2}.
\end{proof}

\begin{proof}[Proof of Theorem \ref{thm:C11}, Dirichlet problems] 

Recall that, from \eqref{isocap:isomar-pippo}, we have 
    \begin{equation}\label{kom:inff}
     \K_{\Om}(r)\leq
        \begin{cases}
            c(n)\,(1+L_\Om)^{5}\,r\,\|\B\|_{L^\infty(\partial \Om)}\quad & \text{if $n\geq 3$}
            \\
            \\
            c\,(1+L_\Om)^{8}\,\omega(r)\big(1+\|\B\|_{L^\infty(\partial \Om)}\big)\quad & \text{if $n=2$,}
        \end{cases}      
    \end{equation}
for $r\in (0,R_\Om)$,
where $\omega : (0, \infty) \to [0, \infty)$ denotes the function given by
$$\omega(r) = r\,\log\Big(1+\frac{1}{r}\Big) \quad \text{for $r\in (0,\infty)$.}$$
Now we apply the result of Step 2 of Theorem \ref{an:mainthm2} with 
    \begin{equation*}
        L = L_\Om \,,\quad d=d_\Om \,,\quad M= \|f\|_{L^2(\Om)} 
    \end{equation*}
and a suitable $\ov{r}=\ov{r}(n,i_b,s_b,\l,\L,{L},R_\Om, \|\B\|_{L^\infty(\partial \Om)})$ such that inequality \eqref{assum:Kom} is satisfied. Hence, inequalities \eqref{princc:est} will hold with constants $c_1,c_2$   now depending only on $n,i_b,s_b,\l,\L, {L},R_\Om, {d}, \|\B\|_{L^\infty(\partial \Om)}$.
\\
    Specifically, when $n\geq 3$,  then inequality \eqref{assum:Kom} is fulfilled provided that
    \begin{equation*}
        \ov{r}\leq 
\min\Big\{\frac{c}{(1+{L})^{9}\,\|\B\|_{L^\infty(\partial \Om)}}\,,R_\Om\Big\}
    \end{equation*}
for a suitable constant $c=c(n,\l,\L,i_b,s_b)$. Thus, the inequalities in \eqref{princc:est}
  follow, with
\begin{align*}
        c_1 &  =c\, {d}^{p(n)}\,(1+{L})^{(n+2)}\,\max\Big\{(1+L)^{t(n)}\,\|\B\|_{L^\infty(\partial \Om)}^{(2n+2)(n+2)},\,R_\Om^{-(2n+2)(n+2)}\Big\}
        \\
        c_2 & =c \,{d}^{p(n)+n}\,(1+{L})^{(n+2)}\,\max\Big\{(1+L)^{t(n)+9(n+2)}\,\|\B\|_{L^\infty(\partial \Om)}^{(2n+3)(n+2)},\,R_\Om^{-(2n+3)(n+2)}\Big\}\,,
        \end{align*}
    where $p(n)=(2n+1)(n+2)+n$,  $t(n)=9\,(n+2)(2n+2)$ and $c=c(n,i_b,s_b,\l,\L)$.
\\ 
When $n=2$, observe that the function $\omega$ is increasing and, for every $s_0 \in (0,1)$, there exist constants $c_1$ and $c_2$ such that
\begin{equation}\label{paseky3}
c_1\, \frac{s}{\log (1+\frac 1s)} \leq \omega ^{-1}(s) \leq c_2 \, \frac{s}{\log (1+\frac 1s)} \quad \text{for $s \in (0,s_0)$.}
\end{equation}
Thereby, inequality  \eqref{assum:Kom}  holds if 
 \begin{equation*}
        \ov{r}
\leq \min\bigg\{ \frac{c\, \log \big(1+ c(1+L)(1+\|\B\|_{L^\infty(\partial \Om)})\big)}{(1+{L})^{12}\, (1+\|\B\|_{L^\infty(\partial \Om)})}\,, R_\Om\bigg\}
    \end{equation*}
for a suitable constant $c=c(n,\l,\L,i_b,s_b)$. 

As a consequence, the inequalities in \eqref{princc:est} are fulfilled with 
\begin{align*}
        c_1 &  =c'\,{d}^{22}\,(1+L)^4\,\max\bigg\{ \frac{(1+{L})^{288}\,(1+\|\B\|_{L^\infty(\partial \Om)})^{24}}{\log ^{24}\big(1+ c(1+L)(1+\|\B\|_{L^\infty(\partial \Om)})\big)}\,, R_\Om^{-24}\bigg\}
        \\
        c_2 & =c'\,{d}^{24} (1+L)^4\,\max\bigg\{\frac{(1+{L})^{336}\,(1+\|\B\|_{L^\infty(\partial \Om)})^{28}}{\log ^{28}\big(1+ c(1+L)(1+\|\B\|_{L^\infty(\partial \Om)})\big)}\,, R_\Om^{-28}\bigg\}
        \end{align*}
for some constant $c=c(n,i_b,s_b,\l,\L)$ and  $c'=c'(n,i_b,s_b,\l,\L)$.
\end{proof}

\section{Global estimates: Neumann problems}\label{sec:neumann}

We conclude with proofs of our global regularity results to Neumann problems of the form \eqref{eq:neu2}.
The definition of generalized solutions to these problems can be given in a spirit analogous to that presented for Dirichlet problems in Section  \ref{sec:dir}. 
Assume that
 $f \in L^2(\Omega)$ and 
\begin{equation}\label{mean}
\int_\Omega f\, dx =0.
\end{equation} 
A function \textcolor{black}{$u \in \mathcal T ^{1,1}(\Omega)$} will be called a  generalized solution to problem \eqref{eq:neu2} if $\A(\nabla u)\in L^1(\Omega)$, 
\begin{equation}\label{232}
\int_\Omega \A(\nabla u) \cdot \nabla \varphi \, dx=\int_\Omega f \varphi \,dx\,
\end{equation}
for every   $\varphi \in C^\infty (\Omega) \cap W^{1,\infty}(\Omega)$, and
there exists a sequence $\{f_k\}\subset  C^\infty _0(\Omega)$,
with  $\int _\Omega f_k(x)\, dx =0$ for $k \in \N$, such that
$f_k \to f$   in $L^2(\Omega)$ and the sequence of  (suitably normalized by additive constants) weak solutions  $\{u_k\}$   to the problem \eqref{eq:neu2}, with $f$ replaced by $f_k$, satisfies 
%
%
$$u_k\to u \quad
\hbox{ a.e. in $\Omega$.}$$
%
Owing to \cite[Theorem 3.8]{cia17}, if $\Omega$ is a bounded Lipschitz domain, then
there exists a unique (up to additive constants) generalized solution $u$ to problem \eqref{eq:neu2},  and
\begin{equation}\label{estgradneu}
\int_\Omega |\A(\nabla u)|\, dx \leq c\,|\Om|^{1/n}\,\int_\Omega |f|\, dx
\end{equation} 
for some constant  $c=c(n,\l,\L,i_b, s_b)$.    Moreover, if
$\{f_k\}$ is any sequence as above, and $\{u_k\}$ is the associated sequence of (normalized) weak solutions, then 
\begin{equation}\label{gen1neu}
u_k \to u \quad \hbox{and} \quad \nabla u_k \to \nabla u \quad \hbox{a.e. in $\Omega$,}
\end{equation}
up to subsequences.

Recall that a  function $u\in W^{1,B}(\Omega)$ is called a weak solution to the Neumann problem \eqref{eq:neu2} if  equation \eqref{232}  holds for every function $\vphi \in W^{1,B}(\Omega)$. 
If $\Omega$ is a bounded Lipschitz domain  and $f \in L^\infty(\Omega)$ and fuflills condition \eqref{mean}, then  one can conclude as in \cite[Theorems 2.13 and 2.14]{cia11} that there exists a unique (up to additive constants) weak solution $u$ to the Neumann problem \eqref{eq:neu2}. Moreover, 
\begin{equation}\label{energy'}
\int_\Omega B(|\nabla u|)\, dx \leq c \|f\|_{L^\infty(\Omega)}b^{-1}( \|f\|_{L^\infty(\Omega)})
\end{equation}
for some constant $c=c(|\Omega|, i_b, s_b, \lambda, \Lambda)$. 

 \begin{proof}[Proof of Theorem \ref{an:mainthm2}, Neumann problems]
We split the proof into steps. 
\\ \emph{Step 1.} We begin by imposing  the additional assumptions that
\begin{equation}\label{fCinf'}
   f\in  C^\infty_0(\Om)
\end{equation}
and
\begin{equation}\label{deominf'}
    \partial\Om\in C^{2}.
\end{equation}
For every $\e\in (0,1)$,  let $\A_\e$ be the function defined by  in \eqref{def:Ae'}. Let $u_\e$  the (unique up to additive constants) weak solution to the Neumann problem
\begin{equation}\label{eq:nue}
    \begin{cases}
    -\mathrm{div}\big(\A_\e(\nabla u_\e)\big)=f\quad\text{in }\Om
    \\
    \A_\e(\nabla u_\e)\cdot\nu=0\quad\text{on }\partial\Om .
    \end{cases}
\end{equation}
We claim that 
\begin{equation}\label{marocco:neum}
        u_\e\in W^{2,2}(\Om)\cap C^{1,\theta}(\overline{\Om})\,,
    \end{equation}
    and  the solutions $u_\e$ can be defined with suitable additive constants in such a way that there exists a sequence $\{\e_k\}$ such that $\e_k \to 0^+$ and 
  \begin{equation}\label{argentina:neum}
        u_{\e_k} \to u\quad\text{in $C^{1,\theta'}_{\rm loc}(\Om)$}\,.
    \end{equation}
%
To this purpose, for $\delta >0$ consider  the  (unique  up to additive constants) solution $w_{\e,\d}\in W^{1,2}(\Om)$ to the problem
 \begin{equation}\label{eq:nd}
           \begin{cases}-\mathrm{div}\big(\A_{\e,\d}(\nabla w_{\e,\d})\big)=f\quad\text{in $\Om$}
           \\
           \A_{\e,\d}(\nabla w_{\e,\d})\cdot\nu=0\quad\text{on $\partial\Om$}\,,
       \end{cases}      
\end{equation}
where the function $\A_{\e,\d}$ is  defined as in \eqref{aed}.
\\
An application of \cite[Theorem 3.1 (a)]{cia90} ensures that the solutions $u_\e$ and $w_{\e,\d}$ can be chosen with proper additive constants so that
        \begin{equation}\label{bd:nn}
               \|u_\e\|_{L^\infty(\Om)}\leq c \quad \text{and} \quad  \|w_{\e,\d}\|_{L^\infty(\Om)}\leq c
        \end{equation}
for some contant $c$ independent of $\d$ and $\e$. 
\\ Thanks to the latter inequality, 
an application of  \cite[Theorem 2]{lieb88} tells us that there exists a constant $c$, independent of $\d$,   such that
        \begin{equation}\label{march380}
            \|w_{\e,\d}\|_{C^{1,\theta}(\overline{\Om})}\leq c\,.
        \end{equation}
On the other hand, via an analogous argument as in the proof of \cite[Theorem 8.2]{ben:fre}, adapted  to (homogeneous) Neumann boundary condition, one can show that
 \begin{equation}\label{march381}
            \|w_{\e,\d}\|_{W^{2,2}(\Om)}\leq c
\end{equation}
for some constant $c$ independent of $\d$. Bounds \eqref{march380} and  \eqref{march381}
 imply that there exist a sequence $\{\d_k\}$ and a  function $w_\e\in C^{1,\theta'}(\overline{\Om})\cap W^{2,2}(\Om)$ such that $\d_k \to 0^+$,
\begin{equation*}
    w_{\e,\d_k}\to w_\e\quad\text{in $C^{1,\theta'}(\overline{\Om})$, \,and }\quad  w_{\e,\d_k}\rightharpoonup w_\e\quad\text{weakly in $W^{2,2}(\Om)$}.
\end{equation*}
Passing to the limit in the weak formulation of problem \eqref{eq:nd} as $k\to \infty$ shows that $w_\e$ is a solution to problem \eqref{eq:nue}.  Hence,  $w_\e = u_\e$, up to additive constants.  Property \eqref{marocco:neum} is thus established.
\\ Next,   the first inequality in \eqref{bd:nn} enables one to 
apply \cite[Theorem 7]{lieb91} and infer that, for every open set  $\Om'\subset\subset \Om$, there exists a constant $c$ independent of $\e$ such that
\begin{equation*}
    \|u_\e\|_{C^{1,\theta}(\Om')}\leq c.
\end{equation*}
Hence,  there exist a function  $v\in C_{\rm loc}^{1,\theta'}(\Om) $ and a  sequence $\{\e_k\}$ such that $\e_k\to 0$ and
\begin{equation*}
    u_{\e_k}\to v\quad\text{in $C^{1,\theta'}(\Om')$}\,,
\end{equation*}
for every $\theta'\in (0, \theta)$ and every open set $\Om'\subset\subset \Om$.
\\ Via a similar argument as in the proof of equation \eqref{march386} one can show that $v\in W^{1,B}(\Om)$. Hence, passing to the limit as $k \to \infty$ in the weak formulation of problem \eqref{eq:nue} with $\e=\e_k$, implies that $v$ is a solution to the  Neumann problem \eqref{eq:neu2}. Thus,
$v=u+c$ for some constant $c$, and \eqref{argentina:neum} follows.
%
%


\smallskip
\par\noindent  \emph{Step 2.}  Assume that hypotheses \eqref{fCinf'} and \eqref{deominf'} are still satisfied.
The following identity holds for $\e \in (0,1)$, and  is a consequence of  \cite[Theorem 3.1.1.1]{gris}:
\begin{align}\label{gris2}
        \int_\Om f^2\,\phi\,dx&=\int_\Omega\mathrm{tr}\big((\nabla\,(\A_\e(\nabla u_\e)))^2\big)\phi\,dx+\int_{\partial\Om}\B\Big(\A_\e(\nabla u_\e)_T,\,\A_\e(\nabla u_\e)_T\Big)\,\phi\,d\H^{n-1}
        \\ \nonumber
        & \quad -\int_\Om\bigg\{ \mathrm{div}\,(\A_\e(\nabla u_\e))\,\A_\e(\nabla u_\e)\cdot\nabla\phi-\nabla (\A_\e(\nabla u_\e))\A_\e(\nabla u_\e)\cdot\nabla\phi\bigg\}\,dx.
    \end{align}
    Let us incidentally note the latter identity could also be deduced from Lemma \ref{quick:form-new}, via an approximation argument.
This identity plays a role in the Neumann problem parallel to that of 
 \eqref{main:start1} in the Dirichlet problem.
Since $\A_\e(\xi)\in C^{0,1}(\rn) \cap C^1(\rn\setminus\{0\})$, and $u_\e\in W^{2,2}(\Om)$, by the chain rule for vector-valued functions \cite{mm}, we have that
\begin{equation*}
    \A_\e(\nabla u_\e)\in W^{1,2}(\Om).
\end{equation*}
Now, 
let $x\in \partial\Om$ and $r\in (0, R_\Om)$, and choose  $\phi=\vphi^2$, with $\varphi \in C^\infty_0(B_r(x))$ in identity \eqref{gris2}. From inequality \eqref{tr:ineqcap}, appplied with $v=\A_\e^i(\nabla u_\e)\,\vphi$, and Young's inequality one deduces that
\begin{align}\label{bdr:N}
    \Big| & \int_{\partial\Om}\B\Big(\A_\e(\nabla u_\e)_T,\,\A_\e(\nabla u_\e)_T\Big)\,\vphi^2\,d\H^{n-1} \Big|\leq \int_{\partial\Om} \big|\B\big|\,|\A_\e(\nabla u_\e)_T|^2\,\vphi^2\,d\H^{n-1}
    \\ \nonumber
    &\leq c_0\,(1+L_\Om)^4\,\mathcal K_\Omega(r)\int_{\Om} \big|\nabla(\A_\e(\nabla u_\e)\,\vphi) \big|^2\,dx
    \\ \nonumber
    &\leq 2\,c_0\,(1+L_\Om)^4\,\mathcal K_\Omega(r)\bigg\{\int_{\Om\cap B_r(x)} \Big|\nabla \A_\e(\nabla u_\e)\Big|^2\,\vphi^2\,dx+\int_{\Om\cap B_r(x)} \big|  \A_\e(\nabla u_\e)\big|^2\,|\nabla \vphi|^2\,dx\bigg\}.
\end{align}
Thus, from \eqref{elem:ineqt}, \eqref{last:intest}, \eqref{gris2} and \eqref{bdr:N} 
 we obtain inequality \eqref{main:start2} again. Starting from that inequality, one can proceed exactly as in the proof for Dirichlet problems and derive inequality \eqref{est:A2} under the current assumptions 
 \eqref{fCinf'} and \eqref{deominf'}. Just notice that properties \eqref{marocco:neum} and \eqref{argentina:neum} have to used in this derivation instead   of \eqref{marocco} and \eqref{argentina}.


\smallskip
\par\noindent  \emph{Step 3.} Here we still assume  \eqref{fCinf'}, but   remove the restriction \eqref{deominf'}.
\\
To this purpose, extend the function $f$ to $\rn$ by  $0$ in $\rn \setminus \Om$, and consider a sequence of sets $\{\Omega_m\}$ as in Step 3 of the proof for Dirichlet problems. For each $m \in \mathbb N$, let 
 $u_m$ be the unique (up to additive constants) solution to the problem
\begin{equation}\label{eq:neummm}
    \begin{cases}
        -\mathrm{div}\big(\A(\nabla u_m)\big)=f\quad\text{in $\Om_m$}
        \\
        \A(\nabla u_m)\cdot\nu=0\quad\text{on $\partial \Om_m$}.
    \end{cases}
\end{equation}
By \cite[Theorem 3.1 (a)]{cia90} and \cite[Theorem 1.7]{lieb91}, there exists a sequence of functions $\{u_m\}$, suitably normalized by additive constants and still indexed by $m$, such that 
\begin{equation*}
    u_m \to v\quad \text{in  $C^1_{\rm loc}(\Om)$}\,,
\end{equation*}
for some function $v\in C^1(\Om)$. 
\\ An analogous argument as in the proof of Step 3 for Dirichlet problems, which relies upon 
 \cite[Theorem 2.14]{cia11} and \cite[Theorem 3.1 (b)]{cia90},  enables one to infer that $v\in W^{1,B}(\Om)$. 
\\ In order to show that $v$ agrees with $u$, up to an additive constant, it suffices to prove that
 $v$ solves the Neumann problem \eqref{eq:neu2}.
Thanks to properties \eqref{om:m1} and \eqref{iscap:omm'},
by Step 2 the sequence $\{\A(\nabla u_m)\}$ is uniformly bounded in $W^{1,2}(\Om)$,
inasmuch as 
\begin{equation}\label{est:Am'}
  \|\A(\nabla u_m)\|_{W^{1,2}(\Om)}  \leq  \|\A(\nabla u_m)\|_{W^{1,2}(\Om_m)}\leq c\,\|f\|_{L^2(\Om_m)}=c\,\|f\|_{L^2(\Om)}
\end{equation}
for some constant $c$  independent of $m$.
\\
 Hence, there exists a subsequence of $\{u_m\}$,  still indexed by $m$, such that
\begin{equation}\label{h1:neumm}
    A(\nabla u_m)\rightharpoonup \A(\nabla v)\quad\text{weakly in $W^{1,2}(\Om)$}\,.
\end{equation}
Let us extend any test function $\vphi \in W^{1,\infty}(\Om)$ to a function in $W^{1,\infty}(\rn)$, and still denote this extension by $\varphi$. The definition of weak solution to problem \eqref{eq:neummm} implies that
\begin{equation}\label{march395}
    \int_{\Om}f\,\vphi\,dx  =\int_{\Om_m} \A(\nabla u_m)\cdot\nabla \vphi\,dx=
   \int_{\Om} \A(\nabla u_m)\cdot\nabla \vphi\,dx+\int_{\Om_m\setminus\Om} \A(\nabla u_m)\cdot\nabla \vphi\,dx
\end{equation}
for $m \in \mathbb N$.
Property  \eqref{h1:neumm} ensures that
\begin{equation}\label{march396}
    \lim_{m\to\infty}\int_{\Om} \A(\nabla u_m)\cdot\nabla \vphi\,dx=\int_\Om \A(\nabla v)\cdot\nabla \vphi\,dx.
\end{equation}
On the other hand, the fact that $|\Omega_m \setminus \Omega| \to 0$ and the dominated convergence theorem yield
\begin{equation}\label{march397}
\lim_{m\to\infty}\int_{\Om_m\setminus\Om} \A(\nabla u_m)\cdot\nabla \vphi\,dx=0.
\end{equation}
Combining equations \eqref{march395}--\eqref{march397} tell us that the function  $v$ satifies the equality
\begin{equation}\label{march398}
   \int_{\Om} \A(\nabla v)\cdot\nabla \vphi\,dx=  \int_{\Om}f\,\vphi\,dx
\end{equation}
for every  $\vphi\in W^{1,B}(\Om)$. Thus, $v$ is a weak solution to the Neumann problem  \eqref{eq:neu2}, whence $v=u$, up to additive constants. Inequality  \eqref{est:A2} follows via 
\eqref{est:Am'}.
\vspace{0.3cm}

\smallskip
\par\noindent  \emph{Step 4.} The remaining additional assumption  \eqref{fCinf'} can be removed by approximating $f$ by a sequence of smooth functions 
 $\{f_k\}$, via the same argument as in Step 4 of the proof for Dirichlet problems. One has just to choose the sequence in such a way the compatibility condition  $\int _\Omega f_k(x)\, dx =0$ is fulfilled for $k \in \mathbb N$. 
\end{proof}

\begin{proof}[Proof of Theorem \ref{an:mainthm1}, Neumann problems] The proof is the same as that 
of Theorem \ref{an:mainthm2}, save that Part (i)  has to be replaced with  Part (ii) in the application of Lemma \ref{traceineq} in Step 2, and equation \eqref{om:m3} has to be used in Step 3 as in the proof of the corresponding Dirichlet problem.
\end{proof}

\begin{proof}[Proof of Theorem \ref{an:mainthm}, Neumann problems] 
The only variants with respect to the proof of Theorem  \ref{an:mainthm2} concern Steps 2 and 3. 
\par \noindent \emph{ Step 2.}  The convexity assumption on $\Omega$ ensures that the quadratic form $\B$ is nonnegative on $\partial \Omega$. Thus, 
choosing $\phi = 1$ in equation \eqref{gris2} and exploiting inequality \eqref{elem:ineqt} yield
\begin{equation}
    \int_\Om f^2\,dx\geq \int_\Omega\mathrm{tr}\big((\nabla\,\A_\e(\nabla u_\e))^2\big)\,dx\geq \bigg(\frac{\l\,\min\{1,i_b\}}{\L\,\max\{1,s_b\}}\bigg)^2\,\int_\Om \big|\nabla \A_\e(\nabla u_\e) \big|^2\,dx.
\end{equation}
This inequality replaces (and simplifies) the use of inequality \eqref{main:start2} in the derivation of \eqref{est:A} in the case when $f\in C^\infty_0(\Omega)$ and $\partial \Omega \in C^2$.
\par \noindent \emph{ Step 3.} The sole variant here is in that the approximating smooth domains $\Omega_m$ have to be chosen convex, as defined as in equation \eqref{may121}, for instance.
\end{proof}

\begin{proof}[Proof of Theorem \ref{thm:C11}, Neumann problems]
Thanks to Corollary \ref{cor:infty},
the conclusions can be deduced via a 
slight variant of the proof of Theorem \ref{an:mainthm2} for Neumann problems. The necessary modifications parallel those mentioned in the proof of the present theorem for Dirichlet problems. The details are omitted for brevity.
\end{proof}

\bigskip
\par\noindent
{\bf Acknowledgement}.  We wish to thank  S.Cingolani and M.De Giovanni for pointing out their paper \cite{cin:deg}.

\section*{Compliance with Ethical Standards}\label{conflicts}


\smallskip
\par\noindent
{\bf Funding}. This research was partly funded by:
\\ (i) GNAMPA   of the Italian INdAM - National Institute of High Mathematics (grant number not available)  (C.A.Antonini, A.Cianchi, G.Ciraolo);
\\ (ii) Research Project   of the Italian Ministry of Education, University and
Research (MIUR) Prin 2017 ``Direct and inverse problems for partial differential equations: theoretical aspects and applications'',
grant number 201758MTR2 (A.Cianchi).

\bigskip
\par\noindent
{\bf Conflict of Interest}. The authors declare that they have no conflict of interest.

\end{document}